\documentclass[11pt, leqno]{amsart}
\usepackage{amsmath,amssymb,amsthm,mathrsfs}
\usepackage{enumerate}
\usepackage{bbm,bm,xcolor,tabularx,array,makecell}

\usepackage{enumitem}
\setlist[itemize]{label=$\diamond$}

\usepackage{tabularx, booktabs, makecell, caption}
\usepackage{xcolor}
\usepackage{hyperref}
\hypersetup{
	colorlinks=true,
	linkcolor=blue,    % internal links (sections, equations)
	citecolor=green,   % citations
	urlcolor=red       % URLs
}

\usepackage{cite}
\usepackage{comment}
\usepackage{enumitem}
\usepackage{tikz}
\usetikzlibrary{arrows.meta, positioning}
\numberwithin{equation}{section}
\newcommand{\bb}{\mathbb}
\newcommand{\cc}{\mathcal}

\newcommand{\tr}{\mathrm{tr}}

\newcommand{\ov}{\overline}

	\newcommand{\bbm}{\mathbbm}
	\newcommand{\Span}{\operatorname{Span}}
	\newcommand{\Range}{\operatorname{Range}}

	\newcommand{\e}{\mathcal{E}}
	\newcommand{\exi}{\e_{\i\xi}}
	
	\newcommand{\wt}{\widetilde}
	
	\newcommand{\rank}{\text{rank}}
	\renewcommand{\ge}{\geqslant}
	\renewcommand{\le}{\leqslant}

	\renewcommand{\bm}{}
	\renewcommand{\L}{\mathrm{L}}
	\newcommand{\R}{\mathbb{R}}
	\newcommand{\Dom}{\mathcal{D}}
	\renewcommand{\i}{\mathrm{i}}
	\renewcommand{\Re}{\mathrm{Re}}

	\renewcommand{\epsilon}{\varepsilon}
	
	\newcommand{\changelocaltocdepth}[1]{%
		\addtocontents{toc}{\protect\setcounter{tocdepth}{#1}}%
		\setcounter{tocdepth}{#1}%
	}

	\numberwithin{equation}{section}
	\newtheorem{theorem}{Theorem}[section]
	\newtheorem{proposition}[theorem]{Proposition}
	\newtheorem{lemma}[theorem]{Lemma}
	\newtheorem{corollary}[theorem]{Corollary}
	\theoremstyle{definition}
	\newtheorem{assumption}{H\!\!}
	\newtheorem{definition}[theorem]{Definition}
	\newtheorem{example}[theorem]{Example}
	\theoremstyle{remark}
	\newtheorem{remark}[theorem]{Remark}
	\newtheorem{notation}[theorem]{Notation}
	
	\makeatletter
	\@namedef{subjclassname@2020}{%
		\textup{2020} Mathematics Subject Classification}
	\makeatother
	
	\title[Spectral theory for non-local OU operators]{Spectral theory for non-local Ornstein-Uhlenbeck operators}
	\author[Rohan Sarkar]{Rohan Sarkar}
	\address{Department of Mathematics\\
		University of Virginia\\
		Charlottesville, VA 22904,  U.S.A.}
	\email{kfg3et@virginia.edu}

	
\begin{document}
		\begin{abstract} We consider non-local Ornstein-Uhlenbeck (OU) operators generating ergodic OU processes driven by L\'evy processes. These operators are generally non-normal in the $L^2$ space weighted by the invariant distribution. We prove that their point spectrum and eigenvalue multiplicities are independent of the driving L\'evy process and are determined by the drift matrix. When the drift matrix is diagonalizable, we derive explicit eigenfunctions and co-eigenfunctions forming a biorthogonal system with respect to the invariant distribution. In the one-dimensional case, the eigefunctions form an \emph{Appell sequence}, which have appeared in the study of optimal stopping problems related to L\'evy processes. We also study spectral expansion and compactness of the semigroup, and provide a class of compact non-local OU semigroups admitting a biorthogonal system of eigenfunctions and co-eigenfunctions but no corresponding spectral expansion.
	\end{abstract}
	\maketitle
	\tableofcontents

		\section{Introduction}
		The Ornstein-Uhlenbeck (OU) operator on $\R^d$ is a second order differential operator defined as
		\begin{align}\label{eq:A-OU}
			A^{\mathrm{OU}}=\frac{1}{2}\tr(\Sigma\nabla^2)+\langle Bx,\nabla\rangle,
		\end{align}
		where $\Sigma$ is a nonnegative definite matrix and $B$ is a real matrix. This operator is the generator of a Markov process, known as the Ornstein-Uhlenbeck process that solves the stochastic differential equation
		\begin{align*}
			dX_t=BX_t+dW_t,
		\end{align*}
		where $W=(W_t)_{t\ge 0}$ is the Brownian motion on $\R^d$ with covariance matrix $\Sigma$. OU process is a classical example of a diffusion process with a wide range of applications in mathematical finance, physics, population dynamics and various other fields. The most interesting case is when such processes are ergodic. It is known that the Markov process $X=(X_t)_{t\ge 0}$ defined above is ergodic if and only if all eigenvalues of $B$ have strictly negative real part. Moreover, the limiting distribution is a non-degenerate Gaussian distribution on $\R^d$ if and only if 
		\begin{align}\label{eq:Q_t}
			\Sigma_t:=\mathrm{Var}(X_t)=\int_0^t e^{sB} \Sigma e^{sB^*}ds
		\end{align}
		is non-singular for every $t>0$. This condition is also equivalent to the hypoellipticity of the operator $A^{\mathrm{OU}}$, that is, $A^{\mathrm{OU}} u=f$ has smooth solutions whenever $f$ is smooth. Denoting the limiting distribution of $X$ by $\nu$, it is known from \cite[Theorem~2.4]{MichalikGoldys2002} that $A^{\mathrm{OU}}$ is non-self-adjoint in $\L^2(\nu)$ unless $\Sigma B=B^* \Sigma$. Hence, $A^{\mathrm{OU}}$ is in general a non-self-adjoint operator in $\L^2(\nu)$. This makes the study of the spectral theory of OU operators a very interesting problem from an analytic point of view. Moreover, having information about the spectrum provides deep insight about the speed of convergence of the process towards its invariant distribution.
		
		The spectrum of $A^{\mathrm{OU}}$ in both $\L^p(\nu)$ and $\L^p(\R^d)$ have been studied in great details in the past two decades. Metafune, Pallara, and Priola \cite{MetafunePallaraPriola2002} gave an exact description of the $\L^p(\nu)$-spectrum of $A^{\mathrm{OU}}$ when $\nu$ is a non-degenerate Gaussian distribution on $\R^d$. More precisely, it was proved in \cite{MetafunePallaraPriola2002} that the spectrum is discrete and equals
		\begin{align}\label{eq:N_lambda}
			\mathbb{N}(\lambda)=\left\{-\sum_{i=1}^d n_i \lambda_i: n_i\in\mathbb{N}_0\right\},
		\end{align}
		where $\lambda_1,\ldots, \lambda_d$ are eigenvalues of $-B$, counted with multiplicities, and the spectrum does not depend on $1<p<\infty$. In addition, they studied multiplicity of the eigenvalues, which essentially depends only on the drift matrix $B$. We refer to \cite{Metafune_et_al2021} for the analysis of $\L^p(\R^d)$-spectrum of $A^{\mathrm{OU}}$, and interestingly, the spectrum is $p$-dependent in this case.
		Aside from OU operators, we refer to Eckmann and Hairer \cite{EckmanHairer2003} for spectral theory of hypoelliptic operators in $\L^2(\R^d)$. We also refer to the works of Hempel and Voigt \cite{HempelVoigt1986} and Davies \cite{Davies1996}  for the $\L^p$-spectral independence phenomenon of some elliptic operators.
		
		In this article, we consider non-local perturbations of the Ornstein-Uhlenbeck operator $A^{\mathrm{OU}}$. For a $\sigma$-finite measure $\Pi$ on $\R^d$ satisfying 
		\begin{align*}
			\Pi(\{0\})=0 \quad \mbox{and} \quad \int_{\bb{R}^d} (1\wedge |y|^2)\Pi(dy)<\infty,
		\end{align*}
		the \emph{L\'evy-Ornstein-Uhlenbeck} (L\'evy-OU) operator on $\R^d$ is defined as
		\begin{equation}\label{eq:OU_gen}
			\begin{aligned}
			Au(x)&=\frac{1}{2}\tr(\Sigma\nabla^2 u(x))+\langle Bx,\nabla u(x)\rangle \\
			& \ \ \ +\int_{\bb{R}^d}\left[u(x+y)-u(x)-\langle \nabla u(x),y\rangle\bbm{1}_{\{|y|\le 1\}}\right]\Pi(dy),
			\end{aligned}
		\end{equation}
		where $\Sigma$ is a nonnegative definite matrix.
		This operator is the generator of a Markov process $X=(X_t)_{t\ge 0}$ which solves the stochastic differential equation with jumps
		\begin{align}\label{eq:Levy-OU-process}
			dX_t=BX_t+dZ_t,
		\end{align}
		where $Z=(Z_t)_{t\ge 0}$ is a L\'evy process (see \S\ref{sec:preliminaries_levy}) with the L\'evy-Khintchine exponent 
		\begin{equation*}
			\begin{aligned}
			\Psi(\xi)=
			 -\frac{1}{2}\langle \Sigma\xi,\xi\rangle+\int_{\R^d}(e^{\i\langle \xi, y\rangle}-1-\i\langle\xi,y\rangle\mathbbm{1}_{\{|y|\le 1\}})\Pi(dy).
			\end{aligned}
		\end{equation*}
		Equivalently, one has the following stochastic integral representation of $X_t$:
		\begin{align}\label{eq:Levy_rep}
			X_t=e^{tB} X_0 + \int_0^t e^{(t-s)B} dZ_s.
		\end{align}
		Throughout the article we assume the following condition.
	\begin{assumption}[Ergodicity]\label{ergodicity}
		 $B$ is a real matrix with all eigenvalues having strictly negative real part, and the L\'evy measure $\Pi$ satisfies 
		\begin{align}\label{eq:log-moment-cond}
			\int_{\{|x|>1\}} \log |x| \Pi(dx)<\infty.
		\end{align}
	\end{assumption}
	Under H\ref{ergodicity}, Sato and Yamazato \cite{SatoYazamato1984} proved that the L\'evy-OU process defined in \eqref{eq:Levy-OU-process} is ergodic and its limiting distribution is infinitely divisible, which we denote by $\mu$.
	 We note that $A$ is a non-local operator and is in general non-normal in $\L^2(\mu)$. In fact, the non-normality is a non-trivial fact and we refer to Theorem~\ref{thm:normality} for a proof. The Markov semigroup corresponding to $X$ given by \eqref{eq:Levy-OU-process} is called the \emph{L\'evy-OU semigroup}, which can be uniquely extended as a strongly continuous contraction semigroup on $\L^p(\mu)$ for all $p\ge 1$. The L\'evy-OU semigroups are also known as \emph{generalized Mehler semigroups} and we refer to \cite{BogachevRocknerSchmuland1996, Michalik1987, FurmanRockner2000, LescotRockner2002, LescotRockner2004, AlbeverioBogachevRockner1999} for a detailed study on generalized Mehler semigroups associated with infinite dimensional L\'evy processes. We also refer the interested reader to the beautiful survey by Bogachev \cite{Bogachev2018} for a thorough treatment of Ornstein-Uhlenbeck operators and generalized Mehler semigroups in infinite dimensional spaces.
		
	 While the functional inequalities and ergodic properties for generalized Mehler semigroups have been studied in the aforementioned references, not much is known about the spectral theory of these non-local operators. In the context of L\'evy processes, Kwa\'snicki \cite{Kwasnicki2025} considered a class of non-normal operators defined on $\L^2((0,\infty))$ related to one-dimensional stable L\'evy processes and provided a spectral theoretic approach for studying the distribution of supremum of such processes. That being said, L\'evy-OU operators, although being a perturbation of L\'evy operators (generators of L\'evy processes) by a first order differential operator, their spectral properties are very different from each other. For instance, the Laplace operator $\frac12\Delta$ on $\R^d$ has purely continuous spectrum in $\L^2(\R^d)$, while the OU operator $\frac12\Delta-\langle x,\nabla\rangle$ has purely discrete spectrum in the $\L^2$ space weighted with the Gaussian measure $e^{-|x|^2}dx$. From analytical point of view, the spectral theory of L\'evy-OU operators is a very interesting and non-trivial problem because of the following reasons:
		\begin{itemize}[leftmargin=*]
			\item These operators are non-local and non-normal. The diffusion component of these operators may be absent, that is, such operators can be purely non-local.
			\item Even in the diffusion case, there are no explicit formulas for the eigenfunctions and co-eigenfunctions in the general non-normal case. We refer to \cite{ChenLiu2014, Chen2015} for eigenfunctions of complex OU operators defined on the complex plane, which are normal operators. When $\Sigma$ and $B$ are simultaneously diagonalizable and $B$ is normal, expressions of the eigenfunctions have been obtained in \cite{ZhangSahaiMarzouk2021}. We point out that $A^{\mathrm{OU}}$ is still a normal operator under such conditions. We are not aware of any prior results in the non-normal or non-local case.
		\end{itemize}
		
		\subsection{Our contribution} The main goal of our work is proving invariance properties of the spectrum, and to obtain eigenfunctions and co-eigenfunctions (eigenfunctions of the adjoint) for all non-local OU operators satisfying some mild conditions. 
		Below we briefly describe the main highlights of this paper.

		\subsubsection*{Singularity of $\Sigma_\infty$} We develop the spectral theory of the L\'evy-OU operator without the non-singularity assumption of  $\Sigma_\infty$ (see \eqref{eq:Q_t} for definition), that is, the corresponding L\'evy-OU process may have a degenerate Gaussian component at all times. In the existing works on L\'evy-OU semigroups in infinite-dimensional spaces, it is assumed that 
		\begin{align*}
			\ker(\Sigma_\infty)=\{0\},
		\end{align*}
		which in the finite dimensional setup is equivalent to assuming that $\Sigma_t$ is invertible for every $t>0$. This non-degeneracy condition is also assumed in \cite{MetafunePallaraPriola2002}. While this assumption is necessary in the diffusion case, as otherwise, the OU process may become completely deterministic, in the presence of jumps, one can still have a non-degenerate L\'evy-OU process (i.e., its distribution is supported on the entire Euclidean space) when $\Sigma=0$. We derive exact formulas for eigenfunctions, co-eigenfunctions and their biorthogonality without any non-degeneracy condition on $\Sigma$, and these results hold when $\Sigma=0$, see \S\ref{sec:eigen_coeigen} for details.

		\subsubsection*{Finding the spectrum} In Theorem~\ref{thm:spectrum}, we prove that the set $\mathbb{N}(\lambda)$ defined in \eqref{eq:N_lambda} is included in the $\L^p(\mu)$-spectrum of $A$ for any $1<p<\infty$. Moreover, when the L\'evy measure $\Pi$ has finite moments for all orders and $\Sigma_t$ defined in \eqref{eq:Q_t} is non-singular for all $t>0$, $\mathbb{N}(\lambda)$ is exactly equal to the point spectrum of $A$ in $\L^p(\mu)$ for all $1<p<\infty$. While the invariance of the spectrum in $\L^p(\mu)$ is not surprising as it also holds in the diffusion case, it is noteworthy that the point spectrum is independent of the non-local perturbation. The multiplicities of the eigenvalues (algebraic and geometric) also remain invariant with respect to the non-local perturbations. To the best of our knowledge, these phenomena have not been observed in the context of L\'evy-OU operators before.
		
		\subsubsection*{Exact formula for eigenfunctions and co-eigenfunctions} In Theorem~\ref{thm:spectral_expansion}, assuming the Hartman-Wintner condition on the invariant distribution, we give an explicit Rodrigues-type formula for the co-eigenfunctions of $A$ in terms of the derivatives of the invariant density. This representation of the co-eigenfunction is reminiscent of the fact that Hermite polynomials are defined as the Rodrigues operator applied to the Gaussian function, which is the invariant distribution for the self-adjoint diffusion OU semigroup. Additionally, we prove uniform boundedness of the co-eigenfunctions with respect to the $\L^2(\mu)$-norm.
		
		The eigenfunction formula is proved in Theorem~\ref{it:3} under the assumption that the L\'evy measure has finite moments of any order. These eigenfunctions are generalizations of Hermite polynomials, and \eqref{eq:eigen_polynomial} provides the exact expression of the coefficients in terms of the limiting distribution of the L\'evy-OU semigroup. Interestingly, for the one-dimensional case, the eigenfunctions are given by Appell sequences (see \S\ref{sss:appell}), which are of key importance when solving optimal stopping problems associated to L\'evy processes, see \cite{Novikov-Shiryaev2004, Kyprianou-Surya2005, Salmien2011}, and for constructing time-space martingales for L\'evy processes, see \cite{Sole_Utzet2008, Sole_Utzet2}.

		We also obtain the exact formula for the $\L^2(\mu)$-norm of the eigenfunctions. This provides a systematic way to check if the L\'evy-OU semigroup admits a spectral expansion in terms of the eigenfunctions and co-eigenfunctions. 
		\subsubsection*{Biorthogonality of eigenfunctions and co-eigenfunctions} For normal operators with discrete spectrum, one gets an orthonormal sequence of eigenfunctions that forms a basis for the Hilbert space. For non-normal operators, one cannot expect orthogonality of the eigenfunctions. In Theorem~\ref{thm:spectral_expansion}, we show that when $B$ is diagonalizable, there exist sequences of \emph{biorthogonal} functions $(\mathcal{H}_n)_{n\in\mathbb{N}^d_0}$ and $(\mathcal{G}_n)_{n\in\mathbb{N}^d_0}$, that is, $\langle\mathcal{H}_n,\mathcal{G}_m\rangle_{\L^2(\mu)}=\delta_{mn}$ such that 
			\begin{align*}
				A\mathcal{H}_n=-\langle n,\lambda\rangle \mathcal{H}_n, \quad A^* \mathcal{G}_n=-\langle n,\lambda \rangle \mathcal{G}_n \text{ for all $n\in \mathbb{N}^d_0$}.
			\end{align*}
			While $(\mathcal{H}_n)_{n\in\mathbb{N}^d_0}$ is a sequence of polynomials spanning the entire space of polynomials, the functions $\mathcal{G}_n$ need not be polynomials. In fact, $\mathcal{G}_n$ is a polynomial if and only if $\Pi=0$, that is, the underlying process is a diffusion.
			Furthermore, the diagonalizability of $B$ is a necessary condition for the existence of a biorthogonal sequence of eigenfunctions and co-eigenfunctions, see Theorem~\ref{thm:diagonalizable}. We point out that biorthogonal functions are of special interest in integrable probability, and we refer the interested reader to the seminal paper by Borodin \cite{Borodin1999} on biorthogonal ensembles, and to Borodin-Corwin-Petrov-Sasamoto \cite{BCPS2015, BCPS2015(2)} for biorthogonality of right and left-eigenfunctions (co-eigenfunctions in our terminology) of some operators originating from interacting particle systems.
			
		\subsubsection*{Spectral expansion} In Theorem~\ref{thm:spectral_exp} we obtain the spectral expansion of the diffusion OU semigroup, which is a non-normal operator in general. More precisely, we prove that when $\Pi=0$ and $B$ is diagonalizable,
			\begin{align}\label{eq:spectral_expansion1}
				P_t f=\sum_{n\in \mathbb{N}^d_0}{e^{-t\langle n,\lambda\rangle}} \langle f,\mathcal{G}_n\rangle_{\L^2(\nu)}\mathcal{H}_n,
			\end{align}
			where $\mathcal{H}_n$ and $\mathcal{G}_n$ are the eigenfunctions and co-eigenfunctions of $P_t$ with respect to the eigenvalue $e^{-t\langle n,\lambda\rangle}$, and the above spectral expansion holds for $t>t_0$ for some positive number $t_0$. This occurs due to the exponential growth of the spectral projection norms of the non-normal diffusion OU semigroup. Such an observation was also made by Davies and Kuijlaars \cite{DaviesKuijlaars2004} in the context of non-self-adjoint harmonic oscillators. We refer to \cite{Patie_Savov, Patie_Savov2, CPSV, Patie_Miclo_Sarkar} for spectral expansion formulas resembling \eqref{eq:spectral_expansion1} in the context of non-local Jacobi semigroups, discrete Laguerre semigroups, Gauss-Laguerre semigroups, and generalized Laguerre semigroups respectively. 
			
			When $\Pi\neq 0$, the L\'evy-OU semigroup may not admit a spectral expansion with respect to its eigenfunctions and co-eigenfunctions, even if the semigroup is compact with eigenvalues in the set $\mathbb{N}(\lambda)$. This interesting phenomenon can be attributed to the super-exponential growth of the norm of eigenfunctions of L\'evy-OU operators, see Theorem~\ref{thm:non-existence}.
			\subsubsection*{Compactness of L\'evy-OU semigroups} We provide some necessary and sufficient conditions for compactness of L\'evy-OU semigroups $P=(P_t)_{t\ge 0}$ in \S\ref{sec:compactness}. In the diffusion case (i.e.~$\Pi\equiv 0$), such conditions are known in infinite-dimensional setting, see \cite{Michalik1987}. In the non-local case (i.e. $\Pi\neq 0$), a sufficient condition for compactness of $P_t$ has been obtained in \cite{OuyangRocknerWang2012}. More precisely, the authors proved (in the infinite dimensional setting) that $P_t$ is eventually compact in $\L^p(\mu)$ for all $1<p<\infty$ if
			\begin{align*}
				\int_{\R^d}\left[\int_{\R^d}\exp\left(-\frac{\alpha \|\Sigma^{-1/2}_t e^{tB}\|^2}{2(\alpha-1)}|x-y|^2\right)\mu(dy)\right]^{-1}\mu(dx)<\infty
			\end{align*}
			for some $\alpha>1$. This condition is difficult to verify even in finite-dimensional setting. In Theorem~\ref{thm:compactness} we provide a simplified sufficient condition for compactness of the L\'evy-OU semigroup. Moreover, we prove in Theorem~\ref{thm:nec_compact} that the L\'evy measure must have finite moments of all orders for $P_t:\L^p(\mu)\longrightarrow \L^p(\mu)$ to be compact for some $t>0$ and some $1<p<\infty$. Our result directly implies that the L\'evy-OU semigroups corresponding to $\alpha$-stable L\'evy processes are not compact for any $\alpha\in (0,2)$.

			\subsection{Use of intertwining} The central idea developed in \cite{MetafunePallaraPriola2002} stems from the following important observation: $\gamma\in\sigma(A^{\mathrm{OU}})$ if and only if there exists a homogeneous polynomial $u$ such that $\langle Bx, \nabla u\rangle =\gamma u$. This has an implication that the two operators $A^{\mathrm{OU}}$ and $\langle Bx, \nabla\rangle$ should be strongly related.
			Chojnowska-Michalik and Goldys \cite[Theorem~1]{Michalik1987} proved an elegant result that any diffusion OU semigroup in $\L^2(H,\nu)$, where $H$ is a Hilbert space and $\nu$ is the unique invariant distribution of the semigroup, can be viewed as a second quantized operator. Simplifying their results in the finite-dimensional case, one can write 
			\begin{align}\label{eq:second_quant}
				Q_t = \Gamma(S^*_0(t)),
			\end{align}
			where $S_0(t): \R^d\longrightarrow \R^d$ is the contraction semigroup generated by the finite dimensional matrix $\Sigma^{-1/2}_\infty B \Sigma^{1/2}_\infty$, and $\Gamma$ is the second quantization operator defined on the symmetric Fock space over $\R^d$, see \cite{Michalik1987, vanNeerven} for details.
			Using \eqref{eq:second_quant}, van Neerven \cite{vanNeerven} extended \cite[Theorem~3.1]{MetafunePallaraPriola2002} for diffusion OU semigroups in infinite-dimensional spaces. For non-local OU operators, a second quantized operator representation was obtained by Peszat \cite{Peszat}. Even though such a representation can be useful for obtaining the spectral gap inequality in some special non-local cases, see \cite[Theorem~7.3]{Peszat}, the other spectral properties such as multiplicities of eigenvalues, eigenfunctions and co-eigenfunctions, and the spectral expansion are difficult to infer from this method.
			
			In this paper, we propose a different approach based on \emph{intertwining relationship}, where we relate the non-local OU operator with some diffusion operators. Two densely defined closed operators $P:\cc{X}\longrightarrow\cc{X}$ and $Q:\cc{Y}\longrightarrow\cc{Y}$, where $\cc{X}, \cc{Y}$ are Banach spaces, are \emph{intertwined via the link operator $\Lambda$} if there exists a bounded operator $\Lambda:\cc{Y}\longrightarrow\cc{X}$ such that 
		\begin{align}\label{eq:intertwining_def}
			P\Lambda=\Lambda Q.
		\end{align}
		In particular, when $\Lambda$ is invertible, $P$ and $Q$ are similar operators, and in such a case, they share the same spectral properties. Although $\Lambda$ is not assumed to be invertible in general, it is still possible to transfer spectral properties between $P$ and $Q$. The theme of this paper is closely aligned with the series of works of Patie with other co-authors \cite{Patie_Savov, Patie_Miclo_Sarkar, Patie_Sarkar, Patie_Vaidyanathan}, and we refer the reader to the above references for a detailed account on spectral expansion of non-self-adjoint Markov semigroups using intertwining. The most crucial component of proving intertwining relationship lies in finding an appropriate link operator. Existence and construction of such link operators for intertwining two diffusion operators satisfying various technical conditions have been studied in a recent article by Budway, Pal, and Shkolnikov \cite{PalShkolnokovBudway} using probabilistic techniques. However, the L\'evy-OU operators being non-local, we propose a different approach for constructing the link operator using the pseudo-differential operator representation of the semigroup, see \eqref{eq:characteristic_fn}. We prove that
		\begin{enumerate}[leftmargin=*]
			\item under some moment conditions on the L\'evy measure, there exists a Markov operator $\Lambda$ such that 
			\begin{align}\label{eq:AL}
				L\Lambda=\Lambda A, \quad L=\langle Bx,\nabla\rangle,
			\end{align}
			see Proposition~\ref{prop:left_int}, and
			\item when $\Sigma_\infty$ is invertible and $B$ is diagonalizable, there exists a normal diffusion operator $A^{\mathrm{OU}}_\varrho$ such that
			\begin{align}\label{eq:AVA}
				A^{\mathrm{OU}}_\varrho V = V A 
			\end{align}
			for some Markov operator $V$, see Theorem~\ref{thm:intertwining_diagonal}.
		\end{enumerate}
		The identity \eqref{eq:AL} explains why the spectrum of $A$ should indeed be related to the spectrum of $L$, which was observed in \cite{MetafunePallaraPriola2002} for diffusion OU operators. The second identity \eqref{eq:AVA} leads to the exact formula of eigenfunctions and co-eigenfunctions, and their biorthogonality. It is to be noted that such results hold even when $\Sigma_\infty$ is singular and this requires subtle approximation techniques along with some additional intertwining relationships, which have been carried out in \S\ref{sec:approx_epsilon}.
		
		The rest of the paper is organized as follows: after introducing the notations and conventions in Section~\ref{sec:notation}, the main results and some examples are discussed in \S\ref{sec:main} and \S\ref{sec:examples} respectively. In \S\ref{sec:intertwining}, we provide details of the intertwining relationships and related results in our setting. Finally, the proofs of the main results are split into \S\ref{sec:pf}, \S\ref{sec:exp_moment}, \S\ref{sec:spect_exp}, and \S\ref{sec:compact}.

		\section{Notation}\label{sec:notation}
		\noindent
		For any operator $T$ on a Banach space $\cc{X}$, we write $\Dom(T)\subseteq \cc{X}$ to indicate its domain, and we use $\sigma(T; \cc X)$, $\sigma_p(T; \cc X)$, $\sigma_c(T; \cc X)$ to denote the full spectrum, point spectrum and continuous spectrum of $T$ respectively. When $\cc X$ is finite dimensional, we simply use $\sigma(T)$ to denote its spectrum. For a bounded operator $T$, we use $\|T\|$ to denote its operator norm. For any nonnegative integer $d$, $C_b(\bb{R}^d)$ stands for the set of all bounded continuous functions on $\bb{R}^d$, while $C^\infty(\bb{R}^d),C^\infty_c(\bb{R}^d)$, and $C^\infty_0(\R^d)$ denote the set of all smooth functions, set of all compactly supported smooth functions, and set of all smooth functions with derivatives vanishing at infinity respectively. We denote the Schwartz space of functions by $\mathcal{S}(\bb{R}^d)$ and for any $\sigma$-finite measure $\mu$ on $\bb{R}^d$, $\L^p(\mu)$ denotes the $L^p$ space with respect to $\mu$. When $p=2$, $\L^2(\mu)$ is equipped with the inner product $\langle f,g\rangle_{\L^2(\mu)}=\int_{\R^d} f(x)\overline{g(x)}\mu(dx)$.  We identify $\L^\infty(\mu)$ with $C_b(\R^d)$. For any $p\ge 1$ and a non-negative integer $k$, $\mathrm{W}^{p,k}(\bb{R}^d)$ denotes the usual Sobolev space on $\bb{R}^d$ and we define the weighted Sobolev space with respect to the measure $\mu$ as 
		\begin{align*}\mathrm{W}^{k,p}(\mu)=\left\{u\in\mathrm{W}^{k,p}_{loc}(\bb{R}^d): \partial^n u\in\L^p(\mu) \text{ for } |n|\le k\right\},\end{align*}
		where for any $d$-tuple $n=(n_1,\ldots, n_d)\in\bb{N}^d_0$, we write $|n|=n_1+\cdots+n_d$, and $\partial^n=\partial^{n_1}_{x_1}\cdots \partial^{n_d}_{x_d}$.
		
		Throughout this paper, $\bb{C}_+$ and $\bb{C}_-$ denote the positive and negative open half-planes. For any $f\in\L^p(\bb R^d)$, we denote its Fourier transform by $\cc{F}_f$, that is, for all $\xi\in\bb{R}^d$, 
		\begin{align*}\cc{F}_f(\xi)=\int_{\bb{R}^d} e^{\i\langle\xi, x\rangle} f(x)dx,\end{align*} where $\langle\cdot,\cdot\rangle$ denotes the natural inner product on $\R^d$, and the above integral is defined in $L^2$ sense. Finally, we emphasize that for any $z_1,z_2\in\bb{C}^d$, we denote
		\begin{align*}
			\langle z_1, z_2\rangle=z^\top_1 z_2.
		\end{align*}
		Note that $\langle \cdot, \cdot\rangle$ is not the complex inner product on $\bb{C}^d$.
		\section{Assumptions}
		In addition to the ergodicity assumption H\ref{ergodicity}, we assume the following conditions. While H\ref{assumption_polynomial} and H\ref{hartman-winter} are assumed for most of the results, we  emphasize that many of our results hold \textbf{without} the non-degeneracy condition in H\ref{non-degeneracy}.
		\begin{assumption}[Moments]\label{assumption_polynomial}
			For all $n\ge 1$, 
			\begin{align*}
				\int\limits_{\{|x|>1\}} |x|^n \Pi(dx)<\infty.
			\end{align*}
			This assumption is equivalent to the existence of moments of all orders for the L\'evy process and the L\'evy-OU process. In particular, under this assumption, the invariant distribution has moments of all orders. This ensures that the space of polynomials is included in $\L^p(\mu)$ for all $p\ge 1$.
		\end{assumption}
		\begin{assumption}[Smoothness of density]\label{hartman-winter} Let $\Psi_\infty$ denote the L\'evy-Khintchine exponent of the invariant distribution $\mu$. Then, $\Psi_\infty$ satisfies the Hartman-Wintner condition:
			\begin{align*}
				\lim_{|\xi|\to\infty} \frac{\Re(\Psi_\infty(\xi))}{\log(1+|\xi|)}=-\infty.
			\end{align*}
		\end{assumption}
		\begin{assumption}[Non-degenerate Gaussian component]\label{non-degeneracy} The L\'evy-OU process has a non-degenerate Gaussian component, that is,
			$\det(\Sigma_t)>0$ for all $t>0$. In the diffusion case, this condition is equivalent to the hypoellipticity of the generator $A^{\mathrm{OU}}$. Under this condition, the invariant distribution $\mu$ has a smooth, positive density. Note that H\ref{non-degeneracy} implies H\ref{hartman-winter}.
		\end{assumption}

		\section{Main Results}\label{sec:main}
		As noted before, the L\'evy-OU semigroup $P=(P_t)_{t\ge 0}$ extends uniquely as a strongly continuous contraction semigroup on $\L^p(\mu)$ for all $p\ge 1$, and we denote its generator by $(A_p,\Dom(A_p))$, where 
		\begin{align*}
			\cc{D}(A_p)&=\left\{f\in\L^p(\mu): \lim_{t\to 0}\frac{P_t f-f}{t} \ \ \mbox{exists in $\L^p(\mu)$} \right\}, \\
			A_p f &= \lim_{t\to 0}\frac{P_t f-f}{t}, \quad f\in\mathcal{D}(A_p).
		\end{align*}
		\subsection{The spectrum} \label{sec:spectrum}
		Let us denote the eigenvalues (counted with multiplicities) of $-B$ by $\lambda=(\lambda_1,\ldots,\lambda_d)$, and we recall the set $\mathbb{N}(\lambda)$ defined in \eqref{eq:N_lambda}.
		\begin{theorem}\label{thm:spectrum}
			\begin{enumerate}[leftmargin=*]
			\item \label{it:point_spectrum} If H\ref{assumption_polynomial} holds, then $\mathbb{N}(\lambda)\subseteq \sigma_p(A_p)$ for all $1<p<\infty$.	
			\item \label{it:coeigen_H4} If H\ref{non-degeneracy} holds, then $\mathbb{N}(\lambda)\subseteq \sigma(A_p)$ for all $1<p<\infty$.
			\item \label{it:poly_eigen} If H\ref{assumption_polynomial} and H\ref{non-degeneracy} hold, then $\sigma_p(A_p)=\mathbb{N}(\lambda)$ for all $1<p<\infty$, and the eigenspaces consist of polynomials.
			\end{enumerate}
		\end{theorem}
		\subsection{Multiplicities of eigenvalues and isospectrality}\label{sec:multiplicity} For any closed operator $T:\Dom(T)\subseteq \cc{X}\longrightarrow \cc{X}$, where $\cc{X}$ is a Banach space, the algebraic multiplicity of an eigenvalue $\theta$ of $T$ is defined as 
		\begin{align*}
			\mathtt{M}_a(\theta, T)=\dim\left(\cup_{n=1}^\infty \ker (T-\theta I)^n\right),
		\end{align*}
		and the geometric multiplicity is defined as $\mathtt{M}_g(\theta,T)=\dim(\ker(T-\theta I))$. We note that both algebraic and geometric multiplicities can be infinite. The next result shows that the multiplicities of eigenvalues of $(A_p,\Dom(A_p))$ are independent of the L\'evy process and $p$.
		\begin{theorem}\label{thm:multiplicity}
			Assume that H\ref{assumption_polynomial} and H\ref{non-degeneracy} hold. Then, for all $\theta\in\bb{N}(\lambda)$ and $1<p<\infty$,
			\begin{align*}
				\mathtt{M}_a(\theta, A_p)&=\mathtt{M}_a(\theta, L) \\
				\mathtt{M}_g(\theta, A_p)&=\mathtt{M}_g(\theta, L),
			\end{align*}	
			where $L=\langle Bx, \nabla\rangle$ restricted to the space of all polynomials in $\R^d$. Moreover, $\mathtt{M}_a(\theta, A_p)=\mathtt{M}_g(\theta, A_p)$ for all $\theta\in\bb{N}(\lambda)$ if and only if $B$ is diagonalizable.
		\end{theorem}
		\subsection{Eigenfunctions, co-eigenfunctions, biorthogonality} \label{sec:eigen_coeigen} For a densely defined operator $(T,\mathcal{D}(T))$ on $\L^2(\mu)$, $f\in\mathcal{D}(T^*)$ is called a co-eigenfunction of $T$ corresponding to an eigenvalue $\gamma$ if $T^* f=\gamma f$. In the following results we provide details about the eigenfunctions and co-eigenfunctions of $A_2$, the $\L^2(\mu)$-generator of the L\'evy-OU semigroup. We recall the following fact from linear algebra: if $B$ is a real, diagonalizable matrix possibly having complex eigenvalues, there exists a real invertible matrix $M$ such that 
		\begin{align}\label{eq:diagonalization}
			MBM^{-1}=B_0=\begin{pmatrix}
				D & 0 & \cdots & 0 \\
				0 & C_1 & \cdots & 0 \\
				\vdots & \vdots & \ddots & \vdots \\
				0 & 0 & \cdots & C_r
			\end{pmatrix},
		\end{align}
		where $D$ is a diagonal matrix with real entries, and \begin{align*}C_i = \begin{pmatrix} -a_i & -b_i \\ b_i & -a_i \end{pmatrix}\end{align*} with $b_i\neq 0$ for all $i=1,\ldots, r$. In particular, when all eigenvalues of $B$ have strictly negative real part, the entries of $D$ are strictly negative and $a_i>0$ for all $i=1,\ldots, r$.
		
		\begin{notation}\label{not:R}
			From the above structure of the matrix $B$, we write $x\in\R^d$ as $x=(u,v)$ where $u\in\R^{k}$ with $k=d-2r$ and $v\in\R^{2r}$, and we denote 
			\begin{align*}
				Rx := (u, \zeta, \overline{\zeta}):= \left(u,\frac{v_1+\i v_2}{\sqrt{2}}, \frac{v_1-\i v_2}{\sqrt{2}}, \ldots, \frac{v_{2r-1}+\i v_{2r}}{\sqrt{2}}, \frac{v_{2r-1}-\i v_{2r}}{\sqrt{2}}\right),
			\end{align*}
			where $\zeta_j=v_{2j-1}+ \i v_{2j}\in\mathbb{C}$.
			Note that $R:\mathbb{C}^d \longrightarrow \mathbb{C}^d$ is a unitary transformation and diagonalizes the block diagonal matrix formed by $C_1,\ldots, C_r$ defined in \eqref{eq:diagonalization}.
		\end{notation}
		\begin{notation}\label{not:multi_derivative}
			For any multi-index $n=(n_1,\ldots, n_d)\in\mathbb{N}^d_0$, we denote 
			\begin{align*}
				\partial^n_{\star}&=\partial^{n_1}_{u_1}\cdots \partial^{n_k}_{u_k}\partial^{n_{k+1}}_{\overline{\zeta}_1}\partial^{n_{k+2}}_{\zeta_1}\cdots \partial^{n_{d-1}}_{\overline{\zeta}_r}\partial^{n_d}_{\zeta_r}, \\
				\overline{\partial}^n_{\star} &= \partial^{n_1}_{u_1}\cdots \partial^{n_k}_{u_k}\partial^{n_{k+1}}_{\zeta_1}\partial^{n_{k+2}}_{\overline{\zeta}_1}\cdots \partial^{n_{d-1}}_{\zeta_r}\partial^{n_d}_{\overline{\zeta}_r},
			\end{align*}
			where for any $1\le i\le r$, 
			\begin{align*}
				\partial_{\zeta_i}=\frac{1}{\sqrt{2}}\left(\partial_{v_{2i-1}}-\i \partial_{v_{2i}}\right), \quad \partial_{\overline{\zeta}_i}=\frac{1}{\sqrt{2}}\left(\partial_{v_{2i-1}}+\i \partial_{v_{2i}}\right).
			\end{align*}
		\end{notation}
		With respect to this notation, it can be easily verified that for any smooth function $f:\R^d\longrightarrow\R$, 
		\begin{align}\label{eq:complex_derivative}
			\partial^n f(\overline{R}^* x)=\partial^n_\star f(x), \quad \partial^n f(R^* x)=\overline{\partial}^n_{\star} f(x)
		\end{align}
		for all $n\in\mathbb{N}^d_0$.
		
\begin{theorem}[Eigenfunctions]	\label{it:3}
	Assume that H\ref{assumption_polynomial} holds, $B$ is diagonalizable, and $M$ is the real invertible matrix defined in \eqref{eq:diagonalization}. Consider the polynomial $\mathcal{H}_n$ defined by
	\begin{equation}\label{eq:eigen_polynomial}
		\begin{aligned}
			\cc{H}_n(x)=\frac{2^{\frac{|n|}{2}}}{\sqrt{n!}}\sum_{0\le k\le n} (-\i)^{|k|} \dbinom{n}{k} (RMx)^{n-k}\partial^{k}_{\star}\left(e^{-\Psi_\infty\circ M^*}\right)(0)
		\end{aligned}
	\end{equation}
	where $k\le n$ means $k_i\le n_i$ for all $1\le i\le d$, 
	\begin{gather*}
		x^{n}=x^{n_1}_1 \cdots x^{n_d}_d, \quad 
		 n!=n_1! \cdots n_d!, \\
		 \dbinom{n}{k}=\prod_{i=1}^d \dbinom{n_i}{k_i}.
	\end{gather*}
	 Then, the following holds.
	 \begin{enumerate}[leftmargin=*]
	  \item \label{it:poly} For all $1\le p<\infty$, 
	\begin{align*}
		A_p \mathcal{H}_n=-\langle n,\lambda\rangle \mathcal{H}_n \quad \mbox{for all $n\in\mathbb{N}^d_0$}.
	\end{align*}
	\item \label{it:norm} $\Span\{\cc{H}_n: n\in\bb{N}^d_0\}=\mathscr{P}$, where $\mathscr{P}$ is the space of all polynomials. Moreover, for any $n\in\mathbb{N}^d_0$,
	\begin{equation}\label{eq:L^2_norm}
		\begin{aligned}
		& \ \  	n! 2^{-|n|}\|\mathcal{H}_n\|^2_{\L^2(\mu)} \\
			& =
			\left.\partial^n_{\star,z}\overline{\partial}^n_{\star,w}\exp(\Psi_\infty( M^*(z-w))-\Psi_\infty( M^* z)-\overline{\Psi_\infty( M^* w)})\right|_{z=w=0}.
		\end{aligned}
	\end{equation}
\item \label{it:eigen_basis} In addition, if H\ref{non-degeneracy} holds, then the eigenspace of $-\langle n,\lambda\rangle$ is given by
	\begin{align*}
		E_{-\langle n,\lambda\rangle}=\Span\{\mathcal{H}_m: \langle m,\lambda\rangle=\langle n,\lambda\rangle\rangle \}.
	\end{align*}
	\end{enumerate}
	\end{theorem}
	\begin{remark}
		\begin{enumerate}[leftmargin=*]
		\item Due to assumption H\ref{assumption_polynomial}, the invariant distribution $\mu$ has moments of all orders, and therefore, $\Psi_\infty$ is a smooth function. This ensures that the right hand sides of \eqref{eq:eigen_polynomial} and \eqref{eq:L^2_norm} are well-defined.
		\item  When $B$ has all real eigenvalues, $\mathcal{H}_n$ is a real polynomial, since $R$ becomes the identity matrix, and the derivatives of $e^{-\Psi_\infty\circ M^*}$ cancel the imaginary coefficient $(-\i)^{|k|}$.
		\item The scaling constant $\frac{2^{|n|/2}}{\sqrt{n!}}$ in \eqref{eq:eigen_polynomial} is to ensure that when $Q=-B=I_{d\times d}$, and $\Pi=0$, $\mathcal{H}_n$ has unit $\L^2$ norm. In this case, $\mathcal{H}_n$ coincides with the Hermite polynomial in $\R^d$.
		\end{enumerate}
	\end{remark}
		\begin{theorem}[Co-eigenfunctions]\label{thm:spectral_expansion}
	 Let $A^*_2$ denote the $\L^2(\mu)$-adjoint of $A_2$. If H\ref{hartman-winter} holds, $B$ is diagonalizable, and $M$ is the matrix defined in \eqref{eq:diagonalization}, then for any $n\in \mathbb{N}^d_0$ we have
	 \begin{align*}
	 	A^*_2 \cc{G}_n=-\overline{\langle n,\lambda\rangle}\cc{G}_n,
	 \end{align*}
	 where
	 \begin{align}\label{eq:co-eigen}
	 	\cc{G}_n(x)=\frac{(-1)^{|n|}}{\sqrt{2^{|n|}n!}}\frac{S_M\partial^n_{\star} S^{-1}_M\mu(x)}{\mu(x)},
	 \end{align}
	 with $n!=n_1!\cdots n_d!$, $S_M f(x)=f(Mx)$ for all $x\in\R^d$, and $\partial^n_\star$ is defined in Notation~\ref{not:multi_derivative}.
In particular, $\|\mathcal{G}_n\|_{\L^2(\mu)}\le 1$ for all $n\in\mathbb{N}^d_0$.
	\end{theorem}
	\begin{remark}
		Note that the above theorem implies that if $B$ is diagonalizable, it is enough to assume H\ref{hartman-winter} to have $\mathbb{N}(\lambda)\subseteq \sigma_p(A^*_2)$. This result is stronger than the statement in Theorem~\ref{thm:spectrum}\eqref{it:coeigen_H4}.
	\end{remark}
	For the next result, we introduce the definition of biorthogonal sequences.
	\begin{definition}
	In a Hilbert space, two sequences $(h_n)$, $(g_n)$ are said to be \emph{biorthogonal} if 
	\begin{align*}
		\langle h_n, g_m\rangle=\delta_{mn} \quad \mbox{for all $m,n$},
	\end{align*}
	where $\langle\cdot, \cdot\rangle$ is the inner product in the Hilbert space.
\end{definition}
	From the above definition it follows that any orthonormal sequence is biorthogonal to itself. The notion of biorthogonality comes from the spectral expansion of non-normal operators. When an operator is normal, its eigenfunctions form a complete orthonormal system in the Hilbert space. For non-normal operators, one does not have orthogonality of eigenfunctions. A non-normal operator may admit a biorthogonal sequence of eigenfunctions and co-eigenfunctions, but existence of such sequences is not always guaranteed. For instance, in a finite dimensional Hilbert space, an operator admits a biorthogonal sequence of eigenfunctions and co-eigenfunctions if and only if it is diagonalizable. The next theorem states that the eigenfunctions and co-eigenfunctions of the L\'evy-OU operator are biorthogonal.
\begin{theorem}[Biorthogonality]\label{thm:biorthogonality}
	Assume H\ref{assumption_polynomial} and H\ref{hartman-winter}, and that $B$ is diagonalizable. Then,
\begin{align*}
	 \langle\cc{H}_n, \cc{G}_m\rangle_{\L^2(\mu)}=\delta_{mn} \quad \mbox{for all $m,n\in\mathbb{N}^d_0$}.
\end{align*}
Moreover, if the space of polynomials is dense in $\L^2(\mu)$, $(\mathcal{G}_n)_{n\in\mathbb{N}^d_0}$ is the only sequence biorthogonal to $(\mathcal{H}_n)_{n\in\mathbb{N}^d_0}$.
\end{theorem}
	\begin{remark}
The assumption about diagonalizability of $B$ in Theorem~\ref{it:3}, Theorem~\ref{thm:spectral_expansion}, and Theorem~\ref{thm:biorthogonality} is not very restrictive. In fact, the diagonalizability condition is necessary for existence of biorthogonal sequences of eigenfunctions and co-eigenfunctions, and a precise statement is given in the next theorem.  
	\end{remark}
		\begin{theorem}\label{thm:diagonalizable} Suppose that H\ref{assumption_polynomial}, H\ref{hartman-winter} are satisfied and there exist sequences of functions $(h_n)_{n\in\mathbb{N}^d_0}$ and $(g_n)_{n\in\mathbb{N}^d_0}$ in $\L^2(\mu)$ such that
			\begin{enumerate}[leftmargin=*]
				\item $A_2h_n=-\langle n,\lambda\rangle h_n$ for all $n\in\mathbb{N}^d_0$,
				\item $A^*_2g_n= -\langle n,\lambda\rangle g_n$ for all $n\in\bb{N}^d_0$, 
				\item $\langle h_n,g_m\rangle=\delta_{mn}$ for all $m,n\in\mathbb{N}^d_0$, and
				\item  $\Span\{h_n: n\in\bb{N}^d_0\}=\mathscr{P}$.
			\end{enumerate}
			Then, $B$ is diagonalizable.
		\end{theorem}
		
	\subsubsection{Relation to Appell sequences}\label{sss:appell} Appell polynomials were originally introduced by Paul Appell \cite{Appell1880} in 1880, which are defined as the system of polynomials $(P_n)_{n\ge 0}$ on $\R$ such that 
	\begin{align}\label{eq:appell_rec}
		\frac{d}{dx} P_n(x)= n P_{n-1}(x) \quad \mbox{for all $n\ge 1$}.
	\end{align}
	 In \cite{Ta2015}, a probabilistic definition of Appell sequence was given as follows:
	Let $\mu$ be a probability measure on $\R$ such that there exists $\epsilon>0$ satisfying
	\begin{align*}
		M_\mu(\xi)=\int_{\R} e^{x\xi}\mu(dx)<\infty \quad \mbox{for all $|\xi|<\epsilon$}.
	\end{align*}
	Then, the Appell polynomials associated with $\mu$ are defined as 
	\begin{align*}
		P^{(\mu)}_n(x)&= \left.\frac{d^n}{d\xi^n} \frac{e^{x\xi}}{M_\mu(\xi)}\right|_{\xi=0}, \\
		P^{(\mu)}_0(x)&\equiv 1
	\end{align*}
	It can be verified that the system $(P^{(\mu)}_n)_{n\ge 0}$ satisfies the recursion in \eqref{eq:appell_rec}. When $\mu$ has a smooth density, by \cite{Streit_et_al1997}, the sequence $(P^{(\mu)}_n)_{n\ge 0}$ admits a unique biorthogonal sequence which is given by 
	\begin{align*}
		Q^{(\mu)}_n(x)=(-1)^n\frac{1}{\mu(x)}\frac{d^n}{dx^n}\mu(x),
	\end{align*} 
	that is,
	\begin{align*}
		\int_{\R} P^{(\mu)}_n(x) Q^{(\mu)}_m(x) \mu(x)dx=n!\delta_{mn},
	\end{align*}
	provided that $Q^{(\mu)}_n\in \L^2(\mu)$ for each $n\ge 1$. A more general treatment of Appell sequences and their biorthogonality can be found in \cite{Albeverio_Daletsky_Kondratiev_Streit1996} in the context of studying non-Gaussian smooth measures on infinite-dimensional spaces.
	
	In the context of one-dimensional L\'evy-OU processes, when $\Pi$ satisfies H\ref{assumption3}, by Lemma~\ref{lem:moment}, the invariant distribution $\mu$ has exponential moment of small orders. Therefore, by Theorem~\ref{it:3} and \cite[Eq.~(2) and (3)]{Ta2015}, the eigenfunctions $(\mathcal{H}_n)_{n\ge 0}$ coincide with the Appell polynomials corresponding to the invariant distribution $\mu$, up to some multiplicative constants. When H\ref{hartman-winter} holds, $\mu$ admits a smooth density, and in that case the co-eigenfunctions $(\mathcal{G}_n)_{n\ge 0}$ are given by the corresponding biorthogonal sequence, up to some normalizing constants.

	\subsection{Normality of L\'evy-OU operator} A densely defined closed operator $(T,\mathcal{D}(T))$ on a Hilbert space is said to be normal if 
	\begin{enumerate}[leftmargin=*]
		\item $\mathcal{D}(T)=\mathcal{D}(T^*)$, and 
		\item $TT^*=T^*T$ on $\mathcal{D}(T)$.
	\end{enumerate}
	In particular, if $T$ generates a strongly continuous contraction semigroup denoted by $(e^{tT})_{t\ge 0}$, $(T,\mathcal{D}(T))$ is normal if and only if $e^{tT}$ is a bounded normal operator for every $t\ge 0$. A L\'evy operator defined by 
	\begin{align*}
		\mathcal{L}f(x)=\frac{1}{2}\tr(\Sigma\nabla^2 u(x)) 
		+\int_{\bb{R}^d}\left[u(x+y)-u(x)-\langle \nabla u(x),y\rangle\bbm{1}_{\{|y|\le 1\}}\right]\Pi(dy)
	\end{align*}
	is always a normal operator in $\L^2(\R^d)$. In particular, $\mathcal{L}$ is self-adjoint in $\L^2(\R^d)$ if and only if $\Pi$ is a symmetric L\'evy measure, that is, $\Pi(E)=\Pi(-E)$ for every Borel subset $E$ of $\R^d$. In the next result we provide a necessary and sufficient condition for the L\'evy-OU operator to be normal (resp. self-adjoint) in $\L^2(\mu)$.
	 \begin{theorem}\label{thm:normality}
		 Assume H\ref{assumption_polynomial} and H\ref{non-degeneracy}. Then,
		 \begin{enumerate}[leftmargin=*]
		  \item \label{it:normal} $A_2$ is normal in $\L^2(\mu)$ if and only if $B$ is diagonalizable, $\Pi=0$, and $\Sigma^{-1/2}_{\infty} B\Sigma^{1/2}_\infty$ is a normal matrix.
		  \item \label{it:sa} \label{it:s.a=normal} When $B$ has real eigenvalues, $A_2$ is normal in $\L^2(\mu)$ if and only if $A_2$ is self-adjoint in $\L^2(\mu)$. The latter holds if and only if $\Pi=0$ and $\Sigma B^*=B\Sigma$.
		  \end{enumerate}
		\end{theorem}

		\subsection{Spectral expansion: existence and nonexistence}\label{sec:spectral_exp}
		
		In the diffusion case, that is, when $\Pi= 0$, H\ref{assumption_polynomial} holds trivially. The next result provides a spectral expansion formula for the diffusion OU semigroup in terms of the eigenfunctions and co-eigenfunctions. Due to Theorem~\ref{thm:normality}, these semigroups are, in general, non-normal in $\L^2(\mu)$.
		\begin{theorem}\label{thm:spectral_exp} Assume H\ref{non-degeneracy}, $\Pi= 0$ and $B$ is diagonalizable. Then for all $f\in\L^2(\mu)$,
			\begin{align}\label{eq:spectral_expansion}
				P_t f=\sum_{n\in\bb{N}^d_0} e^{-t\langle n,\lambda\rangle}\langle f, \cc{G}_n\rangle_{\L^2(\mu)} \cc{H}_n \quad \mbox{for all $t>t_0$},
			\end{align}
			where $t_0=\log(2d\|M\Sigma_\infty M^*\|)/\Re(\lambda_1)$, where $\Re(\lambda_1)\le\cdots\le \Re(\lambda_d)$, and $M$ is defined in \eqref{eq:diagonalization}.
		\end{theorem}
	\begin{remark} 
	We note that the above spectral expansion holds for $t>t_0$. This is due to the non-normality of $P_t$. Unlike normal operators, the biorthogonal sequence of eigenfunctions and co-eigenfunctions may not be uniformly bounded in $\L^2$-norm. We will show that $\mathcal{G}_n$ are uniformly bounded in $\L^2(\mu)$-norm while $\|\mathcal{H}_n\|_{\L^2(\mu)}\le e^{t_0\Re(\lambda_1)|n|}$. As a result, the right-hand side of \eqref{eq:spectral_expansion} is convergent when $t>t_0$.
	\end{remark}

		From Theorem~\ref{thm:spectral_exp}, we obtain the following infinite series representation for the transition density of $P$, providing a generalization of Mehler's formula, see \cite[Theorem~3.1]{ChenLiu2014}.
		\begin{theorem}\label{thm:heat_kernel}
			Suppose that the conditions of Theorem~\ref{thm:spectral_exp} hold, and let $p_t(\cdot, \cdot)$ be the transition density of $P$, that is, \begin{align*}P_t f(x)=\int_{\R^d} f(y)p_t(x,y)\mu(y)dy \quad \mbox{for all $f\in B_b(\R^d)$}.\end{align*} Then, there exists $t_0>0$ such that for all $t> t_0$ and $(x,y)\in\R^d\times\R^d$,
			\begin{align*}
				p_t(x,y)=\sum_{n\in\bb{N}^d_0} e^{-t\langle n,\lambda\rangle}\cc{H}_n(x)\cc{G}_n(y).
			\end{align*}
		\end{theorem}
	\begin{corollary}[Generalized Mehler's formula]\label{cor:Mehler} Suppose that the conditions of Theorem~\ref{thm:spectral_exp} hold. Then, there exists $t_0>0$ such that for any $x,y\in\R^d$ and $t>t_0$,
		\begin{align*}
			&\sum_{n\in\bb{N}^d_0} e^{-t\langle n,\lambda\rangle}\cc{H}_n(x)\cc{G}_n(y)\\
			&=\sqrt{\frac{\det(\Sigma_\infty)}{\det(\Sigma_t)}}
			\times \exp\left(\frac{1}{2}\langle \Sigma^{-1}_\infty y,y\rangle-\frac12\langle \Sigma^{-1}_t(e^{tB}x-y), (e^{tB}x-y)\rangle\right).
		\end{align*}
		
	\end{corollary}
	
	When $\Pi\neq 0$, we observe a very different behavior regarding the spectral expansion of the semigroup. This happens as the norm of the eigenfunctions $\mathcal{H}_n$ can grow super-exponentially.
	\begin{theorem}[Nonexistence of spectral expansion]\label{thm:non-existence}
		Consider the one-dimensional L\'evy-OU operator 
		\begin{align*}
			Af(x)=\frac{\sigma^2}{2}f''(x)-bxf'(x)+\int_0^\infty[f(x+y)-f(x)-\mathbbm{1}_{\{y\le 1\}} yf'(x)]\Pi(dy)
		\end{align*}
		where $\sigma^2, b>0$ and $\Pi$ is a L\'evy measure on $\R$. Assume that $\Pi$ satisfies H\ref{assumption_polynomial}, and 
		\begin{align*}
		\inf\mathrm{Supp}(\Pi)>0.
		\end{align*}
		Then, for any $t>0$, there exists $f\in \L^2(\mu)$ such that
		\begin{align*}
			\sum_{n=0}^\infty e^{-tbn}\langle f,\mathcal{G}_n\rangle_{\L^2(\mu)}\mathcal{H}_n
		\end{align*}
		does not converge. 
	\end{theorem}
	\begin{remark}
		By Proposition~\ref{prop:compactness} in Example~\ref{ex:compact} below, when $\Pi$ is compactly supported, the L\'evy-OU semigroup generated by the above operator is compact in $\L^2(\mu)$ with $\sigma(A_2)=\sigma_p(A_2)=\{-bn:n\in\mathbb{N}_0\}$. However, due to Theorem~\ref{thm:non-existence}, the semigroup does not admit a spectral expansion.
	\end{remark}
	
	\subsection{Compactness of L\'evy-OU semigroups}\label{sec:compactness} It is known that the diffusion OU semigroup $(P_t)_{t\ge 0}$ is compact in $\L^p(\mu)$ for all $1<p<\infty$; see \cite[p.~45]{MetafunePallaraPriola2002} for the proof in finite-dimensional case, and \cite{Michalik1987} for the infinite-dimensional case. In this section, we investigate the compactness of $P_t$ when the L\'evy measure is nonzero. To our surprise, it heavily depends on the existence of moments of the invariant distribution. In the following theorem we show that H\ref{assumption_polynomial} is a necessary condition for compactness of $P_t$.
	
	\begin{theorem}[Necessary condition]\label{thm:nec_compact}
		If $P_t: \L^p(\mu)\longrightarrow \L^p(\mu)$ is compact for some $t>0$ and for some $1<p<\infty$, then H\ref{assumption_polynomial} holds. 
	\end{theorem}
	A L\'evy process with L\'evy-Khintchine exponent $\Psi$ is called $\alpha$-stable if for all $\xi\in\R^d$ and $b>0$, $\Psi(b\xi) = b^\alpha \Psi(\xi)$, see \cite[Definition~13.6]{SatoBook1999}. Let $\Pi$ be the L\'evy measure of an $\alpha$-stable L\'evy process in $\R^d$. By \cite[Theorem~14.2(3)]{SatoBook1999} one can write 
	\begin{align}\label{eq:Pi_alpha}
		\Pi(E)=\int_{\bb S^{d-1}}\sigma(d\xi)\int_0^\infty\mathbbm{1}_{E}(r\xi) r^{-1-\alpha} dr, \quad \forall E\in\cc{B}(\R^d)
	\end{align}
	for some finite measure $\sigma$ on $\bb{S}^{d-1}$, where $\mathbb{S}^{d-1}$ denotes the $(d-1)$-dimensional unit sphere. Moreover, $\sigma$ is uniquely determined and for any $S\in\cc{B}(\bb{S}^{d-1})$, $\sigma(S)=\alpha\Pi_\alpha((1,\infty) S)$.
	By \eqref{eq:Pi_alpha}, $\alpha$-stable L\'evy measures satisfy H\ref{ergodicity}, and therefore the corresponding L\'evy-OU semigroups are ergodic with a unique invariant distribution $\mu$.
	\begin{corollary}\label{prop:non-compactness1}
		Let $P$ be a L\'evy-OU semigroup with $\alpha$-stable L\'evy measure where $\alpha\in (0,2)$. Then $P_t:\L^p(\mu)\longrightarrow \L^p(\mu)$ is not compact for any  $1<p<\infty$, and $t>0$. In particular, the semigroup generated by 
		\begin{align*}
			-(-\Delta)^{\frac{\alpha}{2}}-\langle Bx,\nabla\rangle
		\end{align*}
		is not compact in $\L^p(\mu)$ for any $\alpha\in (0,2)$ and $p\in (1,\infty)$, where $\mu$ is the invariant distribution.
	\end{corollary}
	\begin{remark}
		In the case $d=1$, the above result was proved in \cite[Theorem~7.3]{Peszat} for L\'evy-OU semigroups with symmetric $\alpha$-stable L\'evy measure using second quantized operator representation of the semigroup.
	\end{remark}
	
	We end this section by providing a sufficient condition on the invariant distribution $\mu$ which ensures compactness of the semigroup $P$.
	\begin{theorem}[Sufficient condition]\label{thm:compactness}
			Assume that H\ref{non-degeneracy} holds and let $\mu$ be the invariant distribution of $P=(P_t)_{t\ge 0}$ and $W=-\log \mu$.  If 
			\begin{align}\label{eq:W}
				\lim_{|x|\to\infty} |\nabla W(x)|=\infty,
			\end{align}
		then $P_t:\L^p(\mu)\longrightarrow \L^p(\mu)$ is compact for all $1<p<\infty$ and $t>0$.
		\end{theorem}
		
		\begin{remark}
			\begin{enumerate}[leftmargin=*]
			\item In the proof of Lemma~\ref{lem:compactness}, we show that $\Delta W(x)\le \tr(\Sigma^{-1}_\infty)$ for any $x\in\R^d$, and as a consequence, by \cite[Theorem~8.5.3]{Lorenzi2007}, the condition \eqref{eq:W} is also sufficient for compactness of the self-adjoint Fokker-Planck semigroup generated by
			\begin{align*}
				A_W = \Delta -\nabla W\cdot \nabla
			\end{align*}
			in $\L^2(\mu)$.
			
			\item If $B=-\mathrm{Id}$, note that by Theorem~\ref{thm:spectral_expansion}, $\partial_{x_i} W$ is an eigenfunction of $A^*_2$ corresponding to the eigenvalue $-1$. It is surprising to see that compactness of $P_t$ follows from the unboundedness of the co-eigenfunction.
			\end{enumerate}
		\end{remark}
		\noindent
		\textbf{Open question:} Is there a necessary and sufficient condition in terms of the limiting distribution $\mu$ for the compactness of L\'evy-OU semigroup?
		
		\changelocaltocdepth{1}
		
		\section{Examples}\label{sec:examples}
		\begin{example}[Linear kinetic Fokker-Planck equation on $\R^d\times\R^d$] Consider the linear kinetic Fokker-Planck semigroup $P_t$ with generator 
			\begin{align*}
				A=\frac{1}{2}\Delta_x-\langle x, \nabla_y\rangle-\langle x,\nabla_x\rangle+\gamma \langle y, \nabla_x\rangle, \quad \gamma>0.
			\end{align*}
			In this case, 
			\begin{align}\label{eq:Q_B}
			\Sigma=\begin{pmatrix}
				I_d & 0 \\ 0 & 0
			\end{pmatrix}, \quad B=\begin{pmatrix} -I_d & \gamma I_d \\ -I_d & 0
			\end{pmatrix}.
		\end{align}
	The minimal polynomial of $B$ is given by $p_B(\lambda)=\lambda^2+\lambda+\gamma$. Whenever $\gamma\neq 1/4$, $p_B(\lambda)$ has distinct roots. Therefore, $B$ is diagonalizable with eigenvalues having strictly negative real part whenever $\gamma\in (0,\infty)\setminus\{1/4\}$. The invariant distribution is given by 
	\begin{align*}
		\mu(dx)=\frac{\gamma^{\frac d2}}{(2\pi)^d}e^{-\frac{|x|^2}{2}-\frac{\gamma}{2} |y|^2}.
	\end{align*}
Therefore, the assertions of Theorems \ref{thm:spectrum}, \ref{it:3}, \ref{thm:spectral_expansion}, \ref{thm:biorthogonality}, \ref{thm:spectral_exp}, \ref{thm:heat_kernel} and Corollary~\ref{cor:Mehler} hold.
		\end{example}
	\begin{example}[Compact L\'evy-OU semigroups on $\mathbb{R}$]\label{ex:compact}
	Let $A$ be the generator of a L\'evy-OU process on $\R$ such that 
	\begin{align}\label{eq:A2}
		A u(x)=\frac{\sigma^2}{2} u''(x)-bxu'(x)+\lambda\int_0^\infty [u(x+y)-u(x)]\Pi(dy),
	\end{align}
where $\lambda>0$, and $\Pi$ is a probability measure supported on $(0,c)$ for some $0<c<\infty$. This means that the jumps of the L\'evy-OU process are bounded and distributed according to a compound Poisson process. Clearly, $\Pi$ satisfies H\ref{assumption3}. Also, by \cite[Example~17.10]{SatoBook1999}, the L\'evy-Khintchine exponent of the invariant distribution $\mu$ is given by 
		\begin{align*}
			\log \cc{F}_\mu(\xi)=-\frac{\sigma^2}{4b}\xi^2+\frac{\lambda}{b}\int_{0}^c (e^{\i\xi y}-1)\frac{\Pi(y,\infty)}{y} dy.
		\end{align*}
	In particular, when $\Pi=\delta_1$, the Dirac measure with mass at $1$, the jumps are distributed according to Poisson process with rate $\lambda$.
	\begin{proposition}\label{prop:compactness}
		Let $(P_t)_{t\ge 0}$ be the semigroup generated by $A$ in \eqref{eq:A2}. Then, $P_t:\L^p(\mu)\longrightarrow \L^p(\mu)$ is compact for all $1<p<\infty$.
	\end{proposition}
\begin{proof}
		 By \cite[Theorem~5.2, Equation (5.19)]{KnopovaKulik2011}, for any $T>0$ and $0<c_1<1/c<c_2$, there exists $y(c_1, c_2, T)>0$ such that for all $x>y(c_1, c_2, T)$,
		\begin{align*}
			x^{-c_2 h}\le \frac{\mu(x+h)}{\mu(x)}\le x^{-c_1 h}
		\end{align*}
	uniformly for all $h\in [0,T]$.
	Therefore, using triangle inequality we get 
	\begin{align*}
		\left|\frac{\mu'(x)}{\mu(x)}\right|&=\lim_{h\downarrow 0}\left|\frac{\mu(x+h)-\mu(x)}{h\mu(x)}\right| \\
		&\ge \lim_{h\downarrow 0}\frac{1-x^{-c_1 h}}{h}=c_1\log x.
	\end{align*}
This shows that $\lim_{x\to\infty} |\mu'(x)/\mu(x)|=\infty$. On the other hand, since
\begin{align*}
	\mu(x)=\frac{\sqrt{2b}}{\sqrt{\pi}}\int_{0}^\infty e^{-\frac{2b(x-y)^2}{\sigma^2}} \mu_J(dy)
\end{align*}
with $\cc{F}_{\mu_J}(\xi)=\exp\left(\frac{\lambda}{b}\int_{0}^c (e^{\i\xi y}-1)\frac{\Pi(y,\infty)}{y} dy\right)$,
we also have 
\begin{align*}
	-\frac{\mu'(x)}{\mu(x)}=\frac{4b}{\sigma^2}x-\frac{4b}{\sigma^2}\frac{\int_0^\infty y e^{-\frac{(x-y)^2}{2}}\mu_J(dy)}{\mu(x)}.
\end{align*}
This shows that $\lim_{x\to -\infty}-\frac{\mu'(x)}{\mu(x)}=-\infty$. Therefore, $\mu$ satisfies \eqref{eq:W}, and hence by Theorem~\ref{thm:compactness}, $P_t:\L^p(\mu)\longrightarrow \L^p(\mu)$ is compact for all $1<p<\infty$.
\end{proof}
 As a result, we obtain the following corollary.
\begin{corollary}
	Let $A$ be defined by \eqref{eq:A2} and $(P_t)_{t\ge 0}$ be the semigroup generated by $A$. Then, for all $1<p<\infty$, $\sigma(A_p) =\sigma_p(A_p)= -b\bb{N}_0$. Also, for any $n\in\bb{N}_0$, $\mathtt{M}_a(-bn, A_p)=\mathtt{M}_g(-bn, A_p)=1$.
\end{corollary}
\begin{proof}
	Since $P_t:\L^p(\mu)\longrightarrow \L^p(\mu)$ is compact for all $1<p<\infty$, the first assertion follows immediately from Theorem~\ref{thm:spectrum}. The second assertion follows from Theorem~\ref{thm:multiplicity}.
\end{proof}
	\end{example}
	
\section{Preliminaries on L\'evy-Ornstein-Uhlenbeck semigroups}\label{sec:preliminaries_levy} A L\'evy process is a stochastic process having stationary and independent increments. If $Z=(Z_t)_{t\ge 0}$ is a L\'evy process on $\R^d$, its distribution is infinitely divisible for each $t>0$, and the distribution of the process is completely determined by its \emph{L\'evy-Khintchine exponent}, which is defined as
\begin{equation}\label{eq:Psi}
	\begin{aligned}
		\Psi(\xi)&=\frac{1}{t}\log \mathbb{E}\left[e^{\i\langle \xi, Z_t\rangle}\right] \\
		& = -\frac{1}{2}\langle \Sigma\xi,\xi\rangle+\int_{\R^d}(e^{\i\langle \xi, y\rangle}-1-\i\langle\xi,y\rangle\mathbbm{1}_{\{|y|\le 1\}})\Pi(dy) \\
		&=:-\frac{1}{2}\langle \Sigma\xi,\xi\rangle+\Phi(\xi).
	\end{aligned}
\end{equation}
Recall the L\'evy-OU operator $A$ defined in \eqref{eq:OU_gen}. Then, $(A, C^\infty_c(\R^d))$ generates a strongly continuous contraction semigroup $P=(P_t)_{t\ge 0}$, known as the L\'evy-OU semigroup. From \cite[Lemma~17.1]{SatoBook1999} (see also \cite{LescotRockner2002}), it follows that $P_t$ is a pseudo-differential operator such that
\begin{align}\label{eq:characteristic_fn}
	P_t\exi (x)=\exp\left(\Psi_t(\xi)+\i\langle e^{tB^*}\xi,x\rangle\right),
\end{align}
where $\e_z(x)=e^{\langle z,x\rangle}$ for any $z\in\bb{C}^d$, and $\Psi_t(\xi)=-\frac{1}{2}\langle \Sigma_t\xi,\xi\rangle+\Phi_t(\xi)$ with 
\begin{equation}\label{eq:Psi_t}
	\begin{aligned}
		\Phi_t(\xi)&=\int_0^t \Phi(e^{sB^*}\xi) ds \\
		&=\int_0^t\int_{\bb{R}^d}\left[\exi(e^{sB}y)-1-\i\langle\xi,e^{sB}y\rangle\bbm{1}_{\{|y|\le 1\}}\right]\Pi(dy), \\
		\Sigma_t&=\int_0^t e^{sB} \Sigma e^{sB^*} ds. 
	\end{aligned}
\end{equation}
As a result, for any $t>0$, 
\begin{align}\label{eq:psi_t}
	\Psi_t(\xi)=\int_0^t \Psi(e^{sB^*}\xi) ds.
\end{align}
For each $t>0$, we observe that $\Psi_t$ is a L\'evy-Khintchine exponent with diffusion matrix $\Sigma_t$, and its associated L\'evy measure is given by $\Pi_t:=\int_0^t \Pi\circ e^{-sB} ds$.
Throughout the paper, it is assumed that all eigenvalues of $B$ have strictly negative real part and the L\'evy measure $\Pi$ satisfies the log-moment condition in \eqref{eq:log-moment-cond}.
Under these assumptions the semigroup $P=(P_t)_{t\ge 0}$ has a unique invariant distribution, as shown in \cite{SatoYazamato1984}. Moreover, the above assumptions are also necessary for the existence of an invariant distribution, see \cite[Theorem~17.11]{SatoBook1999}. Denoting the invariant distribution of $P$ by $\mu$, \eqref{eq:characteristic_fn} implies that
\begin{equation}\label{eq:invariant_distribution}
	\begin{aligned}
		&\int_{\bb{R}^d}e^{\i\langle\xi,x\rangle}\mu(dx)=e^{\Psi_\infty(\xi)}, \\
		&\Psi_\infty(\xi):=\lim_{t\to\infty}\Psi_t(\xi)=-\frac{1}{2}\langle \Sigma_\infty\xi,\xi\rangle+\Phi_\infty(\xi).
	\end{aligned}
\end{equation}
Therefore, $\mu$ is also infinitely divisible with the L\'evy measure 
\begin{align}\label{eq:Pi_infty}
	\Pi_\infty:=\int_0^\infty \Pi\circ e^{-sB} ds.
\end{align}
When H\ref{non-degeneracy} holds, $\Sigma_\infty$ is a positive definite matrix, and in this case $\mu$ is absolutely continuous with a smooth, positive density. For convenience, we abuse notation by also denoting this density by $\mu$. Due to the existence of an invariant distribution, $P$ can be extended uniquely as a strongly continuous contraction semigroup on the Banach space $\L^p(\mu)$ for all $p\ge 1$. We end this section with the following two lemmas that will be useful in the proofs of the main results.

\begin{lemma}\label{lem:Psi_t}
	For any $t>0$ and $\xi\in\R^d$, 
	\begin{align*}
		\Psi_\infty(e^{tB^*}\xi)=\Psi_\infty(\xi)-\Psi_t(\xi).
	\end{align*}
\end{lemma}
Proof of the above lemma follows by a straightforward computation. The next lemma relates H\ref{assumption_polynomial} and H\ref{assumption3} to the moments of the invariant distribution $\mu$.\footnote{H\ref{assumption3} is defined in \S\ref{sec:exp_moment}.}

\begin{lemma}\label{lem:moment}
	H\ref{assumption_polynomial} holds if and only if $\int_{\R^d} |x|^n\mu(dx)<\infty$ for all $n\ge 1$. If H\ref{assumption3} holds, then $\int_{\R^d} e^{\kappa' |x|}\mu(dx)<\infty$ for some $0<\kappa'<\kappa$.
\end{lemma}
\begin{proof}
	We observe that $\int_{\R^d}|x|^n\mu(dx)<\infty$ if and only if $\int_{|x|>2}|x|^n \mu(dx)<\infty$. Moreover, from \eqref{eq:Pi_infty}, $\mu$ is infinitely divisible with the L\'evy measure 
	\begin{align*}
		\Pi_\infty=\int_0^\infty \Pi\circ e^{-sB} ds.
	\end{align*}
	Since the function $g(x)=1+|x|^n$ is sub-multiplicative, that is, $g(x+y)\le c g(x) g(y)$ for all $x,y\in\R^d$ and for some constant $c>0$, by \cite[Theorem~25.3]{SatoBook1999} it follows that the $n^{th}$ moment of $\mu$ exists if and only if $\int_{\{|x|>1\}}|x|^n\Pi_\infty(dx)<\infty$. 
	Since $\sigma(B)\subset \bb{C}_-$, for all $\epsilon>0$ with $s(B)+\epsilon<0$ we have 
	\begin{align*}
	\|e^{tB}\|\le C_\epsilon e^{t(s(B)+\epsilon)} \quad \mbox{for all $t\ge 0$},
	\end{align*}
	 where $s(B)=\max\{\Re(-\lambda_i):i=1,\ldots, d\}$, and $C_\epsilon\ge 1$ is a constant. Moreover, there exists a constant $c>0$ such that 
	 \begin{align*}
	 	\|e^{tB}\|\le c \quad \mbox{for all $t\ge 0$}.
	 \end{align*}
	  As a result, 
	\begin{equation}\label{eq:moment_calculation}
		\begin{aligned}
			\int_{\{|x|>1\}} |x|^n \Pi_\infty(dx)&=\int_0^\infty \int_{\{|e^{tB}x|>1\}} |e^{tB} x|^n\Pi(dx) dt\\
			&\le C^n_\epsilon\int_0^\infty e^{(s(B)+\epsilon) tn}\int_{\{|e^{tB}x|>1\}}|x|^n \Pi(dx) dt \\
			&\le -C^n_\epsilon\frac{(s(B)+\epsilon)^{-1}}{n}\int_{\{|x|>1/c\}}|x|^n \Pi(dx),
		\end{aligned}
	\end{equation}
	showing that
	\begin{align*}
		 \int_{\{|x|>1\}}|x|^n\Pi(dx)<\infty \implies \int_{\{|x|>1\}}|x|^n\Pi_\infty(dx)<\infty.
	\end{align*}
	
	To prove the reverse statement, we note that for sufficiently small $\delta>0$, 
	\begin{align*}
		|e^{tB}x|\ge \frac{|x|}{2} \quad \mbox{for all $x\in\R^d$, $t\in [0,\delta]$}
	\end{align*}
	As a result,
	\begin{align*}
		\int_{|x|>1}|x|^n\Pi_\infty(dx)&=\int_0^\infty \int_{\{|e^{tB}x|>1\}} |e^{tB} x|^n\Pi(dx) dt \\
		&\ge 2^{-n}\int_0^\delta \int_{\{|x|>2\}} |x|^n\Pi(dx) dt \\
		&\ge \delta 2^{-n} \int_{\{|x|>2\}} |x|^n\Pi(dx).
	\end{align*}
	Therefore, 
	\begin{align*}
		\int_{\{|x|>1\}}|x|^n\Pi_\infty(dx)<\infty \implies \int_{\{|x|>1\}}|x|^n\Pi(dx)<\infty.
	\end{align*}
	
	Similarly, since $x\mapsto e^{\kappa|x|}$ is a sub-multiplicative function, $\int_{\R^d} e^{\kappa|x|}\mu(dx)<\infty$ if and only if $\int_{\{|x|>1\}} e^{\kappa |x|} \Pi_\infty(dx)<\infty$. Let us write 
	\begin{align*}
		\kappa_\epsilon=\frac{\kappa}{C_\epsilon}.
	\end{align*}
	Therefore, \eqref{eq:moment_calculation} implies that 
	\begin{align*}
		\int_{\{|x|>1\}} e^{\kappa_\epsilon |x|}\Pi_\infty(dx)&=\int_{\{|x|>1\}} \Pi_\infty(dx)+\sum_{n=1}^\infty\frac{1}{n!} \int_{\{|x|>1/c\}} \kappa^n_\epsilon|x|^n \Pi_\infty(dx) \\
		&=\int_{\{|x|>1\}} \Pi_\infty(dx)-\sum_{n=1}^\infty \frac{C^n_\epsilon}{n!}\frac{(s(B)+\epsilon)^{-1}}{n}\int_{\{|x|>1/c\}} \kappa^n_\epsilon|x|^n \Pi(dx) \\
		&\le\int_{\{|x|>1\}} \Pi_\infty(dx)-(s(B)+\epsilon)^{-1}\int_{\{|x|>1/c\}} e^{\kappa |x|}\Pi(dx)<\infty
	\end{align*}
	This completes the proof of the lemma.
\end{proof}

		\section{Intertwining relations}\label{sec:intertwining}  For two closed operators $(A,\Dom(A)), (B,\Dom(B))$ defined on Banach spaces $\cc{X}$ and $\cc{Y}$ respectively, we say that $A$ and $B$ are \emph{intertwined} if there exists a bounded linear operator $\Lambda:\cc{Y}\longrightarrow\cc{X}$ such that 
			$\Lambda(\Dom(B))\subseteq\Dom(A)$ and 
			\begin{align}\label{eq:intertwining}
				A\Lambda=\Lambda B \ \text{ on } \ \Dom(B).
			\end{align}
		Let $\Sigma,\wt{\Sigma}$ be a nonnegative definite matrix such that $\Sigma,\wt{\Sigma}$ satisfy H\ref{non-degeneracy}, and $\Sigma_\infty\succeq \wt{\Sigma}_\infty$, that is, $\Sigma_\infty-\wt{\Sigma}_\infty$ is nonnegative definite, where $\Sigma_t$ is defined in \eqref{eq:Q_t}. For any $f\in B_b(\R^d)$, let us now define
		\begin{align}\label{eq:Lambda}
			\Lambda f(x)=\int_{\R^d} f(x+y)h_\Lambda(dy),
		\end{align}
		where $h_\Lambda$ is a probability measure satisfying 
		\begin{align}\label{eq:h_infty}
			\cc{F}_{h_{\Lambda}}(\xi)=e^{-\frac{1}{2}\langle(\Sigma_\infty-\wt{\Sigma}_\infty)\xi,\xi\rangle+\Phi_\infty(\xi)} \quad \forall \xi\in\R^d
		\end{align}
	with $\Phi_\infty$ being defined by \eqref{eq:invariant_distribution}.
		Since $h_\Lambda$ is a probability measure, $\Lambda$ is a Markov operator, that is, $\Lambda 1=1$ and $\Lambda f\ge 0$ whenever $f\ge 0$. Also, when $\Sigma_\infty=\wt{\Sigma}_\infty$ and $\Phi_\infty\equiv 0$, we have $h_\Lambda=\delta_0$ and therefore, $\Lambda=\mathrm{Id}$.
		\begin{proposition}\label{prop:intertwining_diffusion}
			For any nonnegative measurable function $f$,
			\begin{align*}
				\int_{\R^d}\Lambda f(x) \wt{\mu}(dx)=\int_{\R^d} f(x) \mu(dx), 
			\end{align*}
			where  
			\begin{align*}
				\wt{\mu}(dx)=\frac{1}{(2\pi)^{\frac d2}\sqrt{\det(\wt{\Sigma}_\infty)}} e^{-\frac{1}{2}|\wt{\Sigma}^{-1/2}_\infty x|^2} dx.
			\end{align*}
			As a result, for all $1\le p\le \infty$, $\Lambda$ extends as a bounded operator $\Lambda:\L^p(\mu)\longrightarrow\L^p(\wt{\mu})$ with dense range. 
		\end{proposition}
		\begin{proof}
			Let $\Lambda$ and $h_\Lambda$ be defined as in \eqref{eq:Lambda} and \eqref{eq:h_infty}. To prove the first assertion, for any nonnegative measurable function $f$, using Fubini's theorem we have 
			\begin{align*}
				\int_{\R^d}\Lambda f(x) \wt{\mu}(x)dx&=\int_{\R^d}\left(\int_{\R^d} f(x+y)\wt{\mu}(x) dx\right)h_\Lambda(dy) \\
				&=\int_{\R^d} f(x)\int_{\R^d} \wt{\mu}(x-y)h_\Lambda(dy) dx\\
				&=\int_{\R^d} f(x) \mu(x) dx,
			\end{align*}
			where the last identity follows from the fact that $\wt{\mu} * h_\Lambda=\mu$. As a result, for any $f\in\L^p(\mu)$ and $p\ge 1$, using Jensen's inequality we get
			\begin{align*}
				\int_{\R^d} (\Lambda f)^p(x)\widetilde{\mu}(x) dx\le \int_{\R^d} \Lambda (|f|^p)(x) \mu(x) dx=\int_{\R^d}|f(x)|^p \mu(x) dx,
			\end{align*}
			which shows that $\Lambda$ extends as a bounded operator $\Lambda:\L^p(\mu)\longrightarrow \L^p(\wt{\mu})$. Now, for any $\xi\in\R^d$, let us write $\cc{E}_{\i\xi}(x)=e^{\i\langle\xi, x\rangle}$. Then, from the definition of $\Lambda$ in \eqref{eq:Lambda} it follows that $\Lambda\cc{E}_{\i\xi}=\cc{F}_{h_\Lambda}(\xi)\cc{E}_{\i\xi}$ for all $\xi\in\R^d$. Therefore, $\Span\{\cc{E}_{\i\xi}:\xi\in\R^d\}\subset\Range(\Lambda)$. Since the former subset is dense in $\L^p(\mu)$ for any $1\le p\le \infty$, we conclude that $\Lambda:\L^p(\mu)\longrightarrow \L^p(\wt{\mu})$ has dense range for all $1\le p\le \infty$.
		\end{proof}
		We are now ready to state the main result of this section.
		\begin{proposition}\label{prop:left_int}
			For all $t\ge 0$ and $p\ge 1$ we have,
			\begin{align}\label{eq:intertwining_p}
				\widetilde{Q}_t\Lambda=\Lambda P_t \ \ \mbox{on \ $\L^p(\mu)$},
			\end{align}
			where $\widetilde{Q}$ is the diffusion OU semigroup generated by $\wt{A}=\frac{1}{2}\tr(\wt{\Sigma}\nabla^2)+\langle Bx, \nabla\rangle$ with $\Sigma_\infty\succeq\wt{\Sigma}_\infty$, and $\Lambda$ defined in \eqref{eq:Lambda}.
			When $\wt{\Sigma}=0$, for all $t\ge 0$ we have 
			\begin{align}\label{eq:intertwining_drfit}
				e^{tL}\Lambda_1=\Lambda_1 P_t \quad \mbox{on \ $\L^\infty(\mu)$},
			\end{align}
			where $L=\langle Bx,\nabla\rangle$, and $\Lambda_1 f(x)=\int_{\R^d} f(x+y)\mu(dy)$.
			Additionally, under H\ref{assumption_polynomial}, \eqref{eq:intertwining_drfit} holds on $\mathscr{P}$, the space of all polynomials on $\R^d$.
		\end{proposition} 
		\begin{proof}
			We start with the observation that for all $\xi\in\R^d$,
			\begin{align*}
				\Lambda\e_{\i\xi}=e^{-\frac{1}{2}\langle(\Sigma_\infty-\wt{\Sigma}_\infty)\xi,\xi\rangle+\Phi_\infty(\xi)} \e_{\i\xi},
			\end{align*}
			where $\e_{\i\xi}(x)=e^{\i\langle\xi, x\rangle}$ for all $x\in\R^d$. From \eqref{eq:characteristic_fn} we therefore obtain that for all $t>0$ and $\xi\in\R^d$, 
			\begin{align*}
				\wt{Q}_t\Lambda\e_{\i\xi}=\exp\left(-\frac{1}{2}\langle(\Sigma_\infty-\wt{\Sigma}_\infty)\xi,\xi\rangle+\Phi_\infty(\xi)-\frac{1}{2}\langle \wt{\Sigma}_t\xi,\xi\rangle\right)\e_{\i e^{tB^*}\xi}.
			\end{align*}
			On the other hand, using \eqref{eq:characteristic_fn} again, we also have 
			\begin{align*}
				\Lambda P_t \e_{\i\xi}&=e^{\Psi_t(\xi)}\Lambda\e_{\i e^{tB^*}\xi}\\
				&=\exp\left(\Psi_t(\xi)-\frac{1}{2}\langle e^{tB}(\Sigma_\infty-\wt{\Sigma}_\infty) e^{tB^*}\xi, \xi\rangle+\Phi_\infty(e^{tB^*}\xi)\right)\e_{\i e^{tB^*}\xi}.
			\end{align*}
			A straightforward computation shows that for all $\xi\in\R^d$,
			\begin{align*}
				&\Psi_t(\xi)-\frac{1}{2}\langle e^{tB}(\Sigma_\infty-\wt{\Sigma}_\infty) e^{tB^*}\xi, \xi\rangle+\Phi_\infty(e^{tB^*}\xi)\\
				&=-\frac{1}{2}\langle(\Sigma_\infty-\wt{\Sigma}_\infty)\xi,\xi\rangle+\Phi_\infty(\xi)-\frac{1}{2}\langle \wt{\Sigma}_t\xi,\xi\rangle,
			\end{align*}
			which implies that for all $\xi\in\R^d$,
			\begin{align*}
				\wt{Q}_t\Lambda\e_{\i\xi}=\Lambda P_t\e_{\i\xi}.
			\end{align*}
			Since $\{\e_{\i\xi}\}_{\xi\in\R^d}$ is dense in $\L^p(\mu)$, we conclude that \eqref{eq:intertwining_p} holds on $\L^p(\mu)$. When $\wt{\Sigma}=0$, to prove \eqref{eq:intertwining_drfit} we proceed as before. For any $\xi\in\R^d$, one has 
			\begin{align*}
				e^{tL}\Lambda_1\e_{\i\xi}=e^{\Psi_\infty(\xi)}e^{tL}\e_{\i\xi}=e^{\Psi_\infty(\xi)} \e_{\i e^{tB^*}\xi},
			\end{align*} 
			while 
			\begin{align*}
				\Lambda_1 P_t\e_{\i\xi}=e^{\Psi_t(\xi)}\Lambda_1 \e_{\i e^{tB^*}\xi}=e^{\Psi_t(\xi)} e^{\Psi_\infty(e^{tB^*}\xi)}\e_{\i e^{tB^*}\xi}.
			\end{align*}
			Since $\Psi_\infty(\xi)=\Psi_t(\xi)+\Psi_\infty(e^{tB^*}\xi)$ for all $\xi\in\R^d$, and $(\e_{\i\xi})_{\xi\in\R^d}$ is dense in $\L^\infty(\mu)$, the proof of \eqref{eq:intertwining_drfit} follows. When H\ref{assumption_polynomial} holds, due to \cite[Theorem~25.3]{SatoBook1999}, we first note that $P_t |f|(x)=\bb{E}_x(|f(X_t)|)<\infty$ for all $f\in\mathscr{P}$ and $x\in\R^d$. For any $f\in\mathscr{P}$, we consider a sequence of bounded measurable functions $(f_n)$ such that $|f_n|\le |f|$ for all $n$ and $f_n\to f$ pointwise. Since $\Lambda_1$ is a Markov operator, we have $|\Lambda_1 f_n|\le \Lambda_1 |f_n|\le\Lambda_1 |f|$, and $P_t\Lambda_1 |f|(x)<\infty$ for all $x\in\R^d$. Also, for all $n\ge 1$ we have
			\begin{align*}
				e^{tL}\Lambda_1 f_n=\Lambda_1 P_t f_n.
			\end{align*}
			Finally, letting $n\to\infty$ and using dominated convergence theorem, we conclude that $e^{tL}\Lambda_1 f=\Lambda_1 P_t f$ for all $f\in\mathscr{P}$. This completes the proof of the proposition.
		\end{proof}
		\subsection{Intertwining with a normal diffusion operator}\label{sec:intertwining-normal} In this section we assume that the diffusion matrix $\Sigma$ satisfies H\ref{non-degeneracy}, and the drift matrix $B$ is diagonalizable with eigenvalues $-\lambda_1,\ldots, -\lambda_d$ having strictly negative real part. We recall the diffusion OU operator defined in \eqref{eq:A-OU} which generates an ergodic Gaussian Markov process with invariant distribution $\nu$ given by 
		\begin{align*}
			\nu(dx)=\frac{(2\pi)^{-\frac d2}}{\sqrt{\det(\Sigma_\infty)}} e^{-\frac{1}{2}\langle \Sigma^{-1}_\infty x,x\rangle} dx, \quad x\in\R^d.
		\end{align*}
		Let $M, B_0$ be the matrices defined in \eqref{eq:diagonalization}.
		Using the similarity transform $S_M f(x)=f(Mx)$ for measurable functions $f$, it is easy to verify that 
		\begin{align}\label{eq:similarity}
			S^{-1}_M A^{\mathrm{OU}} S_M= A^{\mathrm{OU}}_M,
		\end{align}
		where $A^{\mathrm{OU}}_M$ is another diffusion OU operator given below:
		\begin{align}\label{eq:gen_M}
			A^{\mathrm{OU}}_M=\frac{1}{2}\tr(M\Sigma M^*\nabla^2)+\langle B_0x, \nabla\rangle
		\end{align}
		We note that $(M\Sigma M^*,B_0)$ satisfies H\ref{non-degeneracy}, that is, 
		\begin{align*}
			(M\Sigma M^*)_t:=\int_0^t e^{sB_0} M\Sigma M^* e^{sB^*_0} ds \succ 0 \quad \mbox{for all $t>0$}.
		\end{align*}
		A straightforward computation shows that $(M\Sigma M^*)_t=M\Sigma _tM^*$ for all $t>0$.
		Let $\varrho$ be the smallest eigenvalue of $\Re(\lambda_1)M\Sigma_\infty M^*$, and $Q^\varrho$ be the semigroup generated by 
		\begin{align}\label{eq:A_rho}
			A^{\mathrm{OU}}_\varrho=\frac{\varrho}{2}\Delta+\langle B_0 x, \nabla \rangle
		\end{align}
		with the invariant distribution 
		\begin{align*}
			\nu_\varrho(dx)=\frac{(\Re(\lambda_1)\cdots\Re(\lambda_d))^{\frac 12}}{(\pi \varrho)^{\frac d2}} e^{-\frac{1}{\varrho}\sum_{i=1}^d \Re(\lambda_i) x^2_i } dx.
		\end{align*}
		We also note that $A_\varrho$ is the tensorization of 1-D or 2-D diffusion operators defined by
		\begin{gather*}
			\frac{\varrho}{2}\frac{d^2}{dx^2}-\lambda_i x\frac{d}{dx} \ \mbox{on $\L^2(\R,(\lambda_i/\pi \rho)^{1/2} e^{-\lambda_ix^2/\varrho})$}, \ \mbox{or} \\
			\frac{\varrho}{2}\Delta+\langle C_j x, \nabla\rangle \ \mbox{on $\L^2(\R^2,(a_j/\pi \rho) e^{-a_j|x|^2/\varrho})$}
		\end{gather*}
		where $C_1,\ldots, C_r$ are defined in \eqref{eq:diagonalization}. While the 1-D diffusion operator is self-adjoint, the 2-D diffusion operator considered above is unitarily equivalent to the complex OU operator, which is a normal operator, see \cite[\S4]{ChenLiu2014} . As a result,
		 $Q^\varrho$ is a Markov semigroup of normal operators on $\L^2(\nu_\varrho)$ with the spectral decomposition
		 	\begin{align}\label{eq:normal_spect_exp}
		 	Q^\varrho_t f=\sum_{n\in\mathbb{N}^d_0} e^{-t\langle n,\lambda\rangle}\langle f, H_n\rangle_{\L^2(\nu_\varrho)} H_n 
		 \end{align}
		 for all $f\in \L^2(\nu_\varrho)$, where $H_n$'s are the scaled Hermite-It\^o-Laguerre polynomials defined by 
		 \begin{align}
		 	H_n(x)=\frac{(-1)^{|n|}}{\sqrt{2^{|n|}n!}}\frac{\partial^n_\star \nu_\varrho(x)}{\nu_\varrho(x)} 
		 	&=\prod_{i=1}^{k} \left(\frac{\lambda_i}{\varrho}\right)^{\frac 12}\varphi_{n_i}\left(\sqrt{\frac{\lambda_i}{\varrho}} u_i\right) \nonumber \\
		 	&\times \prod_{j=1}^r\psi_{n_{k+2j-1}, n_{k+2j}}\left(\sqrt{\frac{\Re(\lambda_{k+2j})}{\varrho}}(\zeta_j,\overline{\zeta}_j)\right) \label{eq:ito-hermite}
		 \end{align}
		 where $\varphi_j$ is the one-dimensional $j^{th}$ Hermite polynomials, that is, for all $j\in\bb{N}_0$,
		 \begin{align*}
		 	\varphi_j(x)=\frac{1}{\sqrt{j!}}e^{x^2}\left(\frac{d}{dx}\right)^j (e^{-x^2}),
		 \end{align*}
		 $k+2r=d$, 
		 and $\psi_{m,n}$ is the 2-dimensional Hermite-It\^o-Laguerre polynomials defined by
		 \begin{align*}
		 	\psi_{m,n}(\zeta, \overline{\zeta})=\frac{1}{\sqrt{m!n!}}e^{|\zeta|^2}\partial^m_{\overline{\zeta}}\partial^n_{\zeta} (e^{-|\zeta|^2}), \quad \zeta\in \mathbb{C}.
		 \end{align*} 
		 In the above formula, we identify $x=(u,v)\in\R^k\times\R^{2r}$ and $\zeta_j = v_{2j-1}+\i v_{2j}$, as introduced in Notation~\ref{not:R}. Also, the derivative operator $\partial^n_\star$ is defined according to Notation~\ref{not:multi_derivative}. The rescaled Hermite-It\^o-Laguerre polynomials $(H_n)_{n\in\mathbb{N}^d_0}$ form a complete orthonormal basis of the Hilbert space $\L^2(\nu_\varrho)$.
		 The following theorem is the main result of this section that establishes an intertwining relationship between the L\'evy-OU semigroup $P$ and the normal diffusion OU semigroup $Q^\varrho$.
		\begin{theorem}\label{thm:intertwining_diagonal}
			For all $t\ge 0$ and $1\le p\le\infty$, we have 
			\begin{align}\label{eq:intertwining_rho}
				Q^\varrho_t V =V P_t \quad \mbox{on \ $\L^p(\mu)$},
			\end{align}
			where $V f(x)=\int_{\R^d} f(M^{-1}x+y) h_V(dy)$ with 
			\begin{align}\label{eq:h_V}
				\cc{F}_{h_V}(\xi)
				=\exp\left(\Psi_\infty(\xi)+\frac{1}{4} |D (M^{-1})^* \xi|^2\right),
			\end{align}
			where $D=\mathrm{diag}\left(\sqrt{\varrho/\Re(\lambda_1)},\ldots, \sqrt{\varrho/\Re(\lambda_d)}\right)$.
		\end{theorem}
		\begin{proof}
			Firstly, from \eqref{eq:similarity} it follows that for all $t\ge 0$, $Q^M_tS^{-1}_M=S^{-1}_M Q_t$ on $\L^p(\ov\mu)$ for all $1\le p\le \infty$, where $Q^M$ is the semigroup generated by $A^{\mathrm{OU}}_M$ defined in \eqref{eq:gen_M}. On the other hand, we also note that $M\Sigma_\infty M^*\succ \mathrm{diag}(\varrho/2\Re(\lambda_1),\ldots, \varrho/2\Re(\lambda_d))$. Therefore, by Proposition~\ref{prop:left_int} it follows that 
			\begin{align*}
				Q^\varrho_t\Gamma=\Gamma Q^M_t \quad \mbox{on \ $\L^p(\ov{\mu}_M)$},
			\end{align*}
			where $\Gamma f(x)=\int_{\R^d} f(x+y)h_\Gamma(y)dy$ with 
			\begin{align*}
				\cc{F}_{h_\Gamma}(\xi)=\exp\left(\frac{1}{4}|D\xi|^2-\frac12\langle M\Sigma_\infty M^* \xi,\xi\rangle\right).
			\end{align*}
			 Defining $V=\Gamma S_{M}\Lambda$, where $\Lambda$ is as in Proposition~\ref{prop:left_int}, we obtain \eqref{eq:intertwining_rho}. Now, for any $\xi\in\R^d$,
			\begin{align}
				V\e_{\i\xi}&=\Gamma S_{M}\Lambda\e_{\i\xi}=\exp(\Phi_\infty(\xi))\Gamma S^{-1}_{M} \e_{\i\xi} \nonumber \\
				&=\exp(\Phi_\infty(\xi))\Gamma\e_{\i (M^{-1})^*\xi} \nonumber \\
				&= \exp(\Psi_\infty(\xi))\exp\left(\frac{1}{4}|D(M^{-1})^*\xi|^2\right)\e_{\i (M^{-1})^*\xi}, \label{eq:V}
			\end{align}
		where the last identity follows from the fact that 
		\begin{align*}
			\Psi_\infty(\xi)=-\frac{1}{2}\langle \Sigma_\infty\xi, \xi\rangle+\Phi_\infty(\xi) \quad  \mbox{for all $\xi\in\R^d$}.
		\end{align*}
		 Since $\Gamma, \Lambda, S_M$ are pseudo-differential operators, \eqref{eq:V} implies that $Vf(x)=\int_{\R^d} f(M^{-1}x+y)h_V(dy)$ with $\cc{F}_{h_V}$ given by \eqref{eq:h_V}.
		\end{proof}
		
		\changelocaltocdepth{1}
		
		\section{Proofs of Theorem~\ref{thm:spectrum} and Theorem~\ref{thm:multiplicity}}\label{sec:pf}
		
\subsection{Regularity of $P_t$} When H\ref{non-degeneracy} holds, the diffusion OU operator is hypoelliptic and the corresponding semigroup maps the $\L^2$ space weighted with the invariant distribution to the weighted Sobolev space, see \cite[Lemma~2.2]{MetafunePallaraPriola2002}.
In the same spirit, we provide a regularity estimate for the L\'evy-OU semigroup $P_t$. In this context, we mention that H\ref{non-degeneracy} is equivalent to 
\begin{align*}
	\rank\left[\Sigma^{\frac{1}{2}},B\Sigma^{\frac{1}{2}},\cdots, B^{d-1}\Sigma^{\frac{1}{2}}\right]=d,
\end{align*}
which is known as the \emph{Kalman rank condition}.
To this end, let us define 
\begin{align*}
	\mathfrak{m}=\min\left\{n: \rank\left[\Sigma^{\frac{1}{2}}, B\Sigma^{\frac 12},\cdots, B^n\Sigma^{\frac 12}\right]=d\right\}.
\end{align*} In particular, $\mathfrak{m}=0$ if and only if $\Sigma$ is invertible. Then from \cite{seidman1988} we know that 
\begin{align}\label{eq:Q_decay}
	\left\|\Sigma_t^{-\frac 12} e^{tB}\right\|\le\frac{C}{t^{\frac 12+\mathfrak{m}}}, \quad t\in (0,1]
\end{align}
for some $C>0$. 
\begin{theorem}\label{thm:regularity} Assume that H\ref{non-degeneracy} holds. Then, for all $t>0$, $P_t$ has smooth transition density. The invariant distribution $\mu$ of $P_t$ is absolutely continuous and denoting the density by $\mu$, we have $\mu\in C^\infty_0(\R^d)$. Finally,
	for any $1<p<\infty$ and $t\in (0,1]$, $P_t$ maps $\L^p(\mu)$ to $\mathrm{W}^{k,p}(\mu)$ continuously. More precisely, for all $t\in (0,1]$, $n\in\bb{N}^d_0$ and $p>1$, 
	\begin{align}\label{eq:L_p_estimate}
		\|\partial^{n}P_t f\|_{\L^p(\mu)}\le \frac{C}{t^{k(\frac 12+\mathfrak{m})}} \|f\|_{\L^p(\mu)},
	\end{align}
	where $|n|=k$ and $C=C(k,p)$ is a positive constant depending on $k,p$. 
\end{theorem}

		\begin{proof}[Proof of Theorem~\ref{thm:regularity}]\label{sec:reg_pf} Due to the assumption H\ref{non-degeneracy}, it follows that $\Sigma_t$ is positive definite for all $t>0$. As a result, for each $t>0$, the function $\xi\mapsto |\xi|^n e^{\Psi_t(\xi)}$ is integrable for all $n\ge 1$. From \eqref{eq:characteristic_fn}, it follows that the transition density of $P_t$, denoted by $p_t$ is smooth. Also, $p_t$ is the convolution of a Gaussian density and a probability measure, which implies that $p_t>0$ for any $t>0$. Smoothness of $\mu$ follows by the same argument.
		To prove the estimate \eqref{eq:L_p_estimate}, we use an idea similar to the proof of \cite[Lemma~2.2]{MetafunePallaraPriola2002} with some modifications needed in our setting. Let us start with the case $k=1$. For any $f\in\mathcal{S}(\bb{R}^d)$, using the Fourier inversion formula we can write
		\begin{align*}
			f(x)=\frac{1}{(2\pi)^d}\int_{\bb{R}^d}e^{\i\langle \xi,x\rangle}\cc{F}_f(\xi)d\xi.
		\end{align*}
		Therefore, for all $t\ge 0$, using \eqref{eq:characteristic_fn} we obtain 
		\begin{align}
			P_tf(x)&=\frac{1}{(2\pi)^d}\int_{\bb{R}^d}P_t\exi(-x)\cc{F}_f(\xi)d\xi \nonumber \\
			&=\frac{1}{(2\pi)^d}\int_{\bb{R}^d}e^{\Psi_t(-\xi)}e^{-\i \langle e^{tB^*}\xi,x\rangle}\cc{F}_f(\xi)d\xi \nonumber \\
			&=\frac{1}{(2\pi)^d}\int_{\bb{R}^d}e^{-\i\langle\xi, e^{tB}x\rangle}e^{-\frac{1}{2}\langle \Sigma_t\xi,\xi\rangle}e^{\Phi_t(-\xi)}\cc{F}_f(\xi)d\xi \nonumber \\
			&=b_{t}*f*g_t(e^{tB}x) \label{eq:semigroup}
			%&=\int_{\R^d} f(y) (b_t\ast \overline{\nu}_t)(e^{tB}x-y)dy
		\end{align}
		where 
		\begin{align}\label{eq:b_t}
			b_t(x)=\frac{(2\pi)^{-\frac d2}}{\sqrt{\det{\Sigma_t}}}e^{-\frac{1}{2}\langle \Sigma^{-1}_tx,x\rangle}, \quad \cc{F}_{g_{t}}(\xi)=e^{\Phi_t(-\xi)}.
		\end{align}
		Since $b_t\in C^\infty_0(\bb{R}^d)$, and $f*g_{t}$ is locally bounded, using dominated convergence theorem the following interchange of derivative under the integral sign is justified:
		\begin{align}
			\nabla P_t f(x)&=\nabla_x \int_{\R^d} f\ast g_t(y) b_t(e^{tB}x-y) dy \nonumber\\
			&=\int_{\R^d} f\ast g_t(y) \nabla_x b_t(e^{tB}x-y) dy \nonumber \\
			&=-e^{tB^*}\Sigma^{-1}_t\int_{\R^d} f\ast g_t(y) b_t(e^{tB}x-y) (e^{tB}x-y) dy \nonumber \\
			&=-e^{tB^*}\Sigma^{-1}_t\int_{\bb{R}^d}f* g_{t}(e^{tB}x-y) b_t(y)ydy. \label{eq:derivative_P}
		\end{align} 
		Next, for any $p>1$, letting $p^{-1}+q^{-1}=1$, and using H\"older's inequality on \eqref{eq:derivative_P} with respect to the measure $b_t(y)dy$, we get that for any $1\le i\le d$,
		\begin{align*}
			&|\partial_i P_tf(x)| \\
			&\le\left(\int_{\bb{R}^d}|( \Sigma^{-1/2}_te^{tB}\bm{e}_i,\Sigma^{-1/2}_ty)|^{q} b_t(y)dy\right)^{\frac{1}{q}} \left(\int_{\bb{R}^d}|f* g_{t}(e^{tB}x-y)|^p b_t(y)dy\right)^{\frac{1}{p}} \\
			&\le|\Sigma^{-1/2}_t e^{tB}e_i|\left(\int_{\bb{R}^d} |\Sigma^{-1/2}_t y|^q b_t(y)dy\right)^{\frac{1}{q}}\left(\int_{\bb{R}^d}|f* g_{t}(e^{tB}x-y)|^p b_t(y)dy\right)^{\frac{1}{p}} \\
			&\le C_q t^{-\frac{1}{2}-\mathfrak{m}} \left(\int_{\bb{R}^d}|f|^p* g_{t}(e^{tB}x-y) b_t(y)dy\right)^{\frac{1}{p}} \\
			&\le C_q t^{-\frac{1}{2}-\mathfrak{m}} \left(P_t |f|^p(x)\right)^{\frac{1}{p}},
		\end{align*}
		where $C_q=\int_{\R^d} |y|^q e^{-|y|^2/2} dy$ and the second to the last inequality follows from Jensen's inequality with respect to the probability measure $g_t$. Since $\mu$ is the invariant distribution of $P_t$, from the above inequality we infer that 
		\begin{align}\label{eq:sobolev_contractivity}
			\|\partial_i P_t f\|_{\L^p(\mu)}\le  C_p t^{-\frac{1}{2}-\mathfrak{m}}\|f\|_{\L^p(\mu)}.
		\end{align}
		Since $\mathcal{S}(\bb{R}^d)$ is dense in $\L^p(\mu)$ and $\mathrm{W}^{k,p}(\mu)$ is a Banach space, \eqref{eq:sobolev_contractivity} implies that $P_t f\in \mathrm{W}^{1,p}(\mu)$ for all $t\in (0,1]$, and \eqref{eq:sobolev_contractivity} extends for all $f\in\L^p(\mu)$,
		which proves \eqref{eq:L_p_estimate} for $k=1$. For $k\ge 2$, the proof follows similarly by iterating the above technique along with the observation 
		\begin{align}\label{eq:derivative_iteration}
			\partial P_t f(x)=e^{tB^*} P_t \partial f(x) \quad \text{for all $f\in\mathrm{W}^{k,p}(\mu)$}.
		\end{align}
		The above identity can be easily verified when $f\in C^\infty_c(\R^d)$ and the rest follows by a density argument.
		\end{proof}
		
		\subsection{Generalized eigenfunctions of $A_p$} For any closed operator $(T,\cc{D}(T))$ defined on a Banach space $\cc X$, $v$ is called a \emph{generalized eigenvector} corresponding to an eigenvalue $\theta$ if there exists $r\ge 1$ such that $v\in\cc{D}(T^r)$ and  $(T-\theta I)^r v=0$. Also, the \emph{index} of an eigenvalue $\theta$ is defined as 
		\begin{align*}
			\iota(\theta; T)=\min\{r\ge 1: \ker(T-\theta I)^r=\ker (T-\theta)^{r+1}\}.
		\end{align*}
		The algebraic and geometric multiplicities of an eigenvalue $\theta$ coincide if and only if $\iota(\theta; T)=1$.
		
		To prove the remaining theorems in \S\ref{sec:spectrum}, we first show that the generalized eigenfunctions of the generator $(A_p, \cc{D}(A_p))$ defined in \eqref{eq:OU_gen} are polynomials for any $1<p<\infty$. We start with the following lemma.
		\begin{lemma}\label{lem:sobolev_estimates}
			Let $k\in\bb{N}$ and $\epsilon>0$ be such that $s(B)+\epsilon<0$, where $s(B)=\max\{\Re(\lambda):\lambda\in\sigma(B)\}$. Then, there exists a constant $C=C(k,\epsilon)>0$ such that for every $u\in\mathrm{W}^{k,p}(\mu)$,
			\begin{align*}
				\sum_{|n|=k}\|\partial^n P_t u\|_{\L^p(\mu)}\le Ce^{tk(s(B)+\epsilon)}\sum_{|n|=k} \|\partial^n u\|_{\L^p(\mu)},
			\end{align*}
		
		\end{lemma}
		\begin{proof}
			Since \eqref{eq:derivative_iteration} holds, proof of this lemma is the same as the proof of \cite[Lemma~3.1]{MetafunePallaraPriola2002}.
		\end{proof}
		\begin{proposition}\label{prop:generalized_eigenfunction}
			Assume H\ref{assumption_polynomial} and H\ref{non-degeneracy}. Then, for all $1<p<\infty$, the generalized eigenfunctions of $(A_p, \cc{D}(A_p))$ corresponding to an eigenvalue $\theta\in\bb{C}_-$ are polynomials of degree at most $|\frac{\Re(\theta)}{s(B)}|$.
		\end{proposition}
		\begin{proof}
			We use an argument very similar to the proof of \cite[Proposition~3.1]{MetafunePallaraPriola2002}. Let $\theta$ be an eigenvalue of $(A_p, \cc{D}(A_p))$ and let $v$ be a generalized eigenfunction of $A_p$, that is, $v\in\cc{D}(A^r_p)$ and $(A_p-\theta I)^r v=0$, $(A_p-\theta I)^{r-1}v\neq 0$ for some $r\ge 1$. Suppose that $r=1$. Then, $v$ is an eigenfunction of $P_t$ and $P_t v=e^{\theta t} v$  for all $t\ge 0$. From Theorem~\ref{thm:regularity} it follows that $v\in\mathrm{W}^{k,p}(\mu)$ for every $k\ge 1$. Also, for any $n\in\bb{N}^d_0$, $\partial^n P_t v=e^{\theta t}\partial^n v$. Therefore, using Lemma~\ref{lem:sobolev_estimates} we get 
			\begin{align*}
				e^{\Re(\theta)t}\sum_{|n|=k}\|\partial^nv\|_{\L^p(\mu)}&=\sum_{|n|=k}\|\partial^n P_t v\|_{\L^p(\mu)} \\
				&\le  Ce^{tk(s(B)+\epsilon)}\sum_{|n|=k} \|\partial^n v\|_{\L^p(\mu)}.
			\end{align*}
			This shows that $\partial^n v=0$ whenever $|n|\ge |\Re(\theta)|/|s(B)|$, and hence $v$ is a polynomial of degree at most $|\frac{\Re(\theta)}{s(B)}|$. For $r>1$, we proceed by induction. Suppose that the statement of the proposition holds for all $1\le j\le r-1$. Since for all $t>0$,
			\begin{align*}
				P_t v= e^{\theta t} v+e^{\theta t}\sum_{j=1}^{r-1}\frac{t^j(A_p-\theta I)^j v}{j!},
			\end{align*}
			by our induction hypothesis, $(A_p-\theta I)^j v$ is a polynomial of degree at most $|\Re(\theta)/s(B)|$ for each $1\le j\le r-1$. As a result, 
			\begin{align*}
				\partial^n P_t v=e^{\theta t} \partial^n v
			\end{align*}
			for all $n\in \bb{N}^d_0$ with $|n|>|\Re(\theta)/s(B)|$. Imitating the same argument as before, we conclude the proof of the proposition.
		\end{proof}
		Next, in the diffusion case, that is, when $Q=(Q_t)_{t\ge 0}$ is the OU semigroup generated by the diffusion operator $A^{\mathrm{OU}}$ defined in \eqref{eq:A-OU}, we first show that $\sigma(Q_t; \L^p(\nu))=\sigma_p(Q_t; \L^p(\nu))=e^{t\bb{N}(\lambda)}$ and our proof is different from \cite[Theorem~3.1]{MetafunePallaraPriola2002}. We do not use \cite[Lemma~3.2]{MetafunePallaraPriola2002}, which is a key observation in the aforementioned paper, but we will rely on the results regarding the compactness of $Q_t$ and the eigenvalues of $e^{tL}$ obtained by the authors. 
		\begin{proposition}\label{prop:diffusion_eigen}
			For all $t>0$ and $1<p<\infty$, 
			\begin{align*}
				\sigma(Q_t; \L^p(\nu))\setminus\{0\}=\sigma_p(Q_t; \L^p(\nu))=e^{t\bb{N}(\lambda)}.
			\end{align*}
		\end{proposition}
		\begin{proof}
			The equality $\sigma(Q_t; \L^p(\nu))\setminus\{0\}=\sigma_p(Q_t; \L^p(\nu))$ follows from the compactness of $Q_t$ and we refer to \cite[p.~45]{MetafunePallaraPriola2002} for the proof of this fact. In what follows, we prove that $\sigma_p(Q_t; \L^p(\nu))=e^{t\bb{N}(\lambda)}$. First, we note that $\nu$ is a Gaussian measure and therefore, the space of polynomials $\mathscr{P}\subset \L^p(\nu)$ for all $p\ge 1$. Moreover, due to Proposition~\ref{prop:generalized_eigenfunction}, any eigenfunction of $Q_t$ is a polynomial. Next, we recall the following identity from Proposition~\ref{prop:left_int}:
			\begin{align*}
				e^{tL}\Lambda_1=\Lambda_1 Q_t \quad \mbox{on \  $\mathscr{P}$}.
			\end{align*}
			Since $\nu$ is Gaussian, $\Lambda_1$ is a Gaussian convolution kernel and therefore, $\Lambda_1: \mathscr{P}\longrightarrow\mathscr{P}$ is bijective. Therefore, the above identity implies that $v\in\mathscr{P}$ is an eigenfunction of $Q_t$ corresponding to an eigenvalue $e^{\theta t}$ if and only if $e^{tL}\Lambda_1v= e^{\theta t} \Lambda_1 v$ with $\Lambda_1 v\in\mathscr{P}$. Hence, $\sigma_p(Q_t; \L^p(\nu))$ consists of the eigenvalues of $e^{tL}$ on the space of polynomials $\mathscr{P}$. In other words, $\sigma_p(Q_t; \L^p(\nu))=\sigma_p(e^{tL}; \cc {P})$. It remains to prove that $\sigma_p(e^{tL}; \mathscr{P})= e^{t\bb{N}(\lambda)}$. While the fact $\sigma_p(e^{tL}; \mathscr{P})\subseteq e^{t\bb{N}(\lambda)}$ is proved in \cite[p.~50]{MetafunePallaraPriola2002}, proof of the reverse inclusion $e^{t\bb{N}(\lambda)}\subseteq\sigma_p(e^{tL}; \mathscr{P})$ follows from the argument described in \cite[p.~52]{MetafunePallaraPriola2002}. This completes the proof of the proposition. 
		\end{proof}
		\begin{proof}[Proof of Theorem~\ref{thm:spectrum}] 
			If H\ref{assumption_polynomial} holds, $\mu$ has finite moments of all orders. As  a result, $\mathscr{P}\subset \L^p(\mu)$ for all $p\ge 1$. Since $\Lambda:\mathscr{P}\longrightarrow\mathscr{P}$ is bijective, from \eqref{eq:intertwining_drfit} we obtain $\sigma_p(e^{tL};\mathscr{P})\subseteq\sigma_p(P_t;\mathscr{P})$ for all $t\ge 0$. From \cite{MetafunePallaraPriola2002} it is known that $e^{t\mathbb{N}(\lambda)} = \sigma_p(e^{tL};\mathscr{P})$. Therefore, \eqref{it:point_spectrum} follows from the spectral mapping theorem.
			
			If H\ref{non-degeneracy} holds, for $1<p<\infty$, taking the adjoint of \eqref{eq:intertwining_p} in Proposition~\ref{prop:left_int}, we obtain 
			\begin{align}\label{eq:intertwining_adjoint}
				P^*_t\Lambda^*=\Lambda^* Q^*_t \quad \mbox{on \ $\L^q(\nu)$},
			\end{align}
			where $p^{-1}+q^{-1}=1$. As $Q_t$ is compact, Proposition~\ref{prop:diffusion_eigen} implies that $\sigma_p(Q^*_t; \L^q(\nu))=e^{t\bb{N}(\lambda)}$.  Since $\Lambda:\L^p(\mu)\longrightarrow\L^p(\nu)$ has dense range, $\Lambda^*:\L^q(\nu)\longrightarrow \L^q(\mu)$ is injective. Therefore, \eqref{eq:intertwining_adjoint} implies that $e^{t\bb{N}(\lambda)}\subseteq \sigma_p(P^*_t; \L^q(\mu))$. As $\overline{\sigma(P_t; \L^p(\mu))}=\sigma(P^*_t; \L^q(\mu))$, we conclude that $e^{t\bb{N}(\lambda)}\subseteq \sigma(P_t; \L^p(\mu))$. By spectral mapping theorem, \eqref{it:coeigen_H4} follows.

			Let us now assume H\ref{assumption_polynomial} and H\ref{non-degeneracy} hold. If $e^{\theta t}$ is an eigenvalue of $P_t$ in $\L^p(\mu)$, then by Proposition~\ref{prop:generalized_eigenfunction}, any eigenfunction corresponding to this eigenvalue is a polynomial of degree at most $|\Re(\theta)/s(B)|$. Therefore, invoking the identity \eqref{eq:intertwining_drfit} in Proposition~\ref{prop:left_int} and using the same argument as in the proof of Proposition~\ref{prop:diffusion_eigen}, we conclude that $\sigma_p(P_t; \L^p(\mu))=e^{t\bb{N}(\lambda)}$ for all $1<p<\infty$. Hence, \eqref{it:poly_eigen} follows by spectral mapping theorem.
		\end{proof}
		
			\begin{proof}[Proof of Theorem~\ref{thm:multiplicity}] Since H\ref{assumption_polynomial} holds, $\mathscr{P}\subset \L^p(\mu)$ for all $p\ge 1$. Also, by Theorem~\ref{thm:spectrum}, $\sigma_p(P_t; \L^p(\mu))= e^{t\bb{N}(\lambda)}$ for all $t>0$ and $1<p<\infty$. This also implies that $\sigma_p(A_p)=\bb{N}(\lambda)$. From \eqref{eq:intertwining_drfit} in Proposition~\ref{prop:left_int}, we note that for all $\theta\in\bb{N}(\lambda)$ and $r\ge 1$,
			\begin{align*}
				(L-\theta I)^r \Lambda_1=\Lambda_1 (A_p-\theta I)^r \quad \mbox{on \ $\mathscr{P}$}.
			\end{align*}
			Since $\Lambda_1:\mathscr{P}\longrightarrow\mathscr{P}$ is invertible, $\ker(A-\theta I)^r=\ker (L-\theta I)^r$ on $\mathscr{P}$ for all $r\ge 1$. This proves that $\mathtt{M}_a(\theta,A_p)=\mathtt{M}_a(\theta, L)$ and $\mathtt{M}_g(\theta, A_p)=\mathtt{M}_g(\theta, L)$ for all $1<p<\infty$. In particular, $\mathtt{M}_a(\theta,A_p)=\mathtt{M}_g(\theta, A_p)$ for all $\theta\in\bb{N}(\lambda)$ if and only if $\mathtt{M}_a(\theta, L)=\mathtt{M}_g(\theta, L)$ for all $\theta\in\mathbb{N}(\lambda)$. In this case, the index of each eigenvalue $\theta$ of $L$ is $1$, which by \cite[Proposition~4.3]{MetafunePallaraPriola2002} holds if and only if $B$ is diagonalizable.
		\end{proof}

		\section{Proofs of results in \S\ref{sec:eigen_coeigen}} \label{sec:exp_moment}
	 We first prove the results under slightly restrictive assumptions, that is, when the diffusion matrix $\Sigma$ satisfies H\ref{non-degeneracy} and the L\'evy measure has exponential moments of some order, see Assumption H\ref{assumption3} below. In this case, the limiting distribution $\mu$ has a smooth density, and its characteristic function has an analytic extension in a cylinder in $\mathbb{C}^d$. In the latter case, we obtain a contour integral representation of the eigenfunctions similar to the Hermite polynomials. Then H\ref{assumption3} is relaxed to H\ref{assumption_polynomial} using finite truncations of $\Pi$. For the co-eigenfunctions, we first prove Theorem~\ref{thm:spectral_expansion} under H\ref{non-degeneracy} by means of intertwining relationship, and this assumption is later relaxed by approximating $\Sigma$ by positive definite matrices.
	
	 Let us now introduce the following assumption about the existence of exponential moments of the L\'evy measure. 
		\begin{assumption}[Exponential moment of small order]\label{assumption3} There exists $\kappa>0$ such that 
			\begin{align*}
				\int\limits_{\{|x|>1\}} e^{\kappa |x|}\Pi(dx)<\infty.
			\end{align*}
			This assumption is stronger than H\ref{assumption_polynomial}. By \cite[Theorem~25.17]{SatoBook1999} this is equivalent to the analyticity of the L\'evy-Khintchine exponent $\Psi$ in the cylinder  $\bb{C}_\kappa=\{z\in\bb C^d: |\Re(z)|<\kappa\}$. Alternatively, this assumption holds if and only if $\bb{E}(\e_{z}(Z_t))<\infty$ for all $t>0$ with $|\Re(z)|\le\kappa$, where $(Z_t)_{t\ge 0}$ is the L\'evy process associated with the L\'evy-Khintchine exponent $\Psi$.
		\end{assumption}

		\begin{proposition}\label{prop:eigen_integral}
			Suppose that H\ref{assumption3} holds. For any $r>0$, let us denote 
			\begin{align*}
				C_r=\{z\in\bb{C}^d: |z_j|= r \ \text{for all } j=1,\ldots, d\}.
			\end{align*}
			Then, for sufficiently small $r>0$, the function $\mathcal{H}_n$ defined by
			\begin{equation}\label{eq:integral_rep}
				\begin{aligned}
					\cc{H}_n(x)&=\frac{\sqrt{2^{|n|}n!}}{(2\pi \i)^d}\int_{C_r}\frac{\exp(\langle Mx, \overline{R}^* z\rangle-\Psi_\infty(-\i M^*\overline{R}^* z))}{z^{n+1}} dz
				\end{aligned}
			\end{equation}
			satisfies $A_p \mathcal{H}_n = -\langle n,\lambda\rangle\mathcal{H}_n$ for all $n\in\mathbb{N}^d_0$, where $\overline{R}$ is the conjugate of the linear map $R$ defined in Notation~\ref{not:R}, that is, $\overline{R} x= \overline{Rx}$ for all $x\in\R^d$. Moreover, for all $n\in\mathbb{N}^d_0$,
			\begin{equation} \label{eq:L2_norm_2}
			\begin{aligned}
				 &2^{-|n|}n!\|\mathcal{H}_n\|^2_{\L^2(\mu)} \\
				&=
				\left.\partial^n_{\star,z}\overline{\partial}^n_{\star,w}\exp(\Psi_\infty( M^*(z-w))-\Psi_\infty( M^* z)-\overline{\Psi_\infty( M^* w)})\right|_{z=w=0}.
			\end{aligned}
			\end{equation}
			Finally, $\mathcal{H}_n$ is a polynomial of degree $|n|$ and
			\begin{equation}\label{eq:eigen_poly_2}
				\begin{aligned}
					\cc{H}_n(x)=\frac{2^{\frac{|n|}{2}}}{\sqrt{n!}}\sum_{0\le k\le n} (-\i)^{|k|} \dbinom{n}{k} (RMx)^{n-k}\partial^{k}_{\star}\left(e^{-\Psi_\infty\circ M^*}\right)(0).
				\end{aligned}
			\end{equation}
			
		\end{proposition}
		\begin{remark}
			The integral representation in \eqref{eq:integral_rep} is a generalization of the integral formula for real Hermite polynomials, which is given by
			\begin{align*}
					H_n(x)=\frac{\sqrt{2^nn!}}{2\pi \i}\oint\frac{e^{zx}e^{-z^2/4}}{z^{n+1}}dz.
			\end{align*}
			$H_n$ is the eigenfunction of the 1-D Ornstein-Uhlenbeck operator $L=\frac{1}{2}\frac{d^2}{dx^2}-xf'(x)$, whose invariant distribution is given by $\nu(dx)=\pi^{-1/2} e^{-x^2}dx$.
		\end{remark}

		\begin{proof}[Proof of Proposition~\ref{prop:eigen_integral}] Due to H\ref{assumption3}, we have $\mathbb{E}_x\left[e^{\kappa |X_t|}\right]<\infty$ for all $x\in\mathbb{R}^d$ and $t>0$. Therefore, by \eqref{eq:characteristic_fn}, for sufficiently small $r>0$ and for all $z\in C_r$, 
		\begin{align*}
			P_t\mathcal{E}_{z}(x)=e^{\Psi_t(-\i z)}\mathcal{E}_{e^{tB^*} z}(x).
		\end{align*}
		Using Fubini's theorem, we get 
		\begin{align*}
			P_t \mathcal{H}_n(x)&=\frac{\sqrt{2^{|n|}n!}}{(2\pi \i)^d}\int_{C_r} \frac{P_t \mathcal{E}_{M^* \overline{R}^* z}(x) e^{-\Psi_\infty(-\i M^*\overline{R}^* z)}}{z^{n+1}} dz \\
			& = \frac{\sqrt{2^{|n|}n!}}{(2\pi \i)^d}\int_{C_r}\frac{e^{\Psi_t(-\i M^*\overline{R}^* z)}e^{-\Psi_\infty(-\i M^*\overline{R}^* z)}\mathcal{E}_{e^{tB^*} M^*\overline{R}^* z}(x)}{z^{n+1}}dz.
		\end{align*}
		Using Lemma~\ref{lem:Psi_t} and the identity $\overline{R}Me^{tB}M^{-1}\overline{R}^*=e^{tD^*_\lambda}$, where \begin{align*}D_\lambda=-\mathrm{diag}(\lambda_1,\cdots, \lambda_d)\end{align*} we obtain
		\begin{align*}
			P_t \mathcal{H}_n(x)= \frac{\sqrt{2^{|n|}n!}}{(2\pi \i)^d}  \int_{C_r}\frac{e^{-\Psi_\infty(-\i M^*\overline{R}^*e^{tD_\lambda} z)}\mathcal{E}_{M^*\overline{R}^* e^{tD_\lambda} z}(x)}{z^{n+1}}dz.
		\end{align*}
		Making the change of variable $z\mapsto e^{tD_\lambda}z$ and noting that $(e^{tD_\lambda}z)^{n+1}=e^{-\langle n+1,\lambda\rangle} z^{n+1}$ and $\det(e^{tD_\lambda})=e^{-t\sum_{i=1}^d \lambda_i}$, the above integral reduces to 
		\begin{align*}
				P_t \mathcal{H}_n(x)&= \frac{\sqrt{2^{|n|}n!}}{(2\pi \i)^d} e^{-t\langle n,\lambda\rangle}  \int_{C'_{r}}\frac{e^{-\Psi_\infty(-\i M^*\overline{R}^* z)}\mathcal{E}_{M^*\overline{R}^*  z}(x)}{z^{n+1}}dz \\
				& = e^{-t\langle n,\lambda \rangle}\mathcal{H}_n(x),
		\end{align*}
		where $C'_{r}$ is the image of $C_r$ under the above transformation.
		By Cauchy integral formula, $\mathcal{H}_n$ is a polynomial of degree $|n|=n_1+\cdots +n_d$, and hence by Lemma~\ref{lem:moment}, $\mathcal{H}_n\in \L^p(\mu)$ for all $n\in\mathbb{N}_0^d$ and $1\le p<\infty$.
		This proves that $\mathcal{H}_n$ is an eigenfunction of the semigroup $P_t$ for all $t>0$. By spectral mapping theorem, we conclude that $A_p \mathcal{H}_n=-\langle n,\lambda \rangle \mathcal{H}_n$ for all $n\in\mathbb{N}^d_0$ and for all $1\le p<\infty$.
		
		To prove the second assertion, we recall the simple identities 
		\begin{align*}
			\overline{\int_{\gamma} f(z) dz}&=\int_{\gamma} \overline{f(z)} d\overline{z}, \quad \text{and} \\
			\int_{\gamma} f(\overline{z}) d\overline{z}&=-\int_\gamma f(z)dz,
		\end{align*}
		where $f$ is a holomorphic function in a neighborhood of a circle $\gamma$. Using this identity we note that for any $x\in\R^d$,
		\begin{align*}
			& \overline{\mathcal{H}_n(x)} =	\frac{\sqrt{2^{|n|}n!}}{(2\pi \i)^{d}} \int_{C_r}\frac{\exp(\langle Mx, R^* w\rangle-\Psi_\infty(-\i M^*R^* w))}{w^{n+1}} dw,
		\end{align*}
		where we used the identity $\overline{\Psi_\infty(z)}=\Psi_\infty(-\overline{z})$. As a result, 
		\begin{align*}
		& \ \ \ 	|\mathcal{H}_n(x)|^2 =	\frac{2^{|n|}n!}{(2\pi\i)^{2d}}  \\
		&\times \int\limits_{C_r\times C_r}\frac{e^{\langle x, M^*(\overline{R}^*z+R^*w)\rangle} e^{-\Psi_\infty(-\i M^* \overline{R}^* z)} e^{-\Psi_\infty(-\i M^* R^* w)}}{z^{n+1} w^{n+1}} dzdw.
		\end{align*}
		Since $\int_{\R^d} \exp(\langle Mx, z+w\rangle)\mu(x)dx=\exp(\Psi_\infty(-\i M^* (z+w)))$ for $z,w\in C_r$ when $r$ is sufficiently small, using Fubini's theorem and the change of variable $(z,w)\mapsto (\i z, \i w)$ we obtain
	
			\begin{align} 
				& \ \ \ \|\mathcal{H}_n\|^2_{\L^2(\mu)} =\frac{2^{|n|}n!(-1)^{|n|}}{(2\pi\i)^{2d}} \label{eq:cauchy_mult}  \\
				 &\times \int\limits_{C_r\times C_r}\frac{e^{\Psi_\infty( M^*(\overline{R}^*z+R^*w))}e^{-\Psi_\infty( M^*\overline{R}^* z)} e^{-\Psi_\infty( M^* R^* w)}}{z^{n+1} w^{n+1}} dz dw.  \nonumber
			\end{align}
			By Cauchy integral formula for several variables, the last identity implies 
			\begin{align*} 
				&\|\mathcal{H}_n\|^2_{\L^2(\mu)} \\
				&=\frac{(-1)^{|n|} 2^{|n|}}{n!}  
				\left.\partial^n_z\partial^n_w e^{\Psi_\infty( M^*(\overline{R}^*z+R^*w))-\Psi_\infty( M^*\overline{R}^* z) -\Psi_\infty( M^* R^* w)}\right|_{z=w=0}.
			\end{align*}
			Therefore, \eqref{eq:L2_norm_2} follows by making the change of variable $(z,w)\mapsto (z, -w)$ along with the identities in \eqref{eq:complex_derivative} and the fact that $\Psi_\infty(-w)=\overline{\Psi_\infty(w)}$ for all $w\in\R^d$. Finally, \eqref{eq:eigen_poly_2} follows by Cauchy integral formula on \eqref{eq:integral_rep} after observing that $\langle Mx, \overline{R}^* z\rangle=\langle RMx, z\rangle$ for all $x\in\R^d$ and $z\in\mathbb{C}^d$. This completes the proof of the proposition.
		\end{proof}
		
		\begin{proposition}\label{prop:co-eigen-2}
			Assume that H\ref{non-degeneracy} holds, and let $B$ be diagonalizable. Then for any $n\in\mathbb{N}^d_0$, 
			\begin{align}\label{eq:coeigen_3}
				A^*_2 \mathcal{G}_n= -\overline{\langle n,\lambda \rangle}\mathcal{G}_n,
			\end{align}
			where $\mathcal{G}_n$ is defined in \eqref{eq:co-eigen}. Moreover, $\|\mathcal{G}_n\|_{\L^2(\mu)}\le 1$ for all $n\in\mathbb{N}^d_0$.
		\end{proposition}
		
		\begin{proof}
			 Since $B$ is diagonalizable, due to Theorem~\ref{thm:intertwining_diagonal}, the intertwining relation \eqref{eq:intertwining_rho} holds on $\L^2(\mu)$ for every $t>0$. As noted in \S\ref{sec:intertwining-normal}, $Q^\varrho_t$ is a normal operator on $\L^2(\nu_\varrho)$ and by \eqref{eq:normal_spect_exp}, its adjoint admits the spectral decomposition 
			\begin{align}
				Q^{\varrho *}_t f=\sum_{n\in\mathbb{N}^d_0} e^{-t\overline{\langle n,\lambda\rangle}}\langle f, H_n\rangle_{\L^2(\nu_\varrho)} H_n 
			\end{align}
			for all $f\in \L^2(\nu_\varrho)$. Taking the adjoint of the identity \eqref{eq:intertwining_rho}, we have $P^*_t V^*=V^* Q^{\varrho *}_t$ for all $t\ge 0$ on $\L^2(\nu_\varrho)$. Since $V^*:\L^2(\nu_\varrho)\longrightarrow \L^2(\mu)$ is injective, defining $\cc{G}_n=V^* H_n$ it follows that 
			\begin{align*}
				P^*_t \cc{G}_n= e^{-t\overline{\langle n,\lambda\rangle}}\cc{G}_n
			\end{align*}
			 for all $n\in\bb{N}^d_0$. Also, $V^*: \L^2(\nu_\varrho)\longrightarrow \L^2(\mu)$ is a bounded operator with $\|V^*\|=1$  as $V:\L^2(\mu)\longrightarrow \L^2(\nu_\varrho)$ is a Markov operator. This shows that 
			 \begin{align*}
			 	\|\mathcal{G}_n\|_{\L^2(\mu)}=\|V^* H_n\|_{\L^2(\mu)}\le \|H_n\|_{\L^2(\nu_\varrho)}=1
			 \end{align*}
			 It remains to prove \eqref{eq:coeigen_3}. Let $\widehat{V}$ denote the $\L^2(\R^d)$-adjoint of $V$. Then, one can easily show that for any $f\in\L^2(\nu_\varrho)$,
			\begin{align*}
				V^* f=\frac{\widehat{V}(f\nu_\varrho)}{\mu}.
			\end{align*}
			From \eqref{eq:intertwining_rho}, it is known that $\widehat{V}\nu_\varrho=\mu$. Using the formula for $H_n$ given by \eqref{eq:ito-hermite} we have
			\begin{align}\label{eq:G_n_rodrigues}
				\cc{G}_n(x)=V^* H_n(x)=\frac{(-1)^{|n|}}{\sqrt{2^{|n|}n!}}\frac{\widehat{V}(\partial^n_\star \nu_\varrho)(x)}{\mu(x)}.
			\end{align}
			From the definition of $V$ in Theorem~\ref{thm:intertwining_diagonal}, we note that $VS_M$ is a convolution operator on $\R^d$, and therefore its $\L^2(\R^d)$-adjoint $\widehat{VS_M}=\widehat{S}_M\widehat{V}$ is also a convolution operator on $\R^d$. Hence, for any $n\in\mathbb{N}^d_0$, 
			\begin{align*}
				\widehat{S}_M\widehat{V}\partial^n_\star =\partial^n_\star \widehat{S}_M\widehat{V}.
			\end{align*}
			Also, for any invertible matrix $M$, $\widehat{S}_M=|\det(M)|^{-1}S^{-1}_M$. This shows that
			\begin{align*}
				\widehat{V}(\partial^n_\star \nu_\varrho)(x) 
				&=|\det(M)|S_M \widehat{S}_M\widehat{V}(\partial^n_\star \nu_\varrho)(x) \\
				&=|\det(M)|S_M\partial^n_\star \widehat{S}_M\widehat{V} \nu_\varrho(x) \\
				& = S_M\partial^n_\star S^{-1}_M \mu(x).
			\end{align*}
			The proof is concluded by combining the last identity with \eqref{eq:G_n_rodrigues} and the spectral mapping theorem.
		
		\end{proof}
		
		\begin{lemma}\label{lem:eigen_mapping}
			Assume H\ref{assumption_polynomial}, H\ref{non-degeneracy}, and that $B$ is diagonalizable. Then, for any $n\in\mathbb{N}^d_0$,
			\begin{align*}
				V\mathcal{H}_n = H_n,
			\end{align*}
			where $\mathcal{H}_n$ and $H_n$ are defined in \eqref{eq:eigen_polynomial} and \eqref{eq:ito-hermite} respectively, and $V$ is defined in Theorem~\ref{thm:intertwining_diagonal}.
		\end{lemma}
		\begin{proof}
			For $n\in\mathbb{N}^d_0$, let us write $p_n(z)=z^n$, $z\in\mathbb{C}^d$. We claim that for any $x\in\R^d$,
			\begin{align}\label{eq:V_poly}
				V(p_n\circ RM)(x)=\sum_{0\le k\le n} (-\i)^{|k|}\dbinom{n}{k} p_{n-k}(Rx)\partial^{k}_\star F(0),
			\end{align}
			where $R$ is defined in Notation~\ref{not:R}, $M$ is defined in \eqref{eq:diagonalization}, and 
			\begin{align*}
				F(\xi)=\exp\left(\Psi_\infty(M^*\xi)-\frac{1}{4}|D\xi|^2\right).
			\end{align*}
			 Indeed, by definition of $V$, 
			\begin{align*}
				V(p_n\circ RM)(x)&=\int_{\R^d}(Rx+RMy)^n h_V(y)dy \\
				&=\int_{\R^d}\sum_{0\le k\le n} \dbinom{n}{k} (Rx)^k (RMy)^{n-k} h_V(y)dy \\
				&=|\det(M)|^{-1}\sum_{0\le k\le n} (Rx)^k\int_{\R^d} (Ry)^{n-k} h_V(M^{-1}y)dy.
			\end{align*}
			We note that 
			\begin{align*}
				\mathcal{F}_{h_V\circ M^{-1}}(\xi)=|\det(M)|\exp\left(\Psi_\infty(M^*\xi)-\frac{1}{4}|D\xi|^2\right).
			\end{align*}
			Since H\ref{assumption_polynomial} holds, by Lemma~\ref{lem:moment}, $\mu$ has moments of all orders and therefore, $\Psi_\infty$ is differentiable with 
			\begin{align*}
				|\det(M)|^{-1}\int_{\R^d} y^{n-k}h_V(M^{-1}y) dy=\partial^{n-k} F(0).
			\end{align*}
			Therefore, \eqref{eq:V_poly} follows due to the identity 
			\begin{align*}
				|\det(M)|^{-1}\int_{\R^d}(Ry)^{n-k} h_V(M^{-1}y) dy = \partial^{n-k}_{\star} F(0).
			\end{align*}
			Coming back to the proof of the lemma, by \eqref{eq:V_poly} we have
			\begin{align*}
				V\mathcal{H}_n(x)&=\frac{2^{\frac{|n|}{2}}}{\sqrt{n!}}\sum_{0\le k\le n}(-\i)^{|k|}\dbinom{n}{k} V(p_{n-k}\circ RM) \partial^{k}_\star(e^{-\Psi_\infty\circ M^*})(0) \\
				&=\frac{2^{\frac{|n|}{2}}}{\sqrt{n!}}\sum_{0\le k\le n}\sum_{0\le r\le n-k}(-\i)^{|k|+|r|}\dbinom{n}{k}\dbinom{n-k}{r} \\
				& \ \ \times p_{n-k-r}(Rx)\partial^r_\star F(0) \partial^{k}_\star(e^{-\Psi_\infty\circ M^*})(0).
			\end{align*}
			Noting that 
			\begin{align*}
				\dbinom{n}{k}\dbinom{n-k}{r}=\frac{n!}{k! r! (n-k-r)!}=\dbinom{n}{k+r}\dbinom{k+r}{k},
			\end{align*}
			and substituting $k+r=l$ we can write 
			\begin{align*}
					V\mathcal{H}_n(x)
				&=\frac{2^{\frac{|n|}{2}}}{\sqrt{n!}}\sum_{0\le l\le n} (-\i)^{|l|}\dbinom{n}{l} p_{n-l}(Rx) \\
				&\ \ \times \sum_{0\le k\le l}\dbinom{l}{k} \partial^{l-k}_\star F(0) \partial^{k}_\star(e^{-\Psi_\infty\circ M^*})(0) \\
				&=\frac{2^{\frac{|n|}{2}}}{\sqrt{n!}}\sum_{0\le l\le n} (-\i)^{|l|}\dbinom{n}{l} p_{n-l}(Rx)\partial^{l}_\star\left(F e^{-\Psi_\infty\circ M^*}\right)(0) \\
				&= \frac{2^{\frac{|n|}{2}}}{\sqrt{n!}}\sum_{0\le l\le n} (-\i)^{|l|}\dbinom{n}{l} p_{n-l}(Rx)\partial^{l}_\star\left(e^{-\frac{1}{4} |D\xi|^2}\right)(0) \\
				& = H_n(x),
			\end{align*}
			where the last identity follows from Proposition~\ref{prop:eigen_integral} for the diffusion OU operator $A_\varrho$ defined in \eqref{eq:A_rho}. This completes the proof of the lemma.
		\end{proof}
		\subsection{Relaxing H\ref{non-degeneracy} and H\ref{assumption3} by approximation}\label{sec:approx_epsilon}
		We now proceed to the proofs of Theorem~\ref{it:3} and Theorem~\ref{thm:spectral_expansion}. We note that Proposition~\ref{prop:eigen_integral} provides the proof of \eqref{it:poly} and \eqref{eq:L^2_norm} in Theorem~\ref{it:3} when H\ref{assumption3} holds. On the other hand, Proposition~\ref{prop:co-eigen-2} proves Theorem~\ref{thm:spectral_expansion} under the restriction that H\ref{non-degeneracy} holds. To relax the assumptions H\ref{assumption3} and H\ref{non-degeneracy}, we need to consider the following perturbation of L\'evy-OU operators.
		
		For any $\varepsilon>0$, let us define 
		\begin{align}\label{eq:A_ep}
			A^{(\varepsilon)} = A+\frac{\varepsilon}{2}\Delta,
		\end{align}
		where $A$ is defined in \eqref{eq:OU_gen}. We note that $A^{(\varepsilon)}$ generates a L\'evy-OU semigroup denoted by $P^{(\varepsilon)}$ with invariant distribution $\mu_\varepsilon$ such that 
		\begin{align}\label{eq:Psi_ep}
			\int_{\R^d}\mu^\varepsilon(x) e^{\i\langle \xi, x\rangle} dx=e^{\Psi^\varepsilon_\infty(\xi)}:=\exp\left(\Psi_\infty(\xi)-\frac{\varepsilon}{2}\int_0^\infty |e^{tB^*}\xi|^2 dt\right).
		\end{align}
		 As $\varepsilon>0$, $A^{(\varepsilon)}$ satisfies H\ref{non-degeneracy}. Let $\Lambda_\varepsilon$ be the Fourier multiplier operator on $\L^2(\R^d)$ defined by
		\begin{align}\label{eq:Lambda_epsilon}
			\mathcal{F}_{\Lambda_\varepsilon f}(\xi)=\exp\left(-\frac{\varepsilon}{2}\int_0^\infty |e^{tB^*}\xi|^2 dt\right)\mathcal{F}_f(\xi).
		\end{align}
		We note the following intertwining relationship which can be proved using the same argument as in the proof of Proposition~\ref{prop:left_int}.
		\begin{lemma} \label{lem:intertwining_ep} For any $\varepsilon, t>0$ and $p\ge 1$,
		\begin{align*}
			P_t\Lambda_\varepsilon =\Lambda_\varepsilon P^{(\varepsilon)}_t \quad \mbox{on \quad $\L^p(\mu^\varepsilon)$}.
		\end{align*}
		\end{lemma}
		\begin{lemma}\label{lem:another_intertwining}
			Assume that $A$ satisfies H\ref{assumption_polynomial} and H\ref{hartman-winter}, and let $\mathcal{H}^\varepsilon_n$ (resp. $\mathcal{G}^\varepsilon_n$) denote the eigenfunction (resp. co-eigenfunction) of $A^\varepsilon$ defined in \eqref{eq:eigen_polynomial} and \eqref{eq:co-eigen} respectively. Then for all $n\in\mathbb{N}^d_0$ and $\varepsilon>0$,
			\begin{align}
				\Lambda_\varepsilon \mathcal{H}^\varepsilon_n & = \mathcal{H}_n, \label{eq:Lambda_ep_H} \\
				\Lambda^*_\varepsilon \mathcal{G}_n & =\mathcal{G}^\varepsilon_n. \label{eq:Lambda_ep_G}
			\end{align}
		\end{lemma}
		\begin{proof}
			To prove \eqref{eq:Lambda_ep_H}, a similar calculation as in the proof of Lemma~\ref{lem:eigen_mapping} yields
			\begin{align*}
				\Lambda_\varepsilon (p_n\circ RM)(x)=\sum_{0\le k\le n} (-\i)^{|k|}\dbinom{n}{k} p_{n-k}(Rx)\partial^{k}_\star e^{-G_\varepsilon(M^*\xi)},
			\end{align*}
			where 
			\begin{align}\label{eq:G_ep}
				G_\varepsilon(\xi)=\frac{\varepsilon}{2}\int_0^\infty |e^{tB^*}\xi|^2 dt.
			\end{align}
			Since $\Psi_\infty(\xi)=\Psi^\varepsilon_\infty(\xi)+G_\varepsilon(\xi)$ for all $\xi\in\R^d$, see \eqref{eq:Psi_ep}, proof of \eqref{eq:Lambda_ep_H} follows by an argument similar to the proof of Lemma~\ref{lem:eigen_mapping}.
			
			Let us now prove \eqref{eq:Lambda_ep_G}. From the intertwining relationship in Lemma~\ref{lem:intertwining_ep}, we note that $\widehat{\Lambda}_\varepsilon \mu=\mu^\varepsilon$, where $\widehat{\Lambda}_\varepsilon$ denotes the $\L^2(\R^d)$-adjoint of $\Lambda_\varepsilon$. Moreover, $\widehat{\Lambda}_\varepsilon$ is also a Fourier multiplier operator on $\L^2(\R^d)$. In fact, as the multiplier function in \eqref{eq:Lambda_epsilon} is real-valued, $\widehat{\Lambda}_\varepsilon = \Lambda_\varepsilon$. Also, by \eqref{eq:Lambda_epsilon}, $\Lambda_\varepsilon$ is also a convolution operator defined as 
			\begin{align*}
				\Lambda_\varepsilon f(x)=\int_{\R^d} f(x-y)h^\varepsilon(y) dy,
			\end{align*}
			where $\mathcal{F}_{h^\varepsilon}(\xi)=\exp\left(-\frac{\varepsilon}{2}\int_0^\infty |e^{tB^*}\xi|^2 dt\right)$. Let us now describe how $\Lambda_\varepsilon$ behaves under the similarity transform $S_M$, where $M$ is an invertible matrix. Using the convolution operator representation above, we get 
			\begin{equation}\label{eq:Lambda_S_M}
			\begin{aligned}
				\Lambda_\varepsilon S_M f(x)&=\int_{\R^d} f(Mx-My) h^\varepsilon(y)dy \\
				&=|\det(M)|^{-1}\int_{\R^d} f(Mx -y) h^\varepsilon(M^{-1}y)dy \\
				&= S_M \Lambda^M_\varepsilon f(x),
			\end{aligned}
			\end{equation}
			where $\Lambda^M_\varepsilon f(x)=|\det(M)|^{-1}\int_{\R^d} f(x-y) h^\varepsilon(M^{-1}y)dy$. Due to the condition H\ref{hartman-winter}, $\partial^n_\star \mu, \partial^n_\star \mu^\varepsilon\in \L^2(\R^d)$ for any $n\in\mathbb{N}^d_0$. Therefore, 
			\begin{align}\label{eq:commute}
				\Lambda^M_\varepsilon \partial^n_\star \mu =\partial^n_\star \Lambda^M_\varepsilon \mu \quad \mbox{for all $n\in\mathbb{N}^d_0$}.
			\end{align}
			Combining \eqref{eq:Lambda_S_M} and \eqref{eq:commute} we obtain 
			\begin{align*}
			&	\Lambda_\varepsilon(S_M \partial^n_\star S^{-1}_M \mu) = S_M \Lambda^M_\varepsilon \partial^n_\star S^{-1}_M \mu  \\
				&= S_M \partial^n_\star \Lambda^M_\varepsilon S^{-1}_M \mu  = S_M \partial^n_\star S^{-1}_M \Lambda_\varepsilon\mu \\
				& = S_M \partial^n_\star S^{-1}_M \mu^\varepsilon.
			\end{align*}
			Hence, 
			\begin{align*}
				\Lambda^*_\varepsilon \mathcal{G}_n(x)& = \frac{\Lambda_\varepsilon(\mathcal{G}_n\mu)(x)}{\mu^\varepsilon(x)}=\frac{\Lambda_\varepsilon(S_M\partial^n_\star S^{-1}_M\mu)(x)}{\mu^\varepsilon(x)} \\
				&=\frac{S_M \partial^n_\star S^{-1}_M \mu^\varepsilon(x)}{\mu^\varepsilon(x)}=\mathcal{G}^\varepsilon_n(x).
			\end{align*}
			This completes the proof of \eqref{eq:Lambda_ep_G}.
		\end{proof}
		
		For any $\tau>0$, let us consider the L\'evy-OU operator with truncated L\'evy measure as follows:
		\begin{align*}
			A_\tau f(x)&=\frac{1}{2}\tr(\Sigma\nabla^2 f(x)) +\langle Bx, \nabla f(x)\rangle \\
			&+\int_{\R^d\setminus \{0\}}[f(x+y)-f(x)-\langle y, \nabla f(x)\rangle\mathbbm{1}_{\{|y|\le 1 \}}]\Pi_\tau(dy),
		\end{align*}
		where $\Pi_\tau =\Pi(\cdot \cap B_\tau(0))$, $B_\tau(0)$ being the ball of radius $\tau$ around $0$. Clearly, $\Pi_\tau$ satisfies H\ref{assumption3} for any $\tau>0$. Let $P^\tau$ denote the L\'evy-OU semigroup generated by $A_\tau$. We denote the invariant distribution of $P^\tau$ by $\mu_\tau$.
		\begin{lemma}\label{lem:moments_limit}
			Assume that $\Pi$ satisfies H\ref{assumption_polynomial}. Then, for any $n\in\mathbb{N}^d_0$, 
			\begin{align*}
				\lim_{\tau\to\infty} P^\tau_t p_n &= P_t p_n, \\
				\lim_{\tau\to\infty} \int_{\R^d} p_n  d\mu_\tau & = \int_{\R^d} p_n d\mu,
			\end{align*}
			where $p_n(x) = x^n$, and the first convergence holds pointwise.
		\end{lemma}
		\begin{proof}
			If $X_t$ (resp. $X^\tau_t$) denote the L\'evy-OU process corresponding to the triple $(\Sigma, B,\Pi)$ (resp. $(\Sigma, B, \Pi_\tau)$), then by \eqref{eq:Levy_rep} it follows that for all $0\le t< \infty$,
			\begin{align*}
				X_t - X^\tau_t=\int_0^t\int_{B_\tau(0)^c}  e^{(t-s)B} x  N(ds, dx)\overset{d}{=} \int_0^t\int_{B_\tau(0)^c}  e^{sB} x  N(ds, dx), 
			\end{align*}
			where $N(\cdot, \cdot)$ is the Poisson random measure associated with the L\'evy process $(Z_t)_{t\ge 0}$ in \eqref{eq:Levy_rep}, see \cite[Theorem~19.2]{SatoBook1999} for details. In particular, the distributional identity above also holds for $t=\infty$. For any $t>0$ we have
			\begin{align*}
				N([0,t),B_\tau(0)^c)\sim \text{Poi}(t\Pi(B_\tau(0)^c).
			\end{align*}
			Let us write 
			\begin{align*}
				Y^\tau_t=\int_0^t \int_{B_\tau(0)^c} |x|N(ds, dx).
			\end{align*}
			Then $Y^\tau$ is a compound Poisson process such that 
			\begin{align*}
				Y^\tau_t=\sum_{i=1}^{N^\tau_t} |J_i|,
			\end{align*}
			where $N^\tau_{t}$ is the Poisson process with rate parameter $\Pi(B_\tau(0)^c)$, and $(J_i)$ is a sequence of i.i.d. random variables satisfying
			\begin{align*}
				\mathbb{P}(J_i\in dx)=\frac{\mathbbm{1}_{\{|x|>\tau\}}\Pi(dx)}{\Pi(B_\tau(0)^c)},
			\end{align*}
			and are independent of $(N^\tau_t)_{t\ge 0}$. For any $k\ge 1$ we have
			\begin{align*}
				\mathbb{E} |Y^\tau_t|^k\le \mathbb{E}\left[(N^\tau_t)^{k-1}\sum_{i=1}^{N^\tau_t} |J_i|^k\right]=\mathbb{E}[(N^\tau_t)^k]\mathbb{E}[|J_1|^k].
			\end{align*}
			Now, 
			\begin{align*}
				\mathbb{E}|J_1|^k=\frac{\int_{B_\tau(0)^c} |x|^k \Pi(dx)}{\Pi(B_\tau(0)^c)}, \ \text{and} \quad \mathbb{E}[(N^\tau_t)^k]\le C_k t \Pi(B_\tau(0)^c)
			\end{align*}
			for some constant $C_k>0$, as $\tau\to \infty$. Combining the two inequalities above one gets
			\begin{align}\label{eq:Y_k_bound}
				\mathbb{E}[(Y^\tau_t)^k]\le C_k \int_{B_\tau(0)^c} |x|^k \Pi(dx).
			\end{align}
			 Since H\ref{assumption_polynomial} holds, 
			\begin{align*}
				\lim_{\tau\to\infty}\int_{B_\tau(0)^c} |x|^k \Pi(dx)=0 \quad \mbox{for all $k\ge 1$}.
			\end{align*}
			Therefore, writing 
			\begin{align*}
				M_t=\sup_{0\le s\le t} \|e^{sB}\|, \quad t\in [0,\infty),
			\end{align*}
			we have 	
			\begin{align*}
				\lim_{\tau\to\infty}\mathbb{E}_x|X_t-X^\tau_t|^k=M_t^k\lim_{\tau\to\infty}\mathbb{E}|Y^\tau_t|^k=0.	
			\end{align*}
			Since for any multi-index $n\in \mathbb{N}^d_0$,
			\begin{align*}
				P_tp_n(x)=\mathbb{E}_x[(X_t)^n], \quad P^\tau_t p_n(x)=\mathbb{E}_x[(X^\tau_t)^n],
			\end{align*}
			proof of the first assertion in the lemma follows.
			
			Let us now prove the second limit. As before, we have
			\begin{align*}
				X_\infty-X^\tau_\infty=\int_0^\infty \int_{B_\tau(0)^c} e^{sB}x N(ds, dx)=:\sum_{j=1}^\infty D^\tau_j,
			\end{align*}
			where 
			\begin{align*}
				D^\tau_j = \int_{j}^{j+1} \int_{B_\tau(0)^c} e^{sB}x N(ds, dx).
			\end{align*}
			Let us also write 
			\begin{align*}
				Y^{\tau}_j:=\int_j^{j+1}\int_{B_\tau(0)^c} |x| N(ds, dx).
			\end{align*}
			As $B$ has eigenvalues with negative real parts, there exits $\omega>0$ so that 
			\begin{align*}
				\sup_{j\le s\le j+1} \|e^{sB}\|\le Ce^{-j\omega} \quad \mbox{for all $j\ge 1$},
			\end{align*}
			and as a result,
			\begin{align*}
				|D^\tau_j|\le Ce^{-j\omega} Y^\tau_j \quad \mbox{for all $j\ge 1$}.
			\end{align*}
			Also, recalling that
			\begin{align*}
				N([j,j+1), B_\tau(0)^c)\sim \mathrm{Poi}(\Pi(B_\tau(0)^c)),
			\end{align*}
			the same argument used to prove \eqref{eq:Y_k_bound} leads to
			\begin{align*}
				\mathbb{E}\left[(Y^\tau_j)^k\right]\le C_k \int_{B_\tau(0)^c} |x|^k \Pi(dx), \quad \mbox{as $\tau\to\infty$},
			\end{align*}
			for some fixed constant $C_k$ depending only on $k$. Therefore using Jensen's inequality it follows that, 
			\begin{align*}
				|X_\infty - X^\tau_\infty|^k\le \left(\sum_{j=1}^\infty Ce^{-j\omega} Y^\tau_j\right)^k\le \frac{C^k}{1-e^{-\omega}}\sum_{j=1}^k e^{-j\omega} (Y^\tau_j)^k.
			\end{align*}
			Taking expectation of the above bound, we obtain
			\begin{align*}
				\mathbb{E}|X_\infty - X^k_\infty|^k&\le \sum_{j=1}^\infty \frac{C^k}{1-e^{-\omega}}\sum_{j=1}^\infty e^{-j\omega} C_k  \int_{B_\tau(0)^c} |x|^k \Pi(dx)\\
				 &=\frac{C^k C_k}{(1-e^{-\omega})^2} \int_{B_\tau(0)^c} |x|^k \Pi(dx)
			\end{align*}
			Since H\ref{assumption_polynomial} holds, the above inequality implies that
			\begin{align*}
				\lim_{\tau\to\infty}\mathbb{E}|X_\infty - X^\tau_\infty|^k=0.
			\end{align*}
			Since for any multi-index $n\in\mathbb{N}^d_0$,
			\begin{align*}
				\int_{\R^d} x^n\mu(dx)=\mathbb{E}[X_\infty^n], \quad \int_{\R^d} x^n\mu_\tau(dx)=\mathbb{E}[(X^\tau_\infty)^n]
			\end{align*}
			proof of the lemma follows.
		\end{proof}

			\begin{proof}[Proof of Theorem~\ref{it:3}] By Proposition~\ref{prop:eigen_integral}, for any $\tau>0$ we have 
				\begin{align}\label{eq:eigen_eta}
					P^\tau_t\mathcal{H}^\tau_n = e^{-t\langle n,\lambda\rangle}\mathcal{H}^\tau_n
				\end{align}
				for all $n\in\mathbb{N}^d_0$. Since $\partial^k_\star (e^{-\Psi^\tau_\infty\circ M^*})(0)$ is a linear combination of moments of $\mu_\tau$, by Lemma~\ref{lem:moments_limit}, we have for any $k\in\mathbb{N}^d_0$,
				\begin{align*}
					\lim_{\tau\to\infty} \partial^k_\star(e^{-\Psi^\tau_\infty\circ M^*})(0)= \partial^k_\star(e^{-\Psi_\infty\circ M^*})(0).
				\end{align*}
				Therefore, applying Lemma~\ref{lem:moments_limit} once again, 
				\begin{align*}
					\lim_{\tau\to\infty} P^\tau_t \mathcal{H}^\tau_n&=\frac{2^{\frac{|n|}{2}}}{\sqrt{n!}}\sum_{0\le k\le n} (-\i)^{|k|} \dbinom{n}{k}\lim_{\tau\to\infty} P^\tau_t (p_{n-k}\circ RM)\partial^{n-k}_\star (e^{-\Psi^\tau_\infty\circ M^*})(0) \\
					&=\frac{2^{\frac{|n|}{2}}}{\sqrt{n!}}\sum_{0\le k\le n} (-\i)^{|k|} \dbinom{n}{k}P_t (p_{n-k}\circ RM)\partial^{n-k}_\star (e^{-\Psi_\infty\circ M^*})(0) \\
					& = P_t\mathcal{H}_n,
				\end{align*}
				and 
				\begin{align*}
					\lim_{\tau\to\infty} \mathcal{H}^\tau_n =\mathcal{H}_n.
				\end{align*}
				Therefore, by \eqref{eq:eigen_eta}, $P_t\mathcal{H}_n = e^{-t\langle n,\lambda\rangle}\mathcal{H}_n$, which completes the proof of \eqref{it:poly}.
				
				Next, we note that $\Lambda_\varepsilon$ defined in Lemma~\ref{lem:another_intertwining} is bijective on the space of polynomials. Therefore, by \eqref{eq:Lambda_ep_H}, 
				\begin{align}\label{eq:span_2}
					\Span\{\mathcal{H}_n: n\in\mathbb{N}^d_0\}=\Span\{\mathcal{H}^\varepsilon_n: n\in\mathbb{N}^d_0\}.
				\end{align}
				 On the other hand, by Lemma~\ref{lem:eigen_mapping}, since $V$ defined in Theorem~\ref{thm:intertwining_diagonal} is also bijective on the space of polynomials, we have 
				 \begin{align*}
				 	\Span\{\mathcal{H}^\varepsilon_n: n\in\mathbb{N}^d_0\}=\Span\{H_n: n\in\mathbb{N}^d_0\},
				 \end{align*}
				  where $(H_n)$ is the orthonormal sequence of Hermite-It\^o-Laguerre polynomials defined in \eqref{eq:ito-hermite}. Since
				 \begin{align*} 
				 	\Span\{H_n: n\in\mathbb{N}^d_0\}=\mathscr{P},
				 \end{align*}
				 by \eqref{eq:span_2} we conclude that $\Span\{\mathcal{H}_n: n\in\mathbb{N}^d_0\}=\mathscr{P}$. 
				 
				 To prove \eqref{eq:L^2_norm}, we note that $\|\mathcal{H}_n\|^2_{\L^2(\mu)}$ is a linear combination of moments of the invariant distribution $\mu$. On the other hand, since $\Pi_\tau$ satisfies H\ref{assumption3}, by \eqref{eq:L2_norm_2} in Proposition~\ref{prop:eigen_integral}, we have 
				 	\begin{equation} 
				 	\begin{aligned}
				 		&2^{-|n|}n!\|\mathcal{H}^\tau_n\|^2_{\L^2(\mu_\tau)} \\
				 		&=
				 		\left.\partial^n_{\star,z}\overline{\partial}^n_{\star,w}\exp(\Psi^\tau_\infty( M^*(z-w))-\Psi^\tau_\infty( M^* z)-\overline{\Psi^\tau_\infty( M^* w)})\right|_{z=w=0}.
				 	\end{aligned}
				 \end{equation}
				 We also note that the right hand side of the above equation is a linear combination of moments of $\mu_\tau$. Therefore, by Lemma~\ref{lem:moments_limit}, we conclude the proof of \eqref{eq:L^2_norm}. This completes the proof of \eqref{it:norm}.
				 
				 Finally, when H\ref{non-degeneracy} holds, by Theorem~\ref{thm:spectrum}\eqref{it:poly_eigen}, the eigenspace $E_{-\langle n,\lambda\rangle}$ consists of polynomials. By Theorem~\ref{thm:intertwining_diagonal} and Lemma~\ref{lem:eigen_mapping}, as $V$ is bijective on the space of polynomials, we conclude that the set $\{\mathcal{H}_m: \langle m,\lambda\rangle =\langle n,\lambda\rangle\}$ forms a basis of $E_{-\langle n,\lambda\rangle}$. This proves \eqref{it:eigen_basis}. Hence the proof of the theorem is now complete.
			\end{proof}
			To prove Theorem~\ref{thm:spectral_expansion}, we need the following lemma.
			\begin{lemma}\label{lem:adj_fourier}
				Let $\widehat{P}_t$ denote the $\L^2(\R^d)$-adjoint of $P_t$. Then for any $f\in\L^2(\R^d)$, 
				\begin{align*}
					\mathcal{F}_{\widehat{P}_t f}(\xi)= e^{\Psi_t(\xi)}\mathcal{F}_{f}(e^{tB^*}\xi)
				\end{align*}
				for all $\xi\in\mathbb{R}^d$.
			\end{lemma}
			\begin{proof}
				The proof follows directly from \eqref{eq:characteristic_fn}.
			\end{proof}
		\begin{proof}[Proof of Theorem~\ref{thm:spectral_expansion}] Since $A^{(\varepsilon)}$ satisfies H\ref{non-degeneracy}, by Proposition~\ref{prop:co-eigen-2}, for any $\varepsilon>0$ we have
			\begin{align}\label{eq:co-eigen_ep}
				P^{(\varepsilon)*}_t \mathcal{G}^\varepsilon_n = e^{-t\overline{\langle n,\lambda\rangle}}\mathcal{G}^\varepsilon_n,
			\end{align}
			where $\mathcal{G}^\varepsilon_n \mu^\varepsilon=S_M \partial^n_\star S^{-1}_M\mu^\varepsilon$.
			Denoting the $\L^2(\R^d)$-adjoint of $P^{(\varepsilon)}_t$ by $\widehat{P}^{(\varepsilon)}_t$, we also have 
			\begin{align}\label{eq:coeigen_ep}
					P^{(\varepsilon)*}_t \mathcal{G}^\varepsilon_n(x)=\frac{\widehat{P}^{(\varepsilon)}_t(\mathcal{G}^\varepsilon_n \mu^\varepsilon)(x)}{\mu^\varepsilon(x)}.
			\end{align}
			We are going to argue that one can take $\varepsilon\to 0$ in \eqref{eq:co-eigen_ep}.
		Using the identity that $\mathcal{F}_{S_M\partial^n_\star S^{-1}_M f}(\xi)=(\i (M^{-1})^{*}\overline{R}\xi)^n\mathcal{F}_f(\xi)$ if $f$ is regular enough, where $R$ is defined in Notation~\ref{not:R}, we obtain
			\begin{align*}
				\mathcal{F}_{\mathcal{G}^\varepsilon_n \mu^\varepsilon}(\xi)=\left(\i (M^{-1})^*\overline{R}\xi\right)^n \mathcal{F}_{\mu^\varepsilon}(\xi)=\left(\i (M^{-1})^*\overline{R}\xi\right)^n e^{\Psi^\varepsilon_\infty(\xi)}.
			\end{align*}
			Therefore by Lemma~\ref{lem:adj_fourier}, 
			\begin{equation}\label{eq:P_hat_coeig}
			\begin{aligned}
				\mathcal{F}_{\widehat{P}^{(\varepsilon)}_t(\mathcal{G}^\varepsilon_n \mu^\varepsilon)}(\xi)&= e^{\Psi^\varepsilon_t(\xi)}\left(\i (M^{-1})^* \overline{R}e^{tB^*}\xi\right)^n e^{\Psi^\varepsilon_\infty\left(e^{tB^*}\xi\right)} \\
				&=\left(\i (M^{-1})^* \overline{R}e^{tB^*}\xi\right)^n e^{\Psi^\varepsilon_\infty\left(\xi\right)}, 
			\end{aligned}
			\end{equation}
			where the last identity follows from Lemma~\ref{lem:Psi_t}. Note that the above identity also holds when $\varepsilon=0$. Since $\Psi^\varepsilon_\infty(\xi)=\Psi_\infty(\xi)-G_\varepsilon(\xi)$ where $G_\varepsilon$ is defined in \eqref{eq:G_ep}, it follows that 
			\begin{align*}
				\left|e^{\Psi^\varepsilon_\infty(\xi)}-e^{\Psi_\infty(\xi)}\right|=e^{\Re(\Psi_\infty(\xi))}\left|1-e^{-G_\varepsilon(\xi)}\right|\le 2e^{\Re(\Psi_\infty(\xi))}.
			\end{align*}
			Also, $\lim_{\varepsilon\to 0}\Psi^\varepsilon_\infty(\xi)=\Psi_\infty(\xi)$. Since $\Psi_\infty$ satisfies H\ref{hartman-winter}, by dominated convergence theorem we get
			\begin{align*}
				\lim_{\varepsilon\to 0}\int_{\R^d}\left|\left(\i\xi\right)^n(e^{\Psi^\varepsilon_\infty(\xi)}- e^{\Psi_\infty(\xi)})\right| d\xi &=0 \\
				\lim_{\varepsilon\to 0}\int_{\R^d}\left|\left(\i (M^{-1})^* \overline{R} e^{tB^*}\xi\right)^n(e^{\Psi^\varepsilon_\infty(\xi)}- e^{\Psi_\infty(\xi)})\right| d\xi &=0
			\end{align*}
			for any $n\in\mathbb{N}^d_0$. Hence, \eqref{eq:P_hat_coeig} implies that for every $n\in\mathbb{N}^d_0$ and $x\in\R^d$,
			\begin{align*}
				\lim_{\varepsilon\to 0} \widehat{P}^{(\varepsilon)}_t(\mathcal{G}^\varepsilon_n \mu^\varepsilon)(x)&=\widehat{P}_t (\mathcal{G}_n\mu)(x), \\
					\lim_{\varepsilon\to 0} \mathcal{G}^\varepsilon(x)\mu^\varepsilon(x)&=\mathcal{G}_n(x)\mu(x), \\
					\lim_{\varepsilon\to 0} \mu^\varepsilon(x)&=\mu(x).
			\end{align*}
			Therefore, by \eqref{eq:co-eigen_ep}, \eqref{eq:coeigen_ep} and letting $\varepsilon\to 0$, we conclude that under H\ref{hartman-winter},
			\begin{equation}\label{eq:coeigen-pointwise}
			\begin{aligned}
				P^*_t\mathcal{G}_n(x)& =\lim_{\varepsilon\to 0} P^{(\varepsilon)*}_t\mathcal{G}^\varepsilon_n(x) \\
				&= e^{-t\overline{\langle n,\lambda\rangle}}\lim_{\varepsilon\to 0}\mathcal{G}^\varepsilon_n(x) \\
				&=e^{-t\overline{\langle n,\lambda\rangle}}\mathcal{G}_n(x)
			\end{aligned}
			\end{equation}
			for every $x\in\mathbb{R}^d$. By Proposition~\ref{prop:co-eigen-2}, $\mathcal{G}^\varepsilon_n\in\L^2(\mu^\varepsilon)$ and $\|\mathcal{G}^{\varepsilon}_n\|_{\L^2(\mu^\varepsilon)}\le 1$ for all $n\in\mathbb{N}^d_0$. Therefore by Fatou's lemma, 
			\begin{align*}
				\int_{\R^d} |\mathcal{G}_n(x)|^2 \mu(x)dx\le \liminf_{\varepsilon\to 0} \int_{\R^d} |\mathcal{G}^\varepsilon_n(x)|^2 \mu^\varepsilon(x) dx\le 1,
			\end{align*}
			which shows that $\mathcal{G}_n\in\L^2(\mu)$ for all $n\in\mathbb{N}^d_0$. Hence, \eqref{eq:coeigen-pointwise} holds on $\L^2(\mu)$, and the proof of the theorem is completed by spectral mapping theorem.
		\end{proof}
	
	\begin{proof}[Proof of Theorem~\ref{thm:biorthogonality}] Let us first assume that H\ref{assumption_polynomial} and H\ref{non-degeneracy} hold.
		Due to H\ref{assumption_polynomial}, we note that the intertwining operator $V$ in Theorem~\ref{thm:intertwining_diagonal} is invertible on $\mathscr{P}$. Therefore by Lemma~\ref{lem:eigen_mapping}, $\cc{H}_n=V^{-1} H_n$. Also by Proposition~\ref{prop:co-eigen-2}, $\mathcal{G}_n = V^* H_n$. Therefore for any $n,m\in\bb{N}^d_0$, 
		\begin{align*}
			\langle\cc{H}_n, \cc{G}_m\rangle_{\L^2(\mu)}&=\langle V^{-1} H_n, V^* H_m\rangle_{\L^2(\mu)} \\
			&=\langle H_n, H_m\rangle_{\L^2(\nu_\varrho)}\\
			&=\delta_{mn}.
		\end{align*}
To relax the assumption H\ref{non-degeneracy}, we consider the approximation $A^{(\epsilon)}$ defined in \eqref{eq:A_ep}. For any $m,n\in\mathbb{N}^d_0$, let $\mathcal{H}^\varepsilon_n$ and $\mathcal{G}^\varepsilon_m$ be defined as before. We claim that
\begin{align}\label{eq:lim_ip}
	\lim_{\varepsilon\to 0}\langle\mathcal{H}^\varepsilon_n, \mathcal{G}^\varepsilon_m\rangle_{\L^2(\mu^\varepsilon)}=\langle\mathcal{H}_n,\mathcal{G}_m\rangle_{\L^2(\mu)}.
\end{align}
Using the formula for $\mathcal{H}^\varepsilon_n$ and $\mathcal{G}^\varepsilon_m$ in Theorem~\ref{it:3} and Theorem~\ref{thm:spectral_expansion}, it suffices to prove that for any $n,m\in\mathbb{N}^d_0$,
\begin{equation}\label{eq:moment_lim_2}
\begin{aligned}
	&\lim_{\varepsilon\to 0}\partial^{n-k}_\star \left(e^{-\Psi^\varepsilon_\infty\circ M^*}\right)(0)\int_{\R^d}p_n(RMx)S_M\partial^m_\star S^{-1}_M \mu^{\varepsilon}(x) dx \\
	& \ \ \ =\partial^{n-k}_\star \left(e^{-\Psi_\infty\circ M^*}\right)(0)\int_{\R^d}p_n(RMx)S_M\partial^m_\star S^{-1}_M \mu(x) dx.
\end{aligned}
\end{equation}
Since $\mu^\varepsilon$ is the convolution of $\mu$ and a centered Gaussian distribution with covariance $\varepsilon \int_0^\infty e^{tB} e^{tB^*}dt$, it follows that 
\begin{align*}
	\lim_{\varepsilon\to 0}\int_{\R^d} p_n d\mu^\varepsilon = \int_{\R^d} p_n d\mu \quad \mbox{for all $n\in\mathbb{N}^d_0$}.
\end{align*}
Observing that $\partial^n_{\star}(e^{-\Psi_\infty\circ M^*})(0)$ is a linear combination of moments of $\mu$, we get
\begin{align*}
\lim_{\varepsilon\to 0}\partial^n_{\star}(e^{-\Psi^\varepsilon_\infty\circ M^*})(0)=\partial^n_{\star}(e^{-\Psi_\infty\circ M^*})(0) \quad \mbox{for every $n$}.
\end{align*}
 Also, as $\mu$ and $\mu^\varepsilon$ vanish at $\infty$, using integration by parts,
\begin{align*}
	\int_{\R^d}p_n(RMx)S_M\partial^m_\star S^{-1}_M \mu^{\varepsilon}(x) dx=(-1)^{|m|}\int_{\R^d}\partial^m_{\star}(p_n\circ R)(Mx)\mu^\varepsilon(x)dx.
\end{align*}
Letting $\varepsilon\to 0$ in the above equation yields \eqref{eq:moment_lim_2}, which further implies 
\begin{align*}
	\delta_{mn} = \lim_{\varepsilon\to 0}\langle\mathcal{H}^\varepsilon_n,\mathcal{G}^\varepsilon_m\rangle_{\L^2(\mu^\varepsilon)}=\langle\mathcal{H}_n,\mathcal{G}_m\rangle_{\L^2(\mu)}.
\end{align*}
This completes the proof of biorthogonality of $(\mathcal{H}_n)$ and $(\mathcal{G}_n)$. By Theorem~\ref{it:3}\eqref{it:norm}, we have $\Span\{\mathcal{H}_n:n\in\mathbb{N}^d_0\}=\mathscr{P}$. Assume that $\mathscr{P}$ is dense in $\L^2(\mu)$ and $(\mathcal{G}'_n)$ is a sequence biorthogonal to $(\mathcal{H}_n)$. Then for any $m,n\in\mathbb{N}^d_0$,
\begin{align*}
	\langle \mathcal{H}_n, \mathcal{G}_m -\mathcal{G}'_m\rangle_{\L^2(\mu)}=0.
\end{align*}
By density of $\mathscr{P}$ we conclude that $\mathcal{G}_m=\mathcal{G}'_m$. This completes the proof of the theorem.
\end{proof}
	
		\begin{proof}[Proof of Theorem~\ref{thm:diagonalizable}] Let us first assume that H\ref{assumption_polynomial} and H\ref{non-degeneracy} hold.
		 Since $h_n$ is an eigenfunction of $A_2$ with eigenvalue $-\langle n,\lambda\rangle$, by Proposition~\ref{prop:generalized_eigenfunction}, $h_n$ is a polynomial of degree at most $|n|$. Now, suppose that there exists an eigenvalue $\theta\in\bb{N}(\lambda)$ of $A_2$ such that $A_2$ has a generalized eigenfunction corresponding to $\theta$, that is, there exists $v\in\L^2(\mu)$ satisfying $(A_2-\theta I)^r v=0$ for some $r\ge 1$. By Theorem~\ref{thm:spectrum}\eqref{it:poly_eigen}, there exists $n\in\mathbb{N}^d_0$ such that $\theta=-\langle n,\lambda\rangle$ for some $n\in\bb{N}^d_0$. Then, for any $\theta'=-\langle m,\lambda\rangle$ with $\theta\neq \theta'$ we have 
		\begin{align*}
			\langle (A_2-\theta I)v,g_m\rangle_{\L^2(\mu)}&=\frac{1}{(\theta'-\theta)^{r-1}}\langle (A_2-\theta I)v, (A^*_2-\theta I)^{r-1}g_m\rangle_{\L^2(\mu)}\\
			&=\frac{1}{(\theta'-\theta)^{r-1}}\langle(A_2-\theta I)^r v, g_m\rangle_{\L^2(\mu)} \\
			&=0
		\end{align*}
		On the other hand, for any $n,m\in\bb{N}^d_0$ with $-\langle n,\lambda\rangle=-\langle m,\lambda\rangle=\theta$, we have 
		\begin{align*}
			\langle (A_2-\theta I)v, g_m\rangle_{\L^2(\mu)}=\langle v, (A^*_2-\theta I)g_m\rangle_{\L^2(\mu)}=0.
		\end{align*}
		Therefore, $\langle (A_2-\theta I)v, g_m\rangle_{\L^2(\mu)}=0$ for all $m\in\bb{N}^d_0$. Since by Proposition~\ref{prop:generalized_eigenfunction} $(A_2-\theta I)v$ is a polynomial and $\Span\{h_n: n\in\bb{N}^d_0\}=\mathscr{P}$, there exist constants $(c_i)_{1\le i\le k}$ and $n_1,\ldots, n_k\in\bb{N}^d_0$ such that 
		\begin{align*}
			(A_2-\theta I)v=\sum_{i=1}^k c_i h_{n_i}.
		\end{align*}
		Since $\langle (A_2-\theta I)v, g_n\rangle_{\L^2(\mu)}=0$ for all $n\in\bb{N}^d_0$ and $(h_n), (g_n)$ are biorthogonal, we conclude that $c_i=0$ for all $1\le i\le k$, which implies that $(A_2-\theta I)v=0$. Therefore, $v$ is an eigenfunction. As a result, we have $\mathtt{M}_a(\theta, A_2)=\mathtt{M}_g(\theta, A_2)$ for all $\theta\in\bb{N}(\lambda)$, which according to Theorem~\ref{thm:multiplicity} implies that $B$ is diagonalizable.
		
		Let us now replace H\ref{non-degeneracy} by H\ref{hartman-winter}. For $\varepsilon>0$, let $A^{(\varepsilon)}$ be defined by \eqref{eq:A_ep}. Also recall the operator $\Lambda_\varepsilon$ defined in \eqref{eq:Lambda_epsilon}. Since $(h_n)$ is a sequence of polynomials and $\Lambda_\varepsilon:\mathscr{P}\longrightarrow \mathscr{P}$ is bijective, by the intertwining relationship in Lemma~\ref{lem:intertwining_ep}, we obtain
		\begin{enumerate}[leftmargin=*]
		\item $A^{(\varepsilon)}_2 \Lambda^{-1}_\varepsilon h_n = -\langle n,\lambda\rangle \Lambda^{-1}_\varepsilon h_n$,
		\item	$A^{(\varepsilon)*}_2 \Lambda^*_\varepsilon g_n =  -\langle n,\lambda\rangle \Lambda^*_\varepsilon g_n$,
		\item $\langle\Lambda^{-1}_\varepsilon h_n, \Lambda^*_\varepsilon g_n\rangle_{\L^2(\mu^\varepsilon)}  = \langle h_n, g_m\rangle_{\L^2(\mu)}=\delta_{mn}$,
		\item $\Span\{\Lambda^{-1}_\varepsilon h_n: n\in\mathbb{N}^d_0\}=\mathscr{P}$.
		\end{enumerate}
		Since $A^{(\varepsilon)}$ satisfies H\ref{non-degeneracy} and it admits a biorthogonal sequence of eigenfunctions and co-eigenfunctions, by the previous argument, $B$ must be diagonalizable. This completes the proof of the theorem.
	\end{proof} 
	
	\begin{proposition}\label{prop:gaussian}
			Assume that H\ref{assumption_polynomial} and H\ref{hartman-winter} hold and $B$ is diagonalizable. Then, $(\mathcal{G}_n)_{n\in\mathbb{N}^d_0}$ is a sequence of polynomials if and only if $\Pi=0$.
	\end{proposition}
	\begin{proof}
		If $(\mathcal{G}_n)_{n\in\mathbb{N}^d_0}$ are polynomials, in particular, 
		\begin{align*}
			x\mapsto\frac{S_M\partial_{x_i} S^{-1}_M \mu(x)}{\mu(x)}
		\end{align*}
		is a polynomial for every $i=1,\ldots, d$. Also, note that $\mathcal{H}_0$ is a constant polynomial and by Theorem~\ref{thm:biorthogonality}, $\mathcal{G}_n$ is orthogonal to $\mathcal{H}_0$ for any $n\in\mathbb{N}^d_0$ with $|n|=1$. Hence, $\deg(\mathcal{G}_n)\ge 1$ whenever $|n|=1$. This shows that $\mu(x)= e^{P(x)}$ for some polynomial $P$ of degree at least $2$. Since $\mu$ is an infinitely divisible distribution, by \cite[Theorem~26.1(ii)]{SatoBook1999}, $P$ must be a quadratic polynomial, that is, $\mu$ has to be a Gaussian density. This is true only if $\Pi=0$. This completes the proof of the proposition.
	\end{proof}
	
	\begin{proof}[Proof of Theorem~\ref{thm:normality}] If $A_2$ is normal in $\L^2(\mu)$, then any eigenvalue $\theta$ of $A_2$ has the same algebraic and geometric multiplicities. Due to assumption H\ref{non-degeneracy}, by Theorem~\ref{thm:multiplicity}, $B$ is diagonalizable. Since $A_2$ is a normal operator, its eigenfunctions and co-eigenfunctions are identical up to multiplicative constants. Again, due to H\ref{non-degeneracy}, invoking Theorem~\ref{thm:spectrum}\eqref{it:poly_eigen} we conclude that the co-eigenfunctions of $A_2$ are polynomials. Therefore by Proposition~\ref{prop:gaussian} we conclude that $\Pi=0$, and we have
	\begin{align*}
		A_2 = \frac{1}{2}\tr(\Sigma\nabla^2)+\langle Bx, \nabla\rangle,
	\end{align*}
	and $A_2$ is normal in $\L^2(\mu)$. By \cite[Theorem~1]{Michalik1987}, the diffusion semigroup $P_t$ can be written as 
	\begin{align*}
		P_t=\Gamma(S^*_0(t)), \quad S_0(t)=\Sigma^{-\frac{1}{2}}_\infty e^{tB} \Sigma^{\frac{1}{2}}_\infty,
	\end{align*}
	where $\Gamma$ is the second quantization operator defined in \cite{Michalik1987}, and $S_0$ is the semigroup defined on $\R^d$. By \cite[Lemma~2]{Michalik1987} (see also \cite[Chapter~1]{Simon1974}), it follows that $P_t$ is normal in $\L^2(\mu)$ if and only if $S_0(t)$ is a normal semigroup in $\R^d$. The latter holds if and only if $\Sigma^{-1/2}_\infty B \Sigma^{1/2}_\infty$ is a normal matrix. This proves \eqref{it:normal}.
	
	When $B$ has real eigenvalues and $A_2$ is normal, by \eqref{it:normal}, $\Pi=0$, and therefore, $\sigma(A_2)=\mathbb{N}(\lambda)\subset \R$, where $\mathbb{N}(\lambda)$ is defined in \eqref{eq:N_lambda}. Hence, $A_2$ is self-adjoint. By \cite[Theorem~2.4]{MichalikGoldys2002}, this is equivalent to the condition $B\Sigma=\Sigma B^*$. This proves \eqref{it:sa}.
	\end{proof}

\section{Proof of results in \S\ref{sec:spectral_exp}}\label{sec:spect_exp}
We first obtain some estimates of the norm of the eigenfunctions in the diffusion case. Recall that the invariant distribution of the diffusion OU operator $A^{\mathrm{OU}}$ is given by 
\begin{align*}
	\nu(dx)=\frac{(2\pi)^{-\frac d2}}{\sqrt{\det(\Sigma_\infty)}} e^{-\frac{\langle \Sigma^{-1}_\infty x,x\rangle}{2}} dx.
\end{align*}
\begin{lemma}\label{lem:diff_norm}
	Assume that $\Pi=0$ and H\ref{non-degeneracy} holds. Then for any $n\in\mathbb{N}^d_0$, 
	\begin{align*}
		\|\mathcal{H}_n\|^2_{\L^2(\nu)}\le 2^d(2d\|M\Sigma_\infty M^*\|)^{|n|}.
	\end{align*}
\end{lemma}
\begin{proof}
	The L\'evy-Khintchine exponent of $\nu$ is given by 
	\begin{align*}
		\Psi_\infty(\xi)=-\frac{1}{2}\langle\Sigma_\infty\xi,\xi\rangle.
	\end{align*}
	Therefore by Theorem~\ref{it:3}\eqref{it:norm}, for any $n\in\mathbb{N}^d_0$,
	\begin{align*}
		\|\mathcal{H}_n\|^2_{\L^2(\nu)}=\frac{2^{|n|}}{n!}\left.\partial^n_{\star,z}\overline{\partial}^n_{\star,w} \exp(\langle M\Sigma_\infty M^* z,w\rangle)\right|_{z=w=0}.
	\end{align*}
	A simple but tedious computation yields that for any multi-index $n=(n_1,\ldots, n_d)$,
	\begin{align*}
		&\left.\partial^n_{z}\partial^n_w \exp(\langle M\Sigma_\infty M^* z, w\rangle)\right|_{z=w=0} \\
		& = (n!)^2\sum_{\substack{a_{ij}\ge 0 \\ \sum_{j}a_{ij}=n_i \\ \sum_{j}a_{ij}=n_j}}\prod_{i,j}\frac{(M\Sigma_\infty M^*)^{a_{ij}}_{ij}}{a_{ij}!}.
	\end{align*}
	Therefore,
	\begin{align*}
		\left.\frac{1}{n!}\partial^n_z\partial^n_w\exp(\langle M\Sigma_\infty M^* z, w\rangle)\right|_{z=w=0}&\le \prod_{i=1}^d \left(\sum_{\sum_{j}a_{ij}=n_i} \frac{n_i!}{\prod_{j}a_{ij}!}\|M\Sigma_\infty M^*\|^{n_i}\right) \\
		&=\prod_{i=1}^d (d \|M\Sigma_\infty M^*\|)^{n_i}=(d \|M\Sigma_\infty M^*\|)^{|n|}.
	\end{align*}
	Since $\partial^n_{\star,z}\overline{\partial}^n_{\star,w}$ is a linear combination of at most $2^{2d}$ derivatives $\partial^m_z\partial^m_w$ where $|m|=|n|$, the above estimates lead to the following
	\begin{align*}
		\|\mathcal{H}_n\|^2_{\L^2(\nu)}=\frac{2^{|n|}}{n!}\left.\partial^n_{\star,z}\overline{\partial}^n_{\star,w} \exp(\langle M\Sigma_\infty M^* z,w\rangle)\right|_{z=w=0}\le 2^{2d} (2d \|M\Sigma_\infty M^*\|)^{|n|}.
	\end{align*}
	This completes the proof of the lemma.
\end{proof}

\begin{proof}[Proof of Theorem~\ref{thm:spectral_exp}] Let us define $S_t$ on $\L^2(\mu)$ as 
	\begin{align*}
		S_t f=\sum_{n\in\bb{N}^d_0} e^{-t\langle n,\lambda\rangle}\langle f, \cc{G}_n\rangle_{\L^2(\mu)}\cc{H}_n.
	\end{align*}
	We first verify that $S_t$ extends as a bounded operator on $\L^2(\mu)$ for sufficiently large values of $t$. Let $t_0=\log(2d\|M\Sigma_\infty M^*\|)/\Re(\lambda_1)$. Indeed, for any $f\in\L^2(\mu)$, using Lemma~\ref{lem:diff_norm} along with Cauchy-Schwarz inequality we have 
	\begin{align}
		\|S_t f\|_{\L^2(\mu)}&\le \sum_{n\in\bb{N}^d_0} e^{-t\langle n,\lambda\rangle}|\langle f, \cc{G}_n\rangle_{\L^2(\mu)}| \|\cc{H}_n\|_{\L^2(\mu)} \nonumber \\
		&\le 2^{2d}\|f\|_{\L^2(\mu)} \sum_{n\in\bb{N}^d_0} e^{-t\langle n,\lambda\rangle}(2d\|M\Sigma_\infty M^*\|)^{|n|}, \label{eq:series}
	\end{align}
	where we used the fact that $\|\cc{G}_n\|_{\L^2(\mu)}=\|V^* H_n\|_{\L^2(\mu)}\le \|V\|\le 1$ for all $n\in\bb{N}^d_0$. Finally, if $t>t_0$, the series on the right hand side of \eqref{eq:series} converges absolutely, which proves that $S_t:\L^2(\mu)\longrightarrow\L^2(\mu)$ extends continuously for $t>t_0$. Now, for every $n\in\bb{N}^d_0$ and $t>t_0$,
	\begin{align*}
		P_t\cc{H}_n= S_t \cc{H}_n=e^{-t\langle n,\lambda\rangle}\cc{H}_n,
	\end{align*}
	which implies that $P_t$ and $S_t$ agree on $\mathscr{P}$. Using the density of $\mathscr{P}$ in $\L^2(\mu)$, we conclude the proof of the theorem.
\end{proof}
In the following two lemmas we provide pointwise upper bounds for the eigenfunctions and co-eigenfunctions of the L\'evy-OU semigroup. The eigenfunction estimates are obtained in the diffusion case. 
\begin{lemma}\label{prop:H_estimate}
	Suppose that the assumptions of Theorem~\ref{thm:spectral_exp} hold. Then, for all $\varepsilon>0$, there exists $b_\varepsilon>0$ such that for all $n\in\bb{N}^d_0$,
	\begin{align}\label{eq:H_estimate}
		\sup_{x\in\R^d} |\cc{H}_n(x)| e^{-\varepsilon |x|^2}\le b^{|n|}_\varepsilon.
	\end{align}
\end{lemma}
\begin{proof}
Since $\Pi=0$, $\Psi_\infty(z)=-\langle \Sigma_\infty z,z\rangle/2$ is an entire function on $\mathbb{C}^d$.
Therefore, in \eqref{eq:integral_rep} we can choose the sets $C_r$ with arbitrarily large $r$. For any $\beta>0$ and $n\in\bb{N}^d_0$, let us choose $r_j=\beta n_j^{\frac12}$ for $j=1,\ldots, d$. Then, we have 
\begin{align}\label{eq:H_n}
	\cc{H}_n(x)=\frac{\sqrt{2^{|n|}n!}}{(2\pi \i)^d}\int_{C_{r_1}}\cdots \int_{C_{r_d}}\frac{\exp(\langle Mx, \overline{R}^* z\rangle-\Psi_\infty(-\i M^*\overline{R}^* z))}{z^{n+1}} dz
\end{align}
Now, for all $z\in C_{r_1}\times\cdots\times C_{r_d}$, $|\Psi_\infty(-\i M^*\overline{R}^*z)|\le e^{\delta_1(\beta)|n|}$ such that $\delta_1(\beta)\to 0$ as $\beta\to 0$. Also, using Cauchy-Schwarz inequality, one has $|\langle Mx, \overline{R}^*z\rangle|\le \beta |n|^{\frac12}|Mx|\le \delta_2(\beta)(|x|^2+|n|)$, where $\delta_2(\beta)\to 0$ as $\beta\to 0$. Finally, $|z^* z|\le\beta^2 |n|$ for all $z\in C_{r_1}\times\cdots\times C_{r_d}$. Using Stirling formula, we have
\begin{align}\label{eq:stirling_estimate}
	\left|\frac{\sqrt{n!}}{z^{n+1}}\right|\le \beta^{-|n|-d}\prod_{j=1}^d n_j^{-\frac12} e^{-\frac{|n|}{2}}
\end{align}
Therefore, combining all estimates, \eqref{eq:H_n} yields 
\begin{align*}
	|\cc{H}_n(x)|\le \frac{2^{\frac{|n|}{2}}}{(2\pi)^{d}} \beta^{-|n|-d} e^{-\frac{|n|}{2}}\exp((\delta_1(\beta)+\delta_2(\beta)+\beta^2)|n|+\delta_2(\beta)|x|^2).
\end{align*}
The proof of \eqref{eq:H_estimate} is completed by choosing $\beta$ such that $\delta_2(\beta)<\epsilon$. 
\end{proof}

\begin{lemma}\label{prop:G_estimate}
			Assume that H\ref{non-degeneracy} holds. Then, there exists $c>0$ such that for all $n\in\bb{N}^d_0$,
			\begin{align}\label{eq:upper_bound_point}
				\sup_{y\in\R^d} |\cc{G}_n(y)\mu(y)|\le c^{|n|}.
			\end{align}
			Moreover, there exists a constant $c_1>0$ such that 
			\begin{align}\label{eq:lower_bound_norm}
				\|\mathcal{G}_n\|_{\L^2(\mu)}\ge c^{|n|}_1.
			\end{align}
		\end{lemma}
	\begin{proof}
		Let us write $\mathcal{V}_n(x)=\mathcal{G}_n(x)\mu(x)$. From the definition of $\mathcal{G}_n$ in \eqref{eq:co-eigen}, we have 
		\begin{align*}
			\mathcal{F}_{\mathcal{V}_n}(\xi)=\frac{1}{\sqrt{2^{|n|}n!}}\left(\i (M^{-1})^* \overline{R}\xi\right)^n e^{\Psi_\infty(\xi)}.
		\end{align*}
		Since $\Re(\Psi_\infty(\xi))\le -\frac{1}{2}\langle\Sigma_\infty \xi, \xi\rangle$ for all $\xi\in\R^d$, by Fourier inversion, the above expression leads to 
		\begin{align*}
			\|\mathcal{V}_n\|_{\L^\infty}\le \frac{\|M^{-1}\overline{R}\|^n}{\sqrt{2^{|n|}n!}}\int_{\R^d}|\xi|^{|n|} e^{-\frac12\langle\Sigma_\infty \xi,\xi\rangle} d\xi \le c^{|n|}\frac{\Gamma(\frac{|n|+d}{2})}{\sqrt{n!}}.
		\end{align*}
		Using Stirling's approximation we have 
		\begin{align*}
			\frac{\Gamma(\frac{|n|+d}{2})}{\sqrt{|n|!}}\lesssim |n|^{\frac{d}{2}-\frac{3}{4}} 2^{-\frac{|n|}{2}}.
		\end{align*}
		Using the bound $|n|!/n!\le d^{|n|}$, we conclude \eqref{eq:upper_bound_point}. On the other hand since $\mu$ is a bounded function such that $\|\mu\|_{\L^\infty}\le K$, using isometry of Fourier transform we obtain
		\begin{align}\label{eq:norm_bound_2}
			\|\mathcal{G}_n\|^2_{\L^2(\mu)}\ge \frac{1}{K}\|\mathcal{V}_n\|^2_{\L^2(\R^d)}=\frac{1}{K2^{|n|}n!}\int_{\R^d}\left|((M^{-1})^* \overline{R}\xi)^{2n}\right| e^{2\Re(\Psi_\infty(\xi))}d\xi.
		\end{align}
		Since $(M^{-1})^*\overline{R}$ is an invertible matrix, there exists $c_1>0$ and a cone $\mathcal{C}$ such that 
		\begin{align*}
			|(M^{-1})^* \overline{R}\xi|\ge c_1|\xi| \quad \mbox{for all $\xi\in\mathcal C$}.
		\end{align*}
		Also, due to H\ref{non-degeneracy}, there exists $c_2>0$ such that
		\begin{align*}
			\limsup_{|\xi|\to\infty}\frac{\Re(\Psi_\infty(\xi))}{|\xi|^2}<-c_2.
		\end{align*}
		Hence \eqref{eq:norm_bound_2} leads to 
		\begin{align*}
			\|\mathcal{G}_n\|^2_{\L^2(\mu)}&\gtrsim \frac{1}{2^{|n|}n!}\int_{\mathcal{C}\cap\{|\xi|>a\}} |\xi|^{2|n|} e^{-c_2|\xi|^2}d\xi\\
			&=\frac{1}{2^{|n|}n!}\omega_{d-1}(\mathcal C\cap \mathbb{S}^{d-1})\int_a^\infty r^{2|n|+d-1} e^{-c_2r^2}dr,
		\end{align*}
		where $\omega_{d-1}$ is the surface area of $\mathbb{S}^{d-1}$ and $a$ is some positive number. By Stirling's approximation, 
		\begin{align*}
			\int_a^\infty r^{2|n|+d-1} e^{-c_2r^2}dr\asymp c^{|n|}_2 \Gamma(|n|+1+d/2)
		\end{align*}
		for some $c_2>0$, which shows that 
		\begin{align*}
			\|\mathcal{G}_n\|^2_{\L^2(\mu)}\gtrsim\frac{c^{|n|}_2 \Gamma(|n|+1+d/2)}{2^{|n|}n!}\ge \left(\frac{c_2}{2}\right)^{|n|}.
		\end{align*}
		This proves \eqref{eq:lower_bound_norm}.
	\end{proof}

	\begin{proof}[Proof of Theorem~\ref{thm:heat_kernel}] When $\Pi=0$ and $B$ is diagonalizable, we first show that for large values of $t$, the infinite sum
		\begin{align}\label{eq:heat_exp}
			\wt{p}_t(x,y):=\sum_{n\in\bb{N}^d_0} e^{-t\langle n,\lambda\rangle}\cc{H}_n(x)\cc{G}_n(y)
		\end{align}
	converges absolutely for all $x,y\in\R^d$. For any fixed $x,y\in\R^d$, invoking Lemma~\ref{prop:H_estimate} and Lemma~\ref{prop:G_estimate}, we obtain that for all $n\in\bb{N}^d_0$,
	\begin{align*}
		|\cc{H}_n(x)|\le b^{|n|}_\epsilon e^{\epsilon |x|^2}, \quad |\cc{G}_n(y)|\le\frac{c^{|n|}}{\mu(y)}.
	\end{align*}
Therefore, for all $t> -\log(cb_\epsilon)/\Re(\lambda_1)$, the series in \eqref{eq:heat_exp} converges absolutely for all $x,y\in\R^d$. We note that for all $t>  -\log(cb_\epsilon)/\Re(\lambda_1)$ and $f\in\L^2(\mu)$,
\begin{align*}
	\int_{\R^d}\wt{p}_t(x,y) f(y)\mu(y) dy=\sum_{n\in\bb{N}^d_0} e^{-t\langle n,\lambda\rangle} \cc{H}_n(x)\langle f,\cc{G}_n\rangle_{\L^2(\mu)},
\end{align*}
and the right hand side above converges absolutely due to Lemma~\ref{prop:H_estimate} and the fact that $|\langle f, \cc{G}_n\rangle|\le\|f\|_{\L^2(\mu)}$ for all $n\in\bb{N}^d_0$. Therefore, 
\begin{align*}
\int_{\R^d}\wt{p}_t(x,y) f(y)\mu(y) dy=P_t f(x)
\end{align*}
for all $f\in\cc{B}_b(\R^d)\subset\L^2(\mu)$, which completes the proof of the theorem.
\end{proof}

\begin{proof}[Proof of Theorem~\ref{thm:non-existence}] 
	We start by noting that for any $u,v\in\R$, 
	\begin{align*}
		F(u,v)=\Psi_\infty(u-v)-\Psi_\infty(u)-\overline{\Psi_\infty(v)}=\sigma^2 uv + \int_{\R}(e^{\i ux}-1)(e^{-\i vx}-1)\Pi_\infty(dx).
	\end{align*}
	Since H\ref{assumption_polynomial} holds, for any $m,n\ge 1$, 
	\begin{align*}
		\partial^m_u \partial^n_v F(0,0)=\begin{cases}
			\sigma^2+c_2 & \mbox{if $m=n=1$} \\
			\i^{m-n}c_{m+n} & \mbox{if $(m,n)\neq (1,1)$},
		\end{cases}
	\end{align*}
		where $c_{j}=\int_\R x^j \Pi_\infty(dx)$.
		Using Fa\`a di Bruno formula for derivatives, we obtain
		\begin{align}\label{eq:faadibruno}
			\partial^n_u\partial^n_v e^F(0,0)=\sum_{k=1}^n\frac{1}{k!}\sum_{\substack{p_1+\cdots+p_k=n \\ q_1+\cdots+q_k=n \\ p_i,q_i\ge 1}}\frac{n!}{p_1!\cdots p_k!}\frac{n!}{q_1!\cdots q_k!}\prod_{j=1}^k c_{p_j+q_j}.
		\end{align}
		For any $j\ge 1$ we have
		\begin{align*}
			c_j&=\int_0^\infty\int_{\mathbb R} |x|^j\Pi\circ e^{sb}(dx) \\
			&=\int_0^\infty e^{-jsb}ds\int_{\R}|x|^j\Pi(dx) \\
			&\ge \frac{\gamma c^j}{bj},
		\end{align*}
		where $c=\inf\mathrm{Supp}(\Pi)>0$ and $\gamma=\Pi([c,\infty))<\infty$. Therefore, for any $p_j, q_j\ge 1$ with $\sum_{j=1}^k p_j=\sum_{i=j}^k q_i=n$,
		\begin{align}\label{eq:c_ineq}
			\prod_{j=1}^k c_{p_j+q_j}\ge \frac{c^{2n}}{\prod_{j=1}^k (p_j+q_j)}\ge c^{2n} \frac{\gamma^k k^k}{(2nb)^k},
		\end{align}
		where we used the inequality
		\begin{align*}
			\prod_{j=1}^k (p_j+q_j)\le \left(\frac{\sum_{j=1}^k (p_j+q_j)}{k}\right)^k.
		\end{align*}
		 Recalling the identity
		\begin{align*}
			\sum_{\substack{p_1+\cdots + p_k=n \\ p_i\ge 1}} \frac{n!}{p_1!\cdots p_k!}=k!S(n,k),
		\end{align*}
		where $S(n,k)$ is the \emph{Stirling number of second kind}, \eqref{eq:faadibruno} and \eqref{eq:c_ineq} lead to 
		\begin{align}
			\partial^n_u\partial^n_v e^F(0,0)\ge
				\sum_{k=1}^n k! k^k S(n,k)^2\left(\frac{\gamma}{b}\right)^{k}\frac{c^{2n}}{(2n)^k} &\ge \frac{(\frac{n}{2})! \left(\frac{n}{2}\right)^{\frac n2}}{(2n)^{n/2}} S(n,n/2)^2\left(\frac{\gamma}{b}\right)^{n/2}c^{2n} \nonumber\\
				&=\left(\frac{n}{2}\right)! 2^{-n}  S(n,n/2)^2c_2^{n}, \label{eq:Bell_LB}
		\end{align}
		where we have assumed that $n$ is even, and
		\begin{align*}
			c_2=c^2\left(\frac{\gamma}{b}\right)^{1/2}.
		\end{align*}
		Now, $S(n,k)$ counts the total number of partitions of $\{1,\ldots, n\}$ into $k$ parts. Therefore,
		\begin{align*}
			S(n,n/2)\ge \frac{n!}{2^{n/2}(n/2)!},
		\end{align*}
		as the right hand side counts the total number of partitions into $n/2$ pairs.
	  Hence, Theorem~\ref{it:3}\eqref{it:norm} together with \eqref{eq:Bell_LB} implies
		\begin{align}\label{eq:norm_LB}
			\|\mathcal{H}_n\|^2_{\L^2(\mu)}=\frac{2^n}{n!} \partial^n_u\partial^n_v e^F(0,0)\ge \frac{n!}{2^{n}(\frac n2)!} c^n_2.
		\end{align}
		By Stirling approximation, we get
		\begin{align*}
			\frac{n!}{2^{n}(\frac n2)!} c^n_2\gtrsim n^{\frac n2} \left(\frac{c_2}{2}\right)^n \quad \mbox{as $n\to \infty$}.
		\end{align*}
	 Now assume that there exists $t>0$ such that for all $f\in \L^2(\mu)$, the series 
		\begin{align*}
			\sum_{n=0}^\infty e^{-tbn}\langle f, \mathcal{G}_n\rangle_{\L^2(\mu)}\mathcal{H}_n
		\end{align*}
		is convergent in $\L^2(\mu)$. Then, for each $f\in\L^2(\mu)$, $\lim_{n\to\infty} e^{-tbn}\langle f, \mathcal{G}_n\rangle_{\L^2(\mu)}\mathcal{H}_n =0$. By uniform boundedness principle, the one-dimensional operators defined by
		\begin{align*}
			T_n f=e^{-tbn}\langle f,\mathcal{G}_n\rangle_{\L^2(\mu)}\mathcal{H}_n
		\end{align*}
		must be norm bounded with respect to $n$. Now, $\|T_n\|=e^{-tbn}\|\mathcal{G}_n\|_{\L^2(\mu)}\|\mathcal{H}_n\|_{\L^2(\mu)}$. By \eqref{eq:norm_LB} and \eqref{eq:lower_bound_norm} in Lemma~\ref{prop:G_estimate},
		\begin{align*}
			\|T_n\|\ge e^{-tbn}\frac{n!}{2^{\frac n2}(\frac n2)!} c^n_2 c^n_1\gtrsim e^{-tbn} n^{\frac n2} c^n_2 c^n_1 \to \infty \quad \text{as $n\to\infty$},
		\end{align*}
		for any $t>0$. This leads to a contradiction and hence the proof of the theorem is concluded.
\end{proof}

\section{Proof of results in \S\ref{sec:compactness}}\label{sec:compact} 

\begin{proof}[Proof of Theorem~\ref{thm:nec_compact}] 
	Suppose that
	\begin{align*}
	P_{t_0}:\L^p(\mu)\longrightarrow \L^p(\mu)
	\end{align*}
	is compact for some $t_0>0$ and $1<p<\infty$. We prove that $\mu$ has moments of all orders, and therefore, by Lemma~\ref{lem:moment} it would follow that H\ref{assumption_polynomial} holds. This will be achieved by proving that for any $\beta>0$, 
	\begin{align}\label{eq:beta_moment}
		\mu\{x: |x|>R\} \lesssim R^{-\beta} \quad \mbox{as $R\to\infty$}.
	\end{align}
	Since $P_{t_0}: \L^p(\mu)\longrightarrow \L^p(\mu)$ is compact, $P_t:\L^p(\mu)\longrightarrow \L^p(\mu)$ is compact for all $t\ge t_0$. Let us assume that there exist $a,c\in (0,1)$ and sufficiently large $t\ge t_0$ such that
	\begin{align}\label{eq:unif_prob_bound}
		P_t \mathbbm{1}_{\{|\cdot|>cR\}}\ge a \quad \mbox{for all $|x|>R$}.
	\end{align}
	Let us define the tail function
	\begin{align*}
		\overline{\mu}(R):=\mu\{x:|x|>R\}.
	\end{align*}
	Since $\mu$ is a non-degenerate infinitely divisible distribution, $\mu$ has unbounded support, and therefore,
	\begin{align*}
		\overline{\mu}(R)>0
	\end{align*}
	for arbitrarily large $R$. For sufficiently large $R$, define
	\begin{align*}
		f_R =\overline{\mu}(cR)^{-1/p}\, \mathbbm{1}_{\{|\,\cdot\,|>cR\}}.
	\end{align*}
	Then,
	\begin{align*}
		\|f_R\|_{\L^p(\mu)}=1.
	\end{align*}
	By \eqref{eq:unif_prob_bound},
	\begin{align*}
		P_tf_R(x)\ge a\overline{\mu}(cR)^{-1/p},\quad \mbox{for any $|x|>R$}.
	\end{align*}
	Hence
	\begin{align}\label{eq:P_T_bd2}
		\left\|\mathbbm{1}_{\{|\cdot|>R\}}P_tf_R \right\|_{\L^p(\mu)}^p \ge	a^p\frac{\overline{\mu}(R)}{\overline{\mu}(cR)}.
	\end{align}
	Since $P_t$ is compact,
	\begin{align*}
		\mathcal K:=\left\{P_tf:\|f\|_{\L^p(\mu)}\le 1\right\}
	\end{align*}
	is relatively compact in $\L^p(\mu)$. Therefore,
	\begin{align}\label{eq:compact_bound}
		\lim_{R\to\infty} \sup_{g\in\mathcal K}\left\|\mathbbm{1}_{\{|\cdot|>R\}}g\right\|_{\L^p(\mu)}=0.
	\end{align}
	Indeed, this follows immediately by covering the closure of $\mathcal K$ by finitely many $\L^p(\mu)$-balls and using
	\begin{align*}
		\lim_{R\to\infty}\|\mathbbm 1_{\{|\cdot|>R\}}g\|_{\L^p(\mu)}=0
	\end{align*}
	for each fixed $g\in \mathcal K$. Since $\|f_R\|_{\L^p(\mu)}=1$, equations \eqref{eq:P_T_bd2} and \eqref{eq:compact_bound} imply
	\begin{align}\label{eq:mu_bar_lim}
		\frac{\overline{\mu}(R)}{\overline{\mu}(cR)} \to 0 \quad\text{as }R\to\infty .
	\end{align}
	Then for any fixed $m>0$, there exists $R_0>0$ such that
	\begin{align*}
		\overline{\mu}(R/c)\le c^{m}\overline{\mu}(R), \quad \mbox{for all $R\ge R_0$}{}.
	\end{align*}
	Iterating,
	\begin{align*}
		\overline{\mu}(c^{-n}R_0)
		\le
		c^{nm}\overline{\mu}(R_0) \quad \mbox{for all $n\ge 1$}.
	\end{align*}
	If
	\begin{align*}
		c^{-n}R_0\le R<c^{-n-1}R_0,
	\end{align*}
	then, by monotonicity of $\overline{\mu}$ we have,
	\begin{align*}
		\overline{\mu}(R) \le \overline{\mu}(c^{-n}R_0) \le C_m R^{-\frac{m}{2}} \quad \mbox{for all $R\ge R_0$},
	\end{align*}
	where 
	\begin{align*}
		C_m=\overline{\mu}(R_0) R^{\frac m2}_0.
	\end{align*}
	Since $m$ is arbitrary, the above bound implies \eqref{eq:beta_moment}, and hence $\mu$ has moments of all orders.
	
	Let us now prove \eqref{eq:unif_prob_bound}. Since all eigenvalues of $B$ have strictly negative real parts,
	\begin{align*}
	\lim_{t\to \infty}\|e^{tB}\|=0.
	\end{align*}
	Choose $t\ge t_0$ sufficiently large so that
	\begin{align*}
	\|e^{tB}\|<1.
	\end{align*}
Let
	\begin{align*}
	\delta:=\inf_{|x|=1}|e^{tB}x|.
	\end{align*}
	Since $e^{tB}$ is invertible, $\delta>0$. Moreover, $\delta\le \|e^{tB}\|<1$. 
	Using the Mehler representation formula for the semigroup $(P_t)_{t\ge 0}$, which is a direct consequence of \eqref{eq:characteristic_fn}, we have
	\begin{align}\label{eq:Mehler_formula}
		P_tf(x) = \int_{\mathbb R^d} f(e^{tB}x+y)\,\mu_t(dy),
	\end{align}
	where $\mu_t$ is such that
	\begin{align*}
		\int_{\R^d} e^{\i\langle\xi, x\rangle}\mu_t(dx)=e^{\Psi_t(\xi)} \quad \mbox{for all $\xi\in\R^d$}.
	\end{align*}
	Let us choose $\tau>0$ large enough so that
	\begin{align*}
		a:=\mu_t(\{y: |y|\le \tau\})>0.
	\end{align*}
 We note that if $|x|>R$, $|y|\le \tau$, and $R\ge 2\tau/\delta$, then
	\begin{align*}
		|e^{tB}x+y|
		&\ge |e^{tB}x|-|y|\\
		&\ge \delta|x|-\tau\\
		&>\delta R-\tau\\
		&\ge \frac{\delta}{2}R.
	\end{align*}
	Consequently, by \eqref{eq:Mehler_formula} we get
	\begin{align}\label{eq:P_T_bd1}
		P_t\mathbbm{1}_{\{|\,\cdot\,|>\delta R/2\}}(x)\ge a, \quad \mbox{for any $|x|>R$}.
	\end{align}
	Therefore, \eqref{eq:unif_prob_bound} follows by choosing $c=\delta/2$. This completes the proof.
\end{proof}
		
		In the next lemma, we prove compactness of the embedding $\mathrm{W}^{1,p}(\mu)\hookrightarrow \L^p(\mu)$ when \eqref{eq:W} holds. This is crucial for proving Theorem~\ref{thm:compactness}.
		
		\begin{lemma}\label{lem:compactness}
		Suppose that $W=-\log \mu$ and $\lim_{|x|\to\infty} |\nabla W(x)|=\infty$. Then, for all $2\le p<\infty$, the embedding $\mathrm{W}^{1,p}(\mu)\hookrightarrow\L^p(\mu)$ is compact. 
		\end{lemma}
		\begin{proof}
			Since by Theorem~\ref{thm:regularity}, $\mu$ is strictly positive and smooth, $W\in C^\infty(\R^d)$. From \cite[Lemma~8.5.3]{Lorenzi2007}, it is enough to show that for all $x\in\bb{R}^d$,
			\begin{align}\label{eq:V_bound}
				\Delta W(\bm x)\le \delta |\nabla W(\bm x)|^2+M
			\end{align}
			for some $\delta\in (0,1)$ and $M>0$, independent of $\bm x$. Let $\mu_J$ be the probability measure defined by
			\begin{align*}
				\int_{\R^d}e^{\i\langle\xi, x\rangle}\mu_J(dx)= e^{\Phi_\infty(\xi)}
			\end{align*}
		for all $\xi\in\R^d$. Then, the invariant distribution $\mu$ is the convolution of $\mu_J$ and a Gaussian density, that is, for all $x\in\R^d$,
		\begin{align*}
			\mu(x)=\frac{1}{(2\pi)^{\frac d2}\sqrt{\det{\Sigma_\infty}}}\int_{\bb R^d} e^{-\frac{1}{2}|\Sigma^{-\frac{1}{2}}_\infty (x-y)|^2}\mu_J (dy).
		\end{align*}
			 Denoting the partial derivative with respect to $x_i$ by $\partial_i$, we note that for any $1\le i\le d$, $\partial_{i}W=-\partial_i\mu/\mu, \partial^2_{i}W= ((\partial_i\mu)^2-\mu \partial^2_{i}\mu)/\mu^2$ and as a result,
			\begin{equation}\label{eq:derivatives_mu}
				\begin{aligned}
					\partial_i\mu(x)&=-\frac{1}{(2\pi)^{\frac d2}\sqrt{\det{\Sigma_\infty}}}\int_{\bb R^d} \langle \Sigma^{-1}_\infty(x-y), e_i\rangle e^{-\frac{1}{2}|\Sigma^{-\frac{1}{2}}_\infty (x-y)|^2}\mu_J (dy), \\
					\partial^2_{i} \mu(x)&=-\langle \Sigma^{-1}_\infty e_i, e_i\rangle \mu(x) \\
					&+\frac{1}{(2\pi)^{\frac d2}\sqrt{\det{\Sigma_\infty}}}\int_{\bb R^d} \langle \Sigma^{-1}_\infty(x-y), e_i\rangle^2 e^{-\frac{1}{2}|\Sigma^{-\frac{1}{2}}_\infty (x-y)|^2}\mu_J (dy).
				\end{aligned}
			\end{equation}
			Using Jensen's inequality we obtain 
			\begin{align*}
				(\partial_i \mu(x))^2&\le \frac{1}{(2\pi)^{\frac d2}\sqrt{\det{\Sigma_\infty}}}\int_{\R^d}\langle \Sigma^{-1}_\infty(x-y), e_i\rangle^2  e^{-\frac{1}{2}|\Sigma^{-\frac{1}{2}}_\infty (x-y)|^2}\mu_J (dy) \\
				&\ \ \ \times \frac{1}{(2\pi)^{\frac d2}\sqrt{\det{\Sigma_\infty}}}\int_{\R^d}e^{-\frac{1}{2}|\Sigma^{-\frac{1}{2}}_\infty (x-y)|^2}\mu_J (dy) \\
				&=\mu(x)(\partial^2_i \mu(x)+\langle \Sigma^{-1}_\infty e_i, e_i\rangle \mu(x)).
			\end{align*}
			As a result, we get 
			\begin{align*}
				\partial^2_i W(x)=\frac{(\partial_i\mu(x))^2-\mu(x) \partial^2_{i}\mu(x)}{\mu(x)^2}\le \langle \Sigma^{-1}_\infty e_i, e_i\rangle.
			\end{align*}
			Therefore, for all $x\in\R^d$, we have 
			\begin{align*}
				\Delta W(x)\le \tr(\Sigma^{-1}_\infty),
			\end{align*}
			which implies \eqref{eq:V_bound} with $\delta=0$ and $M=\tr(\Sigma^{-1}_\infty)$. 
			\end{proof}
		\begin{proof}[Proof of Theorem~\ref{thm:compactness}] We consider the case when $p=2$. Since by Lemma~\ref{lem:compactness}, the embedding $\mathrm{W}^{1,2}(\mu)\hookrightarrow\L^2(\mu)$ is compact when \eqref{eq:W} holds, the estimate in \eqref{eq:L_p_estimate} with $k=1$ and $p=2$ implies that $P_t:\L^2(\mu)\longrightarrow\L^2(\mu)$ is compact. As both $P_t:\L^1(\mu)\longrightarrow\L^1(\mu)$ and $P_t:\L^\infty(\mu)\longrightarrow\L^\infty(\mu)$ are bounded, by interpolation it follows that $P_t:\L^p(\mu)\longrightarrow\L^p(\mu)$ is compact for all $1<p<\infty$. This completes the proof of the theorem.
		\end{proof}
		\section*{Acknowledgments}
		The author would like to thank Ben Goldys, Victoria Knopova, Alexey Kuznetsov, and Mateusz Kwaśnicki for their helpful comments on an earlier version of the manuscript.
%		\section*{Statement and declaration} 
%		\subsection*{Competing interests} The author has no competing interests to declare that are relevant to the content of this article.
%		\subsection*{Data availability statement} No data were used to support this study.

	\providecommand{\bysame}{\leavevmode\hbox to3em{\hrulefill}\thinspace}
	\providecommand{\MR}{\relax\ifhmode\unskip\space\fi MR }
	% \MRhref is called by the amsart/book/proc definition of \MR.
	\providecommand{\MRhref}[2]{%
		\href{http://www.ams.org/mathscinet-getitem?mr=#1}{#2}
	}
	\providecommand{\href}[2]{#2}

	\end{document}